\pdfoutput=1
\documentclass[a4paper,11pt,oneside]{amsart}
\usepackage[paperheight=279mm,paperwidth=18cm,textheight=26cm,textwidth=14cm,includehead]{geometry}
\usepackage{amsfonts,amssymb,amsmath,amsthm}
\usepackage[T1]{fontenc}
\usepackage[utf8]{inputenc}
\usepackage{dsfont}
\usepackage{enumitem}
\usepackage{mathtools}
\usepackage[unicode]{hyperref}
\hypersetup{
bookmarks=true,
colorlinks=true,
citecolor=[rgb]{0,0,0.5},
linkcolor=[rgb]{0,0,0.5},
urlcolor=[rgb]{0,0,0.75},
pdfpagemode=UseNone,
pdfstartview=FitH,
pdfdisplaydoctitle=true,
pdftitle={Jump inequalities for translation-invariant operators of Radon type on Zd},
pdfauthor={Mirek, Stein, Zorin-Kranich},
pdflang=en-US
}

\usepackage[style=alphabetic]{biblatex}
\numberwithin{equation}{section}
\newtheorem{theorem}[equation]{Theorem}
\newtheorem{corollary}[equation]{Corollary}
\newtheorem{lemma}[equation]{Lemma}
\newtheorem{proposition}[equation]{Proposition}
\theoremstyle{definition}
\newtheorem{remark}[equation]{Remark}
\newtheorem{definition}[equation]{Definition}

\newcommand{\R}{\mathbb{R}}
\newcommand{\N}{\mathbb{N}}
\newcommand{\Z}{\mathbb{Z}}

\newcommand{\T}{\mathbb{T}}
\newcommand{\Q}{\mathbb{Q}}
\newcommand{\dif}{\mathrm{d}}
\let\rho\varrho
\let\epsilon\varepsilon

\renewcommand{\C}{\mathbb{C}}
\newcommand{\Prime}{\mathbb{P}}

\newcommand{\boldA}{\mathbf{A}}
\def\PZdefchar#1{
  \expandafter\def\csname frak#1\endcsname{\mathfrak{#1}}
  \expandafter\def\csname bf#1\endcsname{\mathbf{#1}}
  \expandafter\def\csname scr#1\endcsname{\mathcal{#1}}
  \expandafter\def\csname cal#1\endcsname{\mathcal{#1}}}
\def\PZdefloop#1{\ifx#1\PZdefloop\else\PZdefchar#1\expandafter\PZdefloop\fi}
\PZdefloop abcdefghijklmnopqrstuvwxyzABCDEFGHIJKLMNOPQRSTUVWXYZ\PZdefloop
\newcommand{\dscrU}{\Delta\scrU}
\newcommand{\bbI}{\mathbb{I}}

\newcommand{\bbP}{\mathbb{P}}
\newcommand{\FT}{\calF} 

\newcommand{\stPol}[1]{(#1)^{\Gamma}}
\newcommand{\const}{\mathrm{const}}
\newcommand{\divides}{\vert}

\newcommand*{\DMO}[1]{\expandafter\DeclareMathOperator\csname #1\endcsname {#1}}
\DMO{supp}
\DMO{lcm}
\DMO{dist}
\DMO{tr}
\DMO{diam}

\DeclarePairedDelimiter\abs{\lvert}{\rvert}
\DeclarePairedDelimiter\norm{\lVert}{\rVert}
\providecommand\given{}
\newcommand\SetSymbol[1][]{%
\nonscript\:#1\vert
\allowbreak
\nonscript\:
\mathopen{}}
\DeclarePairedDelimiterX\Set[1]\{\}{\renewcommand\given{\SetSymbol[\delimsize]}#1}

\DeclarePairedDelimiter\meas{\lvert}{\rvert}
\DeclarePairedDelimiter\card{\lvert}{\rvert}
\DeclarePairedDelimiter\floor{\lfloor}{\rfloor}
\DeclarePairedDelimiter\ceil{\lceil}{\rceil}
\DeclarePairedDelimiterX\innerp[2]{\langle}{\rangle}{#1,#2}

\newcommand{\one}{\mathbf{1}}
\newcommand{\ind}[1]{\one_{#1}}

\def\clap#1{\hbox to 0pt{\hss#1\hss}}

\def\mathrlap{\mathpalette\mathrlapinternal}

\def\mathrlapinternal#1#2{%
\rlap{$\mathsurround=0pt#1{#2}$}}

\author{Mariusz Mirek}
\address[Mariusz Mirek]{
  Department of Mathematics,
  Rutgers University,
Piscataway, NJ 08854, USA \&
Instytut Matematyczny,
Uniwersytet Wroc{\l}awski,
Plac Grunwaldzki 2/4,
50-384 Wroc{\l}aw
Poland}
\email{mariusz.mirek@rutgers.edu}
\author{Elias M. Stein}
\address[Elias M. Stein]{
Department of Mathematics,
Princeton University,
Princeton,
NJ 08544-100 USA}

\author{Pavel Zorin-Kranich}
\address[Pavel Zorin-Kranich]{Mathematical Institute, University of Bonn, Endenicher Allee 60, 53115 Bonn, Germany}
\email{pzorin@math.uni-bonn.de}

\thanks{Mariusz Mirek was partially supported by the Schmidt Fellowship and the IAS Found for Math.\ and by the National Science Center, NCN grant DEC-2015/19/B/ST1/01149.
Elias M. Stein was partially supported by NSF grant DMS-1265524.
Pavel Zorin-Kranich was partially supported by the Hausdorff Center for Mathematics and DFG SFB-1060.}

\begin{document}
\title[Jump inequalities for operators of Radon type on $\mathbb Z^{d}$]{Jump inequalities for translation-invariant\\
operators of Radon type on  $\mathbb Z^{d}$}
\begin{abstract}
We prove strong jump inequalities for a large class of operators of
Radon type in the  discrete and ergodic theoretical settings. These
inequalities are the $r=2$ endpoints of the $r$-variational estimates
studied in \cite{MR3681393}.
\end{abstract}
\maketitle

\setcounter{tocdepth}{1}

\tableofcontents

\section{Introduction}
The purpose of this paper is to initiate the study of strong uniform $\lambda$-jump inequalities in the context of  discrete translation-invariant operators of Radon type and their applications to ergodic theory.
We extend the previously known results for $r$-variation $V^r$ (see \eqref{eq:def:Vr} for the definition), where $r>2$, to endpoint results
formulated in terms of the jump quasi-seminorm $J_2^p$ (see \eqref{eq:jump-space} for the definition) for these operators. These
are stated in Theorem~\ref{thm:discrete-jump}, Corollary~\ref{cor:3}, Theorem~\ref{thm:11}, and Theorem~\ref{thm:4}.
To describe further the operators we consider and the properties we prove for them, we need to fix the notation and terminology used.

\subsection{Basic setup}
Let $\N_0 = \Set{0,1,2,\dotsc}$ denote the set of non-negative integers.
Throughout the article, we fix a finite set of multi-indices $\Gamma \subset \N_{0}^{k} \setminus \Set{0}$ with lexicographical order.
We denote by $\R^{\Gamma}$ the space of tuples of real numbers labeled by multi-indices $\gamma=(\gamma_1, \ldots, \gamma_k)\in\Gamma$, so that $\R^{\Gamma} \cong \R^{\abs{\Gamma}}$, and similarly for $\Z^{\Gamma} \cong \Z^{\abs{\Gamma}}$.
The canonical polynomial mapping is given by
\[
\R^k \ni x = (x_{1},\dotsc,x_{k}) \mapsto \stPol{x}:= (x_1^{\gamma_1}\dotsm x_k^{\gamma_k}:{\gamma\in\Gamma}) \in \R^\Gamma,
\]
and it restricts to a mapping from $\Z^{k}$ to $\Z^{\Gamma}$.

Let $\Omega$ be a non-empty convex body (not necessarily symmetric) in $\R^k$, which simply means that $\Omega$ is a bounded convex open subset of $\R^k$.
For $t>0$, we define its dilates
\[
\Omega_t := \Set{x\in\R^{k} \given t^{-1}x\in\Omega}.
\]
We will additionally assume that $B(0,c_{\Omega}) \subseteq \Omega \subseteq B(0,1) \subset \R^{k}$ for some $c_{\Omega}\in(0, 1)$, where $B(x, t)$ denotes an open Euclidean ball in $\R^k$ centered at $x\in\R^k$ with radius $t>0$.
This ensures that $\Omega_t\cap\Z^k=\{0\}$ for all $t\in(0, 1)$. A typical choice of $\Omega_{t}$ is a ball of radius $t$ for some norm on $\R^{k}$.

For finitely supported functions $f:\Z^{\Gamma}\to \C$, for every $x\in\Z^\Gamma$ and $t>0$, we define the discrete averaging Radon operator by setting
\begin{align}
\label{eq:247}
M_tf(x):= \frac{1}{\card{\Omega_{t}\cap\Z^{k}}} \sum_{y \in\Omega_{t}\cap\Z^{k}} f(x-{\stPol{y}}).
\end{align}
We will also consider the discrete truncated singular Radon operator
\begin{align}
\label{eq:248}
H_tf(x):= \sum_{y \in\Omega_{t}\cap\Z^{k}\setminus\{0\}} f(x-{\stPol{y}})K(y),
\end{align}
where  $K : \R^{k}\setminus\Set{0} \to \C$ is a  Calder\'on--Zygmund kernel satisfying the following conditions.
\begin{enumerate}
\item The size condition. For every $x\in\R^k\setminus\Set{0}$, we have
\begin{equation}
\label{eq:size-unif}
\abs{K(x)} \leq \abs{x}^{-k}.
\end{equation}
\item The cancellation condition
\begin{equation}
\label{eq:cancel}
\int_{\Omega_{R}\setminus \Omega_{r}} K(y) \dif y = 0,
\quad \text{ for any }\quad
0<r<R<\infty.
\end{equation}
\item
  The H\"older continuity condition.
  For some $\sigma\in(0, 1]$ and every $x, y\in\R^k\setminus\Set{0}$ with $\abs{y}\leq \abs{x}/2$, we have
\begin{equation}
\label{eq:K-modulus-cont}
\abs{K(x)-K(x+y)}
\leq
\abs{y}^{\sigma} \abs{x}^{-k-\sigma}.
\end{equation}
\end{enumerate}

We will follow the notation  used in \cite{arxiv:1808.04592}.
For any $\lambda>0$ and $\bbI \subset \R$, the \emph{$\lambda$-jump counting function} of a function $f : \bbI \to \C$ is defined by
\begin{align}
  \label{eq:def:jump}
  \begin{split}
N_{\lambda}(f)& :=
N_{\lambda}(f(t) : t\in\bbI)\\
&:=\sup \Set{J\in\N \given \exists_{\substack{t_{0}<\dotsb<t_{J}\\ t_{j}\in\bbI}}  : \min_{0<j\leq J} \abs{f(t_{j})-f(t_{j-1})} \geq \lambda},
  \end{split}
\end{align}
and the $r$-variation seminorm by
\begin{align}
\label{eq:def:Vr}
  \begin{split}
V^{r}(f)
:=&V^{r}(f(t) : t\in\bbI)\\
:=&
\begin{cases}
\sup_{J\in\N} \sup_{\substack{t_{0}<\dotsb<t_{J}\\ t_{j}\in\bbI}}\Big(\sum_{j=1}^{J}  \abs{f(t_{j})-f(t_{j-1})}^{r} \Big)^{1/r},  
&
0< r <\infty,\\
\sup_{\substack{t_{0}<t_{1}\\ t_{j}\in\bbI}} \abs{f(t_{1})-f(t_{0})},  
& r = \infty,
\end{cases}
\end{split}
\end{align}
where the former supremum is taken over all finite increasing sequences in $\bbI$.

Throughout the article, $(X,\calB(X),\mathfrak m)$ denotes a $\sigma$-finite measure space.
For a function $f : X \times \bbI \to \C$ and an exponent $1<p<\infty$, the jump quasi-seminorm on $L^p(X)$ is defined by
\begin{align}
  \label{eq:jump-space}
  \begin{split}
  J^{p}_{2}(f):=J^{p}_{2}(f : X \times \bbI \to \C)&:=J^{p}_{2}((f(\cdot, t))_{t\in\bbI}):=J^{p}_{2}((f(\cdot, t))_{t\in\bbI} : X  \to \C) \\
  &:=\sup_{\lambda>0} \norm[\big]{ \lambda N_{\lambda}(f(\cdot, t): t\in \bbI)^{1/2} }_{L^{p}(X)}.
  \end{split}
\end{align}

In view of our assumptions on $\Omega$, we see that $M_tf\equiv f$ and $H_tf\equiv 0$ for $t\in(0, 1)$, therefore only $t\ge1$ will be relevant for us.
The function $(0, \infty)\ni t\mapsto \Omega_t\cap\Z^k$ takes only countably many values by the monotonicity of the sets $\Omega_t$.
Hence, \eqref{eq:jump-space} will be always taken over the parameters $t$ restricted to countable sets $\bbI\subseteq (0, \infty)$.  

\subsection{Statement of the main results}
In \cite{arXiv:1512.07518,MR3681393}, strong maximal and $r$-variational estimates on $\ell^p(\Z^\Gamma)$ were obtained for the operators $M_t$ and $H_t$ with the sharp range of exponents $p\in(1, \infty)$ and $r\in(2, \infty)$.
The main aim of this paper is to strengthen these results and provide strong uniform $\ell^p(\Z^\Gamma)$ bounds for $\lambda$-jumps that are a substitute for the $r$-variational estimates at the $r=2$ endpoint.
In the continuous case, such endpoint estimates were obtained by Jones, Seeger, and Wright \cite{MR2434308}, see also \cite{arxiv:1808.09048} for an alternative approach to some of their results.

Our main discrete result is the following theorem.
\begin{theorem}
\label{thm:discrete-jump}
Let $T_t$ be either $M_t$ or $H_t$.
Then, for every $p\in(1, \infty)$, there is $0<C_{p}<\infty$ such that, for every $f\in \ell^p(\Z^{\Gamma})$, we have
\begin{equation}
\label{eq:93}
J^{p}_{2}((T_tf)_{t\ge1} : \Z^{\Gamma}  \to \C)
\le C_{p}
\norm{f}_{\ell^p(\Z^{\Gamma})}.
\end{equation}
\end{theorem}
Using \cite[Lemma 2.12]{arxiv:1808.04592}, for $r>2$, we obtain that
\begin{equation}
\label{eq:V-vs-jump}
\norm{V^{r}(T_tf: t\ge1)}_{\ell^{p, \infty}(\Z^{\Gamma})}
\lesssim_{p, r}
J^{p}_{2}((T_tf)_{t\ge1} : \Z^{\Gamma}  \to \C).
\end{equation}
One now sees that \eqref{eq:93}, combined with \eqref{eq:V-vs-jump} and real interpolation, implies that, for every $p\in(1, \infty)$ and $r\in(2, \infty]$, there is $0<C_{p, r}<\infty$ such that
\begin{align}
\label{eq:89}
\norm{\sup_{t\ge1}\abs{T_tf}}_{\ell^p(\Z^{\Gamma})}&\le C_{p, \infty}\norm{f}_{\ell^p(\Z^{\Gamma})},\\
\label{eq:92}
\norm{V^{r}(T_tf: t\ge1)}_{\ell^p(\Z^{\Gamma})}
&\le C_{p, r}
\norm{f}_{\ell^p(\Z^{\Gamma})},
\end{align}
for every $f\in \ell^p(\Z^{\Gamma})$.
Thus Theorem~\ref{thm:discrete-jump} refines the main results of \cite{MR3681393}, and 
in this sense the jump inequality \eqref{eq:93} is an endpoint estimate for \eqref{eq:92} when $r=2$.

The proof of Theorem~\ref{thm:discrete-jump} is contained in Section~\ref{sec:disrad}.
It is presented as part of an abstract theory of discrete convolution operators of Radon type established in Section \ref{sec:discrete}.
The main result of Section~\ref{sec:discrete} is Theorem~\ref{thm:4}, which under some conditions (see conditions \ref{MF:A}, \ref{MF:B}, \ref{MF:C}, and \ref{MF:D} in Section~\ref{sec:discrete}) provides strong uniform $\lambda$-jump inequalities for the discrete Radon transforms.
The advantage of this approach is that we handle $M_t$ and $H_t$ (or even more general multipliers satisfying \ref{MF:A}, \ref{MF:B}, \ref{MF:C}, and \ref{MF:D}) simultaneously.
This is very much in spirit of the theory of the continuous Radon transforms.
The results from Section~\ref{sec:discrete} may thus be thought of as discrete analogues of the results from \cite{MR2434308} and \cite{arxiv:1808.09048}.

Theorem~\ref{thm:discrete-jump} immediately applies to Radon type operators modeled on arbitrary polynomial mappings.
Namely, let $P=(P_1,\dotsc, P_m):\Z^k\to\Z^m$ be a polynomial mapping, where each component $P_j:\Z^k\to\Z$ is a polynomial of $k$ variables with integer coefficients.
As in \eqref{eq:247} and \eqref{eq:248}, for finitely supported functions $f:\Z^{m}\to \C$, for every $x\in\Z^m$ and $t>0$, we define  the discrete averaging and singular Radon operators  along $P$ by setting
\begin{align}
\label{eq:256}
\begin{split}
M_t^Pf(x)&:= \frac{1}{\card{\Omega_{t}\cap\Z^{k}}} \sum_{y \in
\Omega_{t}\cap\Z^{k}} f(x-P(y)),\\
H_t^Pf(x)&:= \sum_{y \in
\Omega_{t}\cap\Z^{k}\setminus\{0\}} f(x-P(y))K(y),
\end{split}
\end{align}
where $K$ is a Calder\'on--Zygmund kernel satisfying \eqref{eq:size-unif}, \eqref{eq:cancel}, and \eqref{eq:K-modulus-cont}.

Taking $\Gamma=\Set{(\gamma_1, \ldots,\gamma_k)\in\N_0^k \given 0<\abs{\gamma_1}+\ldots+\abs{\gamma_k}\le {\rm deg}P}$ in Theorem~\ref{thm:discrete-jump} and invoking the lifting procedure for the Radon transforms described in \cite[p.~515]{MR1232192}, (see also \cite{arXiv:1512.07518}), we obtain the following results.
\begin{corollary}
\label{cor:3}
Let $T_t^P$ be either $M_t^P$ or $H_t^P$.
Then, for every $p\in(1, \infty)$, there is $0<C_{p}<\infty$ such that, for every $f\in \ell^p(\Z^{m})$, we have
\begin{equation}
\label{eq:257}
J^{p}_{2}((T_t^Pf)_{t\ge1} : \Z^{m}  \to \C)
\le C_{p}
\norm{f}_{\ell^p(\Z^{m})}.
\end{equation}
In particular, \eqref{eq:257} implies that, for every $p\in(1, \infty)$ and $r\in(2, \infty]$, there is $0<C_{p, r}<\infty$ such that
\begin{align}
  \label{eq:258}
\norm{\sup_{t\ge1}\abs{T_t^Pf}}_{\ell^p(\Z^{m})}&\le C_{p, \infty}\norm{f}_{\ell^p(\Z^{m})},\\
\label{eq:259}
\norm{V^{r}(T_t^Pf: t\ge1)}_{\ell^p(\Z^{m})}&
\le C_{p, r}
\norm{f}_{\ell^p(\Z^{m})},
\end{align}
for every $f\in \ell^p(\Z^{m})$.
Moreover, the constants $C_{p}$ and $C_{p,r}$ depend on $k$ and $\deg P$, but not otherwise on the coefficients of $P$.
\end{corollary}

\subsection{Historical background and some further results}
Maximal inequalities for the discrete averaging Radon transforms $M_t^{P}$ with $k=m=1$ and arbitrary polynomials $P$ were obtained by Bourgain in a foundational series of papers \cite{MR937581,MR937582} culminating in \cite{MR1019960}.
Variational refinements of Bourgain's results with $r\in (2,\infty)$ were obtained in \cite{arXiv:1402.1803} for $p=2$ and in \cite{arxiv:1403.4085} for $p$ in a small neighborhood of $2$.
Variational estimates for multidimensional variants of Bourgain's averaging operator were investigated in \cite{MR3595493}.

Systematic studies of discrete singular Radon transforms
\begin{align}
\label{eq:120}
Hf(x):= \sum_{y \in
\Z^{k}\setminus\{0\}} f(x-{\stPol{y}})K(y),
\end{align}
were initiated in \cite{MR1056560}.
However, at that time their $\ell^p(\Z^{\Gamma})$ boundedness was only obtained for $p$ in a certain neighborhood of $2$.
Estimates in the full range of $p\in(1, \infty)$ for \eqref{eq:120} were first obtained by Ionescu and Wainger \cite{MR2188130}, see also \cite{MR3738256} for a different approach.
Their ideas were taken up in \cite{arXiv:1512.07518,MR3681393} in order to prove sharp variational estimates for $M_t$ and $H_t$.
In this article we further developed their ideas.
In particular, Theorem~\ref{thm:discrete-jump} and Corollary~\ref{cor:3} contain the main results of \cite{MR3681393}.

Finally, let us mention that the results in Section \ref{sec:discrete} can be applied to operators modeled on the set of prime numbers $\bbP$.
Namely, we fix non-negative integers $k', k''$ such that $k'+k''=k$, and for $t>0$ we define $\pi_{\Omega}(t):=\sum_{y'\in\N^{k'}}\sum_{y''\in\bbP^{k''}} \ind{\Omega_t}(y', y'')$.
As above, for finitely supported functions $f:\Z^{m}\to \C$ and for every $x\in\Z^m$, we introduce discrete Radon operators over the primes given by
\begin{align}
\label{eq:1}
\begin{split}
\tilde{M}_t^Pf(x)&:= \frac{1}{\pi_{\Omega}(t)}\sum_{y'\in\N^{k'}}\sum_{y''\in\bbP^{k''}}  f(x-P(y', y''))\ind{\Omega_t}(y', y''),\\
\tilde{H}_t^Pf(x)&:= \sum_{y'\in\N^{k'}}\sum_{y''\in(\pm\bbP)^{k''}}  f(x-P(y', y''))\ind{\Omega_t}(y', y'')K(y', y'')\bigg(\prod_{j=1}^{k''}\log|y_j''|\bigg),
\end{split}
\end{align}
where $K$ is a Calder\'on--Zygmund kernel satisfying \eqref{eq:size-unif}, \eqref{eq:cancel}, and \eqref{eq:K-modulus-cont}. The logarithmic weight in definition of the operator  $\tilde{H}_t^P$ corresponds to the density of $\bbP$.
Theorem~\ref{thm:4} allows us to prove that if $T_t^P$ is either $\tilde{M}_t^P$ or $\tilde{H}_t^P$, then \eqref{eq:257} holds for all $p\in(1, \infty)$, and consequently we recover \eqref{eq:259} and \eqref{eq:258}, which, under more restrictive conditions on $K$, were studied in \cite{arxiv:1803.05406}.
See Remark~\ref{rem:1} for more details.

\subsection{Applications in pointwise ergodic theory}
Another motivation for considering $\lambda$-jumps is their applicability to pointwise convergence problems in ergodic theory.
The classical strategy for handling pointwise convergence of $T_tf(x)$ (as $t\to0$ or $t\to\infty$) requires $L^p(X)$ boundedness for the corresponding maximal function $\sup_{t>0}\abs{T_tf(x)}$, reducing the matters to proving pointwise convergence of $T_tf(x)$ for a dense class of $L^p(X)$ functions.
Although in questions in harmonic analysis there are many natural dense subspaces which could be used to establish pointwise convergence, in the discrete or in ergodic theoretical questions this may not be the case.

Bourgain's approach to pointwise ergodic theorems consists in quantifying the convergence property.
In Bourgain's articles starting with \cite[Lemma 7.3]{MR937581}, so-called oscillation inequalities were used to this purpose.
More refined estimates involving $r$-variations were first obtained by Krause \cite{arXiv:1402.1803} and extended to the full range of exponents in \cite{MR3681393}.
The $\lambda$-jump inequalities quantify the pointwise convergence even more precisely.

The above-mentioned operators have an ergodic theoretical interpretation.
Let $(X, \calB(X), \frakm)$ be a $\sigma$-finite measure space with a family of invertible commuting and measure preserving transformations $S_\gamma:X\to X$, $\gamma\in\Gamma$.
For any function $f:X\to \C$, for every $x\in X$ and $t>0$, let
\begin{align}
\label{eq:254}
\begin{split}
\bfM_t f(x)
&:= \frac{1}{\card{\Omega_{t}\cap\Z^{k}}} \sum_{y \in \Omega_{t}\cap\Z^{k}}
f\big( \big( \prod_{\gamma\in\Gamma} S_\gamma^{(y^{\gamma})} \big) x \big),\\
\bfH_t f(x)
&:= \sum_{y \in \Omega_{t}\cap\Z^{k}\setminus\{0\}}
f\big( \big( \prod_{\gamma\in\Gamma} S_\gamma^{(y^{\gamma})} \big) x \big)K(y),
\end{split}
\end{align}
where $K$ is the Calder\'on--Zygmund kernel as above.
The setting of Theorem~\ref{thm:discrete-jump} can be recovered with $X=\Z^{\Gamma}$, $\calB(X)=\bfP(\Z^\Gamma)$ the $\sigma$-algebra of all subsets of $\Z^{\Gamma}$, $\frakm=\abs{\; \cdot\; }$ the counting measure on $\Z^{\Gamma}$, and $S_{\gamma}:\Z^{\Gamma} \to \Z^{\Gamma}$ the shift operator acting on the $\gamma$-th coordinate, i.e.\ $S_{\gamma}(x) = x-e_{\gamma}$, where $e_{\gamma}$ is the $\gamma$-th standard basis vector in $\Z^{\Gamma}$.
The setting of Corollary~\ref{cor:3} can be similarly recovered with $X=\Z^{m}$ and $S_{\gamma} = \prod_{j=1}^{m} S_{j}^{a_{j,\gamma}}$, where $S_{j}(x)=x-e_{j}$ is the shift operator acting on the $j$-th coordinate, and $a_{j,\gamma}$ are the coefficients of the polynomial $P_{j}(y) = \sum_{\gamma\in\Gamma} a_{j,\gamma} y^{\Gamma}$.

We now state our main ergodic theorem.

\begin{theorem}
\label{thm:11}
Let $\bfT_t$ be either $\bfM_t$ or $\bfH_t$.
Then, for every $p\in(1, \infty)$, there is $0<C_{p}<\infty$ such that, for every $f\in L^p(X)$, we have
\begin{equation}
\label{eq:260}
J^{p}_{2}((\bfT_t f)_{t\ge1} : X  \to \C)
\le C_{p}
\norm{f}_{L^p(X)}.
\end{equation}
In particular, for every $f\in L^p(X)$, there exists a function
$f^*\in L^p(X)$ such that
\[
\lim_{t\to\infty}\bfT_t f(x)=f^*(x)
\]
for $\frakm$-almost every $x\in X$, and in $L^p(X)$.
\end{theorem}

Theorem~\ref{thm:11} easily follows from Theorem~\ref{thm:discrete-jump} by invoking Calder\'on's transference principle \cite{MR0227354}.
In the discrete singular integral case $\bfT_t=\bfH_t$, Theorem~\ref{thm:11} extends a well-known theorem of Cotlar \cite{MR0084632}, who established pointwise convergence for the truncated ergodic Hilbert transform.

\subsection{Overview of the paper and methods}
We list first some of the technical innovations used in this paper.
\begin{enumerate}
\item Certain properties of $J_2^p$ proved in~\cite{arxiv:1808.04592}:
that the quantity is equivalent to a norm corresponding to a  certain real interpolation space (in the sense of Peetre's $K$-method),
see~\cite[Lemma 2.7 and Corollary 2.11]{arxiv:1808.04592}; also that
sampling methods of~\cite{MR1888798} arising in the passage from the
continuous case to the discrete case work for $J_2^p$, (although this
space is not of the form $L^p(B)$, with $B$ a Banach space). See the
sampling principle for jumps
\cite[Theorem 1.7]{arxiv:1808.04592} and a more general sampling
principle for real interpolation spaces \cite[Proposition 4.7]{arxiv:1808.04592}. 
\item The conclusion that the basic Ionescu--Wainger theorem holds for
multipliers that are operator-valued (in the Hilbert space setting),
which does not readily follow from the previously known scalar-valued
case. See Theorem~\ref{thm:IW-mult}.
\item The more efficient partition lemma (see
Lemma~\ref{lem:partition-O}) that gives $O(\log N)$ partitions, where
$O((\log N)^{D-1})$ partitions were needed before.
\item The technique of splitting long and short variations along
subexponential sequences (arising in~\eqref{eq:21}), which goes
back to \cite{arxiv:1403.4085}.  
\end{enumerate}

The proof of Theorem~\ref{thm:discrete-jump} will be given in Section~\ref{sec:disrad} as a consequence of the more general Theorem~\ref{thm:4}.
One of the novelties of the paper is that the present proof of Theorem~\ref{thm:4} is, to a significant extent, based on the ideas which work in the theory of the continuous Radon transforms.
However, many aspects are still more involved, mainly due to the arithmetic nature of the operators $M_t$ and $H_t$.
We can try to explain these complications as follows.
Multipliers corresponding to the discrete
Radon transforms are periodic functions, which turn out to be
concentrated around rational fractions with small denominators (on
``major arcs'' in the language of number theory), and even though the
sampling principles from  \cite{MR1888798} and \cite{arxiv:1808.04592}
provide optimal bounds for periodic Fourier multipliers, they cannot
be applied directly. Although it is possible to organize the parts of
the multipliers near rational points on the torus with denominators
$\leq N$ and express these as a combination of averaging operators and
periodic multipliers, it is necessary to use denominators up to
$\lcm(1,\dotsc,N)$, and it thus becomes difficult to retain control of
the $\ell^{p}$ operator norms for $p$ far away from $2$. A different
combinatorial organization of rational fractions with small
denominators was introduced by Ionescu and Wainger in \cite[Theorem
1.5]{MR2188130}, (see also \cite{arXiv:1512.07518} and
\cite{MR3738256}). It makes possible to exploit a strong
orthogonality in $\ell^p$ when $p$ is an even
integer, and thus their ideas allow one to control (up to a logarithmic loss) multipliers concentrated at rational frequencies with denominators $1,\dotsc,N$ at the expense of introducing an auxiliary family $\scrU_N$ of rational frequencies with denominators of order $o(e^{N^{\rho}})$, where $\rho>0$ can be arbitrarily small.
To be more specific, Ionescu and Wainger proved in \cite[Theorem 1.5]{MR2188130} that, if $\Theta$ is a multiplier supported in the unit cube in $\R^d$ and defines a bounded operator on $L^p(\R^d)$ for $p\in(1, \infty)$, then its periodic extension
\begin{equation}
\label{eq:255}
\Delta_N(\xi):
=\sum_{a/q \in\scrU_N}
\Theta(\epsilon_N^{-1}(\xi - a/q)),
\end{equation}
satisfies, with
the same range of $p\in(1, \infty)$ for some constant $C_{p, \rho, d}>0$ independent of $N\in\N$ and all $f\in\ell^p(\Z^{d})$, the inequality
\begin{align}
\label{eq:261}
\norm[\big]{ \FT^{-1}\big(\Delta_{N} \hat{f}\big)}_{\ell^p(\Z^{d})}
\le C_{p,\rho, d}
(\log N)
\norm{f}_{\ell^p(\Z^{d})},
\end{align}
where $0<\epsilon_N \le e^{-N^{\rho}}$.
We refer to Section~\ref{sec:iw} for more
detailed definitions.
The inequality \eqref{eq:261} was the main tool in the proof of $\ell^p$ boundedness of \eqref{eq:120}, and it was very useful in \cite{arXiv:1512.07518,MR3681393,MR3738256}.

We extend inequality \eqref{eq:261} by proving its vector-valued variant, see Theorem~\ref{thm:IW-mult}, which allows the underlying multipliers in \eqref{eq:261} to be operator-valued and act on a separable Hilbert space valued functions.
An immediate consequence of Theorem~\ref{thm:IW-mult} is the discrete Littlewood--Paley theory established in \cite{MR3738256}.

Theorem~\ref{thm:IW-mult}, the sampling principle for interpolation spaces from \cite{arxiv:1808.04592}, and the concepts from the circle method of Hardy and Littlewood, are the core of Section~\ref{sec:discrete} and another important novelty of the paper.
We first appeal to the observation introduced in \cite{arxiv:1403.4085}, which trivializes the estimates for short variations by considering jumps along sequences that grow subexponentially.
This argument did not work in \cite{MR3681393}, since the vector-valued version of the Ionescu--Wainger theorem was not available at that time.
Secondly, we use the circle method to construct certain multi-frequency multipliers to approximate the discrete Radon transforms.
Then the analysis of the approximating multiplier is split into two parts, where one distinguishes between small and large time parameters $t$ depending on the size of the underlying rational frequencies.
To control the small scales, we use a variant of  Rademacher--Menshov theorem independently introduced in \cite{MR2885959} and \cite{MR3595493}, and then appeal to the vector-valued version of inequality \eqref{eq:261}.
Here, Theorem~\ref{thm:IW-mult} is indispensable.
Large scales are treated by the sampling principle for jumps alluded to
earlier.

Finally, Appendix~\ref{sec:weyl} is devoted to a tool needed for
the above: control of multidimensional  Weyl sums with an appropriate
logarithmic loss. This appeared first in~\cite{arXiv:1512.07518} and
was in turn based on the corresponding approach in~\cite{MR1719802}
that allowed a power loss. Here we give a more precise formulation and
proof of the result, which implies the corresponding results in~\cite{MR1719802} and~\cite{arXiv:1512.07518}.

\subsection{Notation}
\begin{enumerate}
\item We write $A \lesssim_{D_{1},D_{2},\dotsc} B$ if there is a constant $C=C(D_{1},D_{2},\dotsc)>0$ such that $A\leq CB$.
We will omit some of the parameters $D_{i}$ when they are clear from the context, for instance we always allow implicit constants to depend on $\Gamma\subseteq\N^k_0$.
We write $A \simeq B$ if $A \lesssim B$ and $A\gtrsim B$ hold simultaneously.

\item
Let $\N=\Set{1,2,\dotsc}$ be the set of positive integers and $\N_0 = \N\cup\{0\}$.
For $N \in \N$, we set
\[
\N_N = \Set{1, 2, \dotsc, N}.
\]

\item
For a vector $x \in \R^d$, we will use the following norms
\[
\abs{x}_\infty = \max\Set{\abs{x_j} : 1 \leq j \leq d}, \quad \text{and} \quad
\abs{x}=\abs{x}_2 = \Big(\sum_{j = 1}^d \abs{x_j}^2\Big)^{1/2}.
\]
\item The standard scalar product on $\R^d$ will be denoted by
\[
x\cdot\xi=\langle x, \xi\rangle=\sum_{j=1}^dx_j\xi_j
\]
for any $x, \xi\in\R^d$.
\item
If $\gamma$ is a multi-index from $\N_0^k$, then
$\abs{\gamma} := \gamma_1 + \dotsb + \gamma_k$.
It should be clear from the context whether the argument of $\abs{\cdot}$ is a multi-index $\gamma\in\N_0^k$ or a vector $x\in\R^d$, so that $\abs{\cdot}$ can be interpreted accordingly.

\item
Let $\FT$ denote the Fourier transform on $\R^d$, defined for any function $f \in L^1(\R^d)$ by
\[
\FT f(\xi) = \int_{\R^d} f(x) e(\xi \cdot x) \dif x,
\quad
\text{ for } \quad\xi\in\R^{d},
\]
where
\[
e(y) := \exp(2\pi i y).
\]
The discrete Fourier transform of a function $f \in \ell^1(\Z^d)$ is denoted by
\[
\hat{f}(\xi) = \sum_{x \in \Z^d} f(x) e(\xi \cdot x),
\quad \text{ for } \quad
\xi\in\T^{d} = (\R/\Z)^{d}.
\]
For simplicity of notation, we denote by $\FT^{-1}$ the inverse Fourier transform on $\R^d$ or the inverse Fourier transform on the torus $\T^d\equiv[-1/2, 1/2)^d$ (Fourier coefficients), depending on the context.
\end{enumerate}

\subsection{Acknowledgment}
We thank the careful anonymous referee for pointing out several errors in a previous revision of this work.


\section{Hilbert space valued Ionescu--Wainger multiplier theorem}
\label{sec:iw}
The advantage of a Hilbert space valued version of the Ionescu--Wainger multiplier theorem, see Theorem~\ref{thm:IW-mult} below, is that we can directly transfer square function estimates from the continuous to the discrete setting without randomization.
In particular, Theorem~\ref{thm:IW-mult} implies \cite[Theorem 5.1]{MR3681393} (see also \cite[Theorem 5]{MR3738256});
but we apply it to different square functions.
We present the full proof with streamlined notation and arguments.

\begin{theorem}
\label{thm:IW-mult}
For every $\rho>0$, there exists a family $(P_{N})_{N\in\N}$ of subsets $P_{N} \subset \N$, satisfying
\begin{gather}
\N_N\subseteq P_N\subseteq\N_{\max\{N, e^{N^{\rho}}\}}, \label{eq:IW-denom-bound}\\
N_1\le N_2 \implies P_{N_1}\subseteq P_{N_2}, \label{eq:IW-denom-nested}\\
\lcm P_{N} \leq 3^{N} \label{eq:IW-denom-lcm}.
\end{gather}

Furthermore, for every $p \in (1,\infty)$ such that $p \in 2\N$ or $p' \in 2\N$, there exists a constant $0<C_{p, \rho, d}<\infty$ such that, for every $N\in\N$, the following holds.

Let $0<\epsilon_N \le e^{-N^{\rho}}$, and let $\Theta : \R^{d} \to L(H_0,H_1)$ be a measurable function supported on $\epsilon_{N}\bfQ$, where $\bfQ=[-1/2, 1/2]^d$ is a unit cube, with values in the space $L(H_{0},H_{1})$ of bounded linear operators between separable Hilbert spaces $H_{0}$ and $H_{1}$.
Let $0 \leq \boldA_{p} \leq \infty$ denote the smallest constant such that, for every function $f\in L^2(\R^d;H_0)\cap L^{p}(\R^d;H_0)$, we have
\begin{align}
\label{eq:IW-Lp-hypothesis}
\norm[\big]{\FT^{-1}\big(\Theta\FT f\big)}_{L^{p}(\R^{d};H_1)}
\leq
\boldA_{p} \norm{f}_{L^{p}(\R^{d};H_0)}.
\end{align}
Then the multiplier
\begin{equation}
\label{eq:IW-mult}
\Delta_N(\xi)
:=\sum_{b \in\scrU_N}
\Theta(\xi - b),
\end{equation}
where $\scrU_{N} \subset \Q^{d} \cap [0,1)^{d}$ is the set of all reduced fractions with denominators in $P_{N}$, satisfies
\begin{align}
\label{eq:111}
\norm[\big]{ \FT^{-1}\big(\Delta_{N} \hat{f}\big)}_{\ell^p(\Z^{d};H_1)}
\le C_{p,\rho,d}
(\log N) \boldA_{p}
\norm{f}_{\ell^p(\Z^{d};H_0)}
\end{align}
for every $f\in \ell^p(\Z^d;H_0)$.
\end{theorem}

The hypothesis \eqref{eq:IW-Lp-hypothesis}, unlike the support hypothesis, is scale-invariant, in the sense that the constant $\bfA_{p}$ does not change when $\Theta$ is replaced by $\Theta(A\cdot)$ for any invertible linear transformation $A$.

If we assume that \eqref{eq:IW-Lp-hypothesis} holds for every $p \in (1,\infty)$ with some finite $\bfA_{p}$, then, by complex interpolation, also the multiplier \eqref{eq:IW-mult} is bounded $\ell^{p}(H_{0}) \to \ell^{p}(H_{1})$ for every $p\in (1,\infty)$.
It would be interesting to know whether the operator norm of \eqref{eq:IW-mult} is bounded by a constant times $\bfA_{p}$ for all $p\in (1,\infty)$.

Replacing $\Theta$ by the operator matrix
\[
\begin{pmatrix}
0 & 0\\ \Theta & 0
\end{pmatrix}
\]
on $H=H_0\oplus H_1$, we may assume that $H_0=H_1=H$.
Throughout this section we assume that $f:\Z^d\rightarrow H$ has finite support.

\subsection{Construction of the sets $P_{N}$}
The strength of Theorem~\ref{thm:IW-mult} increases as $\rho$ decreases, so it suffices to consider $\rho < 2$. 
Fix $\rho \in (0,2)$ and let $D=D_{\rho}:=\ceil{2/\rho}$. Let $\Prime$ denote the set of all prime numbers.
For every $N\in\N$, define $\Prime_N:=\Prime\cap (N^{\rho/2}, N]$ and
\[
Q_0:=\lcm \Set{ n\in\N_{N} \given n \text{ is not divisible by any element of } \Prime_{N}}.
\]

We will use the sumset and product set notation
\begin{align*}
S_{1}+\dotsb+S_{k} &= \Set{ s_{1}+\dotsb+s_{k} \given \ s_{1}\in S_{1},\dotsc,s_{k}\in S_{k}},\\
S_{1}\cdots S_{k} &= \Set{ s_{1}\cdots s_{k} \given \ s_{1}\in S_{1},\dotsc,s_{k}\in S_{k}},
\end{align*}
with the convention that the sumset equals $\Set{0}$ and the product set equals $\Set{1}$ in the case $k=0$.
For small $N$ (depending only on $\rho$, more precisely those $N$ for which \eqref{eq:116} does not hold yet), we set $P_{N} := \N_{N}$, while for larger $N$ we set
\[
P_N := \Set{q\in\N \given  q \text{ divides } Q_0} \cdot \Set{ n\in\N_{N} \given \text{all prime factors of } n \text{ are in } \Prime_{N}}.
\]
It is easy to verify \eqref{eq:IW-denom-nested} for these sets.

In order to verify \eqref{eq:IW-denom-bound} and \eqref{eq:IW-denom-lcm}, we recall an estimate for the least common multiple of the first $N$ numbers based on the prime number theorem.
\begin{lemma}[{\cite{MR0313179}}]
\label{lem:lcm}
For every $N\in\N$, we have $\lcm (1,\dotsc,N) \leq 3^{N}$.
\end{lemma}
For the purposes of the present article, a weaker estimate using $\lcm (1,\dotsc,N) \leq N!$ and Stirling's formula would also suffice, but Lemma~\ref{lem:lcm} has a more appealing form.

By Lemma~\ref{lem:lcm}, we have
\begin{align}
  \label{eq:9}
Q_{0} \leq \lcm (1,\dotsc,\floor{N^{\rho/2}})^{2D} \leq 3^{2D N^{\rho/2}},
\end{align}
so
\begin{align}
\label{eq:116}
Q_0N \le 3^{2D N^{\rho/2}} N \le e^{N^{\rho}}
\end{align}
for sufficiently large $N$ (depending on $\rho$), and we obtain \eqref{eq:IW-denom-bound}.
Finally, $\lcm P_{N} = \lcm(1,\dotsc,N)$, and \eqref{eq:IW-denom-lcm} follows from Lemma~\ref{lem:lcm}.

For $q\in\N$, the set of all fractions with denominator $q$ will be denoted by
\[
\scrQ(q) := \N_{q}^{d}/q
\]
($\scrQ$ for ``quotients'').
The set of coprime elements in the fundamental domain of $(\Z/q\Z)^{d}$ is denoted by
\begin{align}
\label{eq:99}
A_q
:=
\Set[\big]{ a\in\N_q^d \given \gcd(q, a_{1}, \dotsc, a_{d}) = 1 }.
\end{align}
The set of reduced fractions with denominator $q$ is given by
\[
\calR(q) := A_{q}/q
\]
($\calR$ for ``reduced''),
and for $S\subseteq\N$ we write
\[
\calR(S) := \bigcup_{q\in S}\calR(q).
\]
Finally, for each $N\in\N$, we will consider
\begin{align}
\label{eq:156}
\scrU_N:=\calR(P_N).
\end{align}
Note for future reference that $N_1\le N_2$ implies $\scrU_{N_1}\subseteq \scrU_{N_2}$ and
\begin{equation}
\label{eq:UN-size}
\card{\scrU_N} \lesssim e^{(d+1)N^{\rho}}.
\end{equation}

Since the estimate \eqref{eq:111} clearly holds with $\log N$ replaced by $\abs{\scrU_{N}}$, we may assume that $N$ is larger than some constant depending on $p$ and $\rho$.

\subsection{Partitioning denominators into product sets}
The main idea of the proof of Theorem~\ref{thm:IW-mult} is to split $P_N$ into $O_{\rho}(\log N)$ disjoint subsets and show that the operator norm of the multiplier \eqref{eq:IW-mult} with summation restricted to the reduced fractions whose denominators correspond to a fixed set of the aforementioned partition of $P_N$  has a  bound independent of $N$.

For this purpose, notice that
\[
\Set{ n\in\N_{N} \given \text{all prime factors of } n \text{ are in } \Prime_{N}}
\subseteq
\Pi(\Prime_{N}),
\]
where, for a subset $V\subseteq\Prime$, we define
\[
\Pi(V):=\bigcup_{k=0}^{D} \Pi_k(V)
\]
with $\Pi_0(V)=\Set{1}$ and, for $k\in\N$,
\[
\Pi_k(V):=\Set{p_{1}^{\gamma_{1}}\dotsm p_{k}^{\gamma_{k}} \given \ \gamma_{1},\dotsc,\gamma_{k}\in\N_D\ \text{and}\ p_{1},\dotsc,p_{k}\in V\ \text{are distinct} }
\]
is the set of numbers with exactly $k$ distinct prime factors in $V$ at powers from $\N_D$.

We partition the set of denominators $\Pi(\Prime_{N})$ into disjoint subsets with property $\calO$ introduced by Ionescu and Wainger.
\begin{definition}[{\cite{MR2188130}}]
\label{def:propO}
A subset $\Lambda\subseteq\Pi(V)$ has \emph{property $\calO$} if there exists $k\in\Set{0,\dotsc,D}$ and pairwise disjoint sets $S_1,\dotsc, S_k \subseteq \Pi_{1}(V)$ such that
\begin{enumerate}
\item $\Lambda \subseteq S_{1}\dotsm S_{k}$ and
\item the elements of $S_{1} \cup \dotsb \cup S_{k}$ are pairwise coprime.
\end{enumerate}
\end{definition}
Notice that each $S_{j}$ above consists of pure powers of primes in $V$.

The next result was proved in \cite[Lemma 3.1]{MR2188130} with the bound $O_{\rho}((\log N)^{D-1})$ instead of $O_{\rho}(\log N)$.
\begin{lemma}[{\cite[Lemma 5.1]{arXiv:1512.07518}}]
\label{lem:partition-O}
For every $\rho>0$ and every $N\in\N$, the set $\Pi(\Prime_N)$ can be partitioned into $O_{\rho}(\log N)$ sets with property $\calO$.
\end{lemma}
The proof is based on the following probabilistic argument.

\begin{lemma}[{\cite[Lemma 5.2]{arXiv:1512.07518}}]
\label{lem:31}
For every $k\in\N$ and every finite set $V$, there exists a natural number $r\lesssim_{k} \log \card{V}$ and surjective functions
\[
f_{1},\dotsc,f_{r} : V \to \N_{k}
\]
such that for every subset $E\subseteq V$ with $\card{E} \geq k$ there exists $i\in\N_r$ with $\card{f_{i}(E)}=k$.
\end{lemma}

\begin{proof}
We will assume $2\leq k\leq \card{V}$, since other cases are easy.
Moreover, it suffices to construct the functions $f_{i}$ without the restriction of them being surjective, since the non-surjective functions can be dropped.
It also suffices to consider $E\subseteq V$ with $\card{E}=k$.
Denote the set of such subsets by $\binom{V}{k}$.
Let $k^{V}$ denote the set of all functions $f: V\to\N_k$.
Note that, for every set $E \in\binom{V}{k}$, we have
\[
\card{\Set{f\in k^{V} \given \card{f(E)}=k}}
=
\frac{k!}{k^k}\card{k^{V}}.
\]
Let $r=\ceil[\big]{\frac{k^{k+1}}{k!}\ln\abs{V}}$, and suppose for a contradiction that the set
\[
\Set{(f_1,\dotsc, f_r)\in (k^{V})^r \given \forall_{E \in\binom{V}{k}}\ \exists_{1\le m\le r} \ \card{f_m(E)}=k}
\]
is empty.
Then
\begin{align*}
\card{k^{V}}^r
&=
\card{\Set{(f_1,\dotsc, f_r)\in (k^{V})^r \given \exists_{E\in\binom{V}{k}} \forall_{1\le m\le r} \ \card{f_m(E)}<k}}\\
&\leq
\sum_{E \in\binom{V}{k}} \card{\Set{(f_1,\dotsc, f_r)\in (k^{V})^r\given
\forall_{1\le m\le r}\ \card{f_m(E)}<k}}\\
&=
\sum_{E \in\binom{V}{k}}
\card{\Set{f\in k^{V} \given  \card{f(E)}<k}}^r \\
&=
\sum_{E \in\binom{V}{k}}
\bigg(1-\frac{k!}{k^k}\bigg)^r\card{k^{V}}^r\\
&=
\binom{\card{V}}{k}
\bigg(1-\frac{k!}{k^k}\bigg)^r\card{k^{V}}^r.
\end{align*}
Dividing both sides by $\card{k^{V}}^{r}$, we get the contradiction
\[
1
\leq
\binom{\card{V}}{k}\big(1-{k!}{k^{-k}}\big)^r
<
\card{V}^{k} e^{-r\frac{k!}{k^k}}
=
e^{k\ln\abs{V}-r\frac{k!}{k^k}}
\leq
1.
\qedhere
\]
\end{proof}

\begin{proof}[Proof of Lemma~\ref{lem:partition-O}]
Since each subset of set with property $\calO$ also has property $\calO$, it suffices to show that, for every $V\subseteq \Prime$ and $k\in\N_{D}$, the set $\Pi_{k}(V)$ can be written as a (not necessarily disjoint) union of $O_{\rho}(\log \card{V})$ sets with property $\calO$.
Let $f_{1},\dotsc,f_{r} : V \to \N_{k}$ be functions given by Lemma~\ref{lem:31}.
Then the claim is witnessed by the decomposition
\[
\Pi_k(V)
=
\bigcup_{\gamma\in\N_D^k} \bigcup_{i=1}^{r} \Set{p_{1}^{\gamma_1}\dotsm p_{k}^{\gamma_k} \given  f_{i}(p_{j})=j \text{ for each } 1\leq j\leq k}.
\qedhere
\]
\end{proof}
By the Chinese remainder theorem, we have $\calR(\Lambda\cdot \Lambda') = \calR(\Lambda) + \calR(\Lambda')\mod \Z^d$, whenever each element of $\Lambda$ is coprime to each element of $\Lambda'$.
By Lemma~\ref{lem:partition-O}, the set $\scrU_{N}$ can be partitioned into $O_{\rho}(\log N)$ sets of the form
\[
\scrQ(Q_{0}) + \calR(\Lambda) \mod \Z^d,
\]
where each
\begin{equation}
\label{eq:300}
\Lambda \subseteq S_{1}\dotsm S_{k}
\end{equation}
is a set with property $\calO$ as in Definition~\ref{def:propO} with $S_{1},\dots,S_{k}\subseteq \Pi_{1}(\Prime_N) \cap \N_{N}$.
Therefore,
\begin{align*}
\scrQ(Q_{0}) + \calR(\Lambda)
&=
\scrQ(Q_{0}) + \bigcup_{\substack{q_{1}\in S_{1},\dotsc,q_{k}\in S_{k} :\\ q_{1}\dotsm q_{k}\in\Lambda}} \calR(q_{1}\dotsm q_{k})\\
&=
\scrQ(Q_{0}) + \bigcup_{\substack{q_{1}\in S_{1},\dotsc,q_{k}\in S_{k} :\\ q_{1}\dotsm q_{k}\in\Lambda}} (A_{q_{1}}/q_{1} + \dotsb + A_{q_{k}}/q_{k})\mod \Z^d.
\end{align*}
It follows that $\Delta_{N}(\xi)$ can be written as the sum of $O_{\rho}(\log N)$ multipliers of the form
\[
\sum_{q_1\in S_{1}}\sum_{a_{1}\in A_{q_{1}}}\dotsi\sum_{q_k\in S_{k}}\sum_{a_{k}\in A_{q_{k}}}
\underbrace{\ind{\Lambda}(q_{1}\dotsm q_{k})\sum_{b\in\scrQ(Q_0)}
\Theta\Big(\xi-b-\sum_{j=1}^ka_{j}/q_{j}\Big)}_{=:m^{\Lambda}_{a_{1}/q_{1}+\dotsb+a_{k}/q_{k} }(\xi)}.
\]
It suffices to obtain $L^{p}$ bounds, which do not depend on $N$, for these multipliers.
We will do so for even integer exponents $p=2r$ with $r\in\N$.
The case $p' \in 2\N$ can be reduced to this case by duality, considering the adjoint multiplier $\Theta^{*}$.

From now on, for brevity, we will use the notation
\begin{equation}
\label{eq:def-fu}
f_u(x)
:=
\FT^{-1}\big(m_u^{\Lambda}\hat{f}\big)(x)
\quad \text{for}\quad
u\in \calR(S_{1}\dotsm S_{k}).
\end{equation}

We restate our claim with the newly introduced notation.

\begin{theorem}
\label{thm:IW:fu}
For every $\rho>0$ and $r\in\N$, there exist constants  $0<c_{r, \rho}, C_{r, \rho, d} < \infty$ such that, for every $N \ge c_{r, \rho}$ and every set $\Lambda \subseteq S_{1}\dotsm S_{k} \subseteq \Pi(\Prime_{N})$ as in \eqref{eq:300}, we have
\begin{align}
\label{eq:677}
\norm[\Big]{\sum_{u\in \calR(S_{1}\dotsm S_{k})}f_u}_{\ell^{2r}(\Z^{d};H)}
\le C_{r, \rho, d}\boldA_{2r}\norm{f}_{\ell^{2r}(\Z^d;H)}
\end{align}
for every
$f\in \ell^{2r}(\Z^d;H)$.
\end{theorem}

In the remaining part of Section~\ref{sec:iw} we prove Theorem~\ref{thm:IW:fu}.

\subsection{Uniqueness property and orthogonality}
\label{sec:IW:red:comb}
\begin{definition}[{\cite[Section 2]{MR2188130}}]
\label{def:propU}
A finite sequence $(x_1, x_2,\dotsc, x_m)$ has the \emph{uniqueness property} if there is $k\in\N_m$ such that $x_l\not=x_k$ for every $l\in\N_{m}\setminus\Set{k}$.
\end{definition}
We will need the following combinatorial fact that reduces to a particularly easy special case of Hall's marriage theorem.

\begin{lemma}
\label{lem:part}
Let $r\in\N$, and let $(q_{1}(0), q_{1}(1), \ldots, q_{r}(0), q_{r}(1))$ be a sequence which does not have the uniqueness property.
Then there exists a function $\kappa : \N_{r} \to \Set{0,1}$ such that
\begin{equation}
\label{eq:cor:part}
\Set{ q_j(\kappa(j)) \given  j\in\N_{r}}
=
\Set{ q_j(1-\kappa(j)) \given  j\in\N_{r}}
=
\Set{ q_j(i) \given  j\in\N_{r}, i\in\Set{0,1}}.
\end{equation}
\end{lemma}
\begin{proof}
We may assume that the set $Z:=\Set{q_{1}(0), q_{1}(1), \ldots, q_{r}(0), q_{r}(1)}$ has cardinality $r$, that is, each element in the image of the sequence has multiplicity $2$.
If this was not the case, then either at least two elements $a,b$ would have multiplicity $\geq 3$, or at least one element $a$ would have multiplicity $\geq 4$.
In the first case, replace one of the occurrences of $a$ and $b$ by a new symbol, and in the second case replace two of the occurrences of $a$ by a new symbol.
This increases $\abs{Z}$ while preserving the lack of uniqueness property, and the function $\kappa$ constructed for the new sequence still works for the old sequence.

We define a bipartite multigraph with vertex sets $\Set{1,\dotsc,r}$ and $Z$, in which each $q_{i}(l)$ defines an edge between $i$ and $z=q_{i}(l)$.
Then each vertex appears in exactly $2$ edges, and hence the graph consists of finitely many cycles, each of which has even length since the graph is bipartite.
In each cycle, we color the edges alternatingly red and blue.
Each element of $\Set{1,\dotsc,r}$ is contained in exactly one red and one blue edge, and we declare $\kappa(i)$ to be the number for which the edge $(i,q_i(\kappa(i)))$ is red; the edge $(i,q_i(1-\kappa(i)))$ is then blue.
Since each element of $Z$ is contained in one red and one blue edge, this ensures \eqref{eq:cor:part}.
\end{proof}

\begin{corollary}
\label{cor:up-est}
Let $(X,\calB(X), \mu)$ be a measure space, $r\in\N$, let $S_{1},\dotsc,S_{r}$ be finite sets (not necessarily disjoint), and let $F^{i}_{q} \in L^{2r}(X;H)$ for every $i\in\N_r$ and $q\in S_{i}$.
Suppose that, for every sequence
\begin{equation}
\label{eq:qil}
(q_{1}(0), q_{1}(1), \ldots, q_{r}(0), q_{r}(1))
\in S_{1}^{2} \times \dotsm \times S_{r}^{2}
\end{equation}
with the uniqueness property, we have
\begin{equation}
\label{eq:up-orth}
\int_{X} \prod_{i=1}^{r} \innerp{F^{i}_{q_{i}(0)}(x)}{F^{i}_{q_{i}(1)}(x)}_{H} \dif \mu(x) = 0.
\end{equation}
Then
\begin{equation}
\label{eq:up-est}
\int_{X} \prod_{i=1}^{r} \norm[\big]{\sum_{q\in S_{i}} F^{i}_{q}(x)}_{H}^{2}\dif\mu(x)
\lesssim_{r}
\int_{X} \prod_{i=1}^{r} \Big( \sum_{q\in S_{i}} \norm{F^{i}_{q}(x)}_{H}^{2} \Big)\dif\mu(x).
\end{equation}
The implicit constant does not depend on $X$ and $S_{1},\dotsc,S_{r}$.
\end{corollary}
\begin{proof}
For any $Z\subseteq S_{1} \cup \dotsb \cup S_{r}$ with $\abs{Z}\le 2r$, let
\begin{align*}
S_Z(x)=\sum\prod_{i=1}^{r} \innerp{F^{i}_{q_{i}(0)}(x)}{F^{i}_{q_{i}(1)}(x)}_{H},
\end{align*}
where the summation is taken over all sequences as in \eqref{eq:qil} which do \emph{not} have the uniqueness property and $\Set{q_{1}(0), q_{1}(1), \dotsc, q_{r}(0), q_{r}(1)}=Z$.

Expanding the product on the left-hand side of \eqref{eq:up-est}, we obtain
\[
\sum_{(q_{1}(0), q_{1}(1), \ldots, q_{r}(0), q_{r}(1))} \int_{X} \prod_{i=1}^{r} \innerp{F^{i}_{q_{i}(0)}(x)}{F^{i}_{q_{i}(1)}(x)}_{H} \dif\mu(x)
=\sum_{\abs{Z}\le r}\int_{X}S_Z(x) \dif\mu(x),
\]
since the integral vanishes for all summation sequences with the uniqueness property by the hypothesis \eqref{eq:up-orth}.
For each sequence $(q_{1}(0), q_{1}(1), \ldots, q_{r}(0), q_{r}(1))$ without the uniqueness property such that $\Set{q_{1}(0), q_{1}(1), \dotsc, q_{r}(0), q_{r}(1)}=Z$, we apply Lemma~\ref{lem:part}, and we obtain
\begin{align}
\label{eq:up-est:2}
\begin{split}
\abs[\Big]{ \prod_{i=1}^{r} \innerp{F^{i}_{q_{i}(0)}(x)}{F^{i}_{q_{i}(1)}(x)}_{H} }
&\le
\prod_{i=1}^{r} \norm{F^{i}_{q_{i}(\kappa(i))}(x)}_H \norm{F^{i}_{q_{i}(1-\kappa(i))}(x)}_H\\
&\leq
\frac12 \prod_{i=1}^{r} \norm{F^{i}_{q_{i}(\kappa(i))}(x)}_H^{2}
+
\frac12 \prod_{i=1}^{r} \norm{F^{i}_{q_{i}(1-\kappa(i))}(x)}_H^{2}\\
&\le\sum_{\substack{q_1 \in S_{1},\dotsc, q_r\in S_{r}:\\ \Set{q_1,\dotsc, q_r}=Z}}\prod_{i=1}^{r}\norm{F^{i}_{q_{i}}(x)}_H^{2}.
\end{split}
\end{align}
Since for each $Z$ there are $O_{r}(1)$ sequences of length $2r$ taking values in $Z$, by \eqref{eq:up-est:2}, we conclude
\begin{align}
\begin{split}
\sum_{\abs{Z}\le r}
\int_{X}S_Z(x) \dif\mu(x)
&\lesssim_r\int_{X}\sum_{\abs{Z}\le r}
\sum_{\substack{q_1 \in S_{1},\dotsc, q_r\in S_{r}:\\ \Set{q_1,\dotsc, q_r}=Z}}\prod_{i=1}^{r}\norm{F^{i}_{q_{i}}(x)}_H^{2} \dif\mu(x)\\
&=
\int_{X} \prod_{i=1}^{r} \Big( \sum_{q\in S_{i}} \norm{F^{i}_{q}(x)}_{H}^{2} \Big)\dif\mu(x). \qedhere
\end{split}
\end{align}
\end{proof}

\subsection{Orthogonality between denominators}
For every $j\in\N_k$, we will show
\begin{multline}
\label{eq:128}
\sum_{x\in\Z^{d}} \Big(\sum_{q_{1}\in S_{1}} \dotsb \sum_{q_{j-1}\in S_{j-1}} \norm[\big]{\sum_{u\in \calR(q_{1} \dotsm q_{j-1}\cdot S_{j}\dotsm S_k)}f_u(x)}_{H}^{2} \Big)^r\\
=
\sum_{q_{1,1},\dotsc,q_{1,r}\in S_{1}} \dotsb \sum_{q_{j-1,1},\dotsc,q_{j-1,r}\in S_{j-1}} \sum_{x\in\Z^{d}} \prod_{i=1}^{r} \norm[\big]{\sum_{q_j\in S_j}\sum_{u\in \calR(q_{1,i} \dotsm q_{j-1,i}q_j\cdot S_{j+1}\dotsm S_k)}f_u(x)}_{H}^{2} \\
\lesssim_{r}
\sum_{q_{1,1},\dotsc,q_{1,r}\in S_{1}} \dotsb \sum_{q_{j-1,1},\dotsc,q_{j-1,r}\in S_{j-1}} \sum_{x\in\Z^{d}}
\prod_{i=1}^{r} \Big(\sum_{q_{j}\in S_{j}} \norm[\big]{\sum_{u\in \calR(q_{1} \dotsm q_{j}\cdot S_{j+1}\dotsm S_k)}f_u(x)}_{H}^{2} \Big).
\end{multline}
The last inequality follows from Corollary~\ref{cor:up-est}, since
for every sequence with the uniqueness property $(q_{j,1}(0), q_{j,1}(1), \ldots, q_{j,r}(0), q_{j,r}(1))\in S_j^{2r}$
the orthogonality condition \eqref{eq:up-orth}
is satisfied. Namely,
\begin{align}
\label{eq:129}
\sum_{x\in\Z^{d}}\prod_{i=1}^{r} \innerp[\big]{\sum_{u\in \calR(q_{1,i} \dotsm q_{j-1,i}q_{j, i}(0)\cdot S_{j+1}\dotsm S_k)}f_u(x)}{\sum_{u\in \calR(q_{1,i} \dotsm q_{j-1,i}q_{j, i}(1)\cdot S_{j+1}\dotsm S_k)}f_u(x)}_{H}=0.
\end{align}
In order to verify condition \eqref{eq:129}, we note that the function under the sum in \eqref{eq:129} can be written as a finite sum of functions of the form
\begin{align}
\label{eq:131}
\prod_{i=1}^{r} \innerp[\big]{\sum_{u\in\calR(q_{1,i}(0)\dotsm q_{k,i}(0))}f_{u}(x)}{\sum_{u\in\calR(q_{1,i}(1) \dotsm q_{k,i}(1))} f_{u}(x)}_{H},
\end{align}
where $q_{j,i}(l)\in S_{j}$ for each $i\in\N_r$, $j\in\N_k$, and $l\in\Set{0, 1}$.
Fixing $q_{j,i}(l)\in S_{j}$ for each $i\in\N_r$, $j\in\N_k$ and $l\in\Set{0, 1}$, it is not difficult to see (by the Chinese remainder theorem) that the Fourier transform of the function \eqref{eq:131} is supported in the set
\begin{align}
\label{eq:134}
\bigcup_{b\in\scrQ(Q_{0})}\bigcup_{u\in \sum_{i=1}^{r} \sum_{j=1}^{k} \sum_{l\in\Set{0, 1}} \pm\calR(q_{j,i}(l))}\big(b+u+2r \epsilon_N \bfQ\big),
\end{align}
which does not contain zero provided that for some $j\in\N_k$ the sequence
\begin{align}
\label{eq:135}
(q_{j,1}(0), q_{j,1}(1), \ldots, q_{j,r}(0), q_{j,r}(1))\in S_j^{2r}
\end{align}
has the uniqueness property.
Indeed, if the set in \eqref{eq:134} does contain zero, then $\abs{b+u}_{\infty}\le r\epsilon_N$ for some $b\in\scrQ(Q_{0})$ and $u\in \sum_{i=1}^{r} \sum_{j=1}^{k} \sum_{l\in\Set{0, 1}} \pm\calR(q_{j,i}(l))$.
Due to the uniqueness property, we can assume, without loss of generality, that $q_{j,1}(0)\neq q_{j,i}(l)$ unless $i=1$ and $l=0$.
Hence, $b+u \neq 0$ can be written as a fraction with denominator at most $Q_{0} N^{2kr}$, so that
\[
Q_{0}^{-1}N^{-2kr}
\leq
\abs{b+u}_{\infty}
\le
re^{-N^{\rho}},
\]
which is impossible for sufficiently large $N$, due to \eqref{eq:9}.

Using \eqref{eq:128} in each step, we obtain the following chain of estimates:
\begin{align}
\label{eq:158}
\nonumber\norm[\big]{\sum_{u\in \calR(S_{1}\dotsm S_{k})}f_u}_{\ell^{2r}(\Z^{d};H)}^{2r}
&=
\sum_{x\in\Z^{d}} \prod_{i=1}^{r}\norm[\big]{\sum_{u\in \calR(S_{1} S_{2} \dotsm S_{k})}f_u(x)}_{H}^{2}\\
\nonumber&\lesssim_{r}
\sum_{x\in\Z^{d}} \Big( \sum_{q_{1}\in S_{1}} \norm[\big]{\sum_{u\in \calR(q_{1} S_{2} \dotsm S_{k})}f_u(x)}_{H}^{2} \Big)^{r}\\
& \lesssim_{r} \dotsb \lesssim_{r, \rho}
\sum_{x\in\Z^{d}} \Big( \sum_{q_{1}\in S_{1}} \dotsb \sum_{q_{k}\in S_{k}} \norm[\big]{\sum_{u\in \calR(q_{1} \dotsm q_{k})}f_u(x)}_{H}^{2} \Big)^{r}.
\end{align}
Notice that we are summing positive quantities over the product set $S_{1}\dotsm S_{k}$, so at this point we may drop the characteristic function $\ind{\Lambda}$ from the definition of the multipliers $m_{u}^{\Lambda}$.

We have already exhausted all orthogonalities between
denominators. The task now is to exploit orthogonalities between
numerators.

\subsection{Orthogonality between numerators}
The following result allows us to split summation into diagonal and fully off-diagonal terms.
\begin{lemma}[{\cite[Lemma 2.3]{MR2188130}}]
\label{lem:l1-lr}
For every $n, r\in\N$ and arbitrary numbers $a_1, \dotsc, a_n\ge0$, we have
\begin{equation}
\label{eq:26}
(a_1+\dotsb+a_n)^r
\leq
(r(r-1))^{r-1} \sum_{1\le i\le n}a_i^r+2\sum_{\substack{i_1,\dotsc,i_r\in\N_n\\\text{pairwise distinct}}}a_{i_1}\dotsm a_{i_r}.
\end{equation}
\end{lemma}
We give a simplified proof.
\begin{proof}
We may assume $r\geq 2$.
The inequality \eqref{eq:26} is clearly verified if
\begin{align}
\label{eq:136}
(a_1+\dotsb+a_n)^r
\leq
2 \sum_{\substack{i_1,\dotsc,i_r\in\N_n\\\text{pairwise distinct}}}a_{i_1}\dotsm a_{i_r}.
\end{align}
Otherwise, expanding the $r$-th power, we obtain
\[
(a_1+\dotsb+a_n)^r
\leq
\sum_{\substack{i_1,\dotsc,i_r\in\N_n\\\text{pairwise distinct}}}a_{i_1}\dotsm a_{i_r} + \frac{r(r-1)}{2} (a_1^{2}+\dotsb+a_n^{2})(a_1+\dotsb+a_n)^{r-2}.
\]
Using the failure of inequality \eqref{eq:136}, this implies
\[
(a_1+\dotsb+a_n)^r \leq r(r-1) (a_1^{2}+\dotsb+a_n^{2})(a_1+\dotsb+a_n)^{r-2},
\]
so that $\norm{a}_{\ell^{1}} \leq \sqrt{r(r-1)} \norm{a}_{\ell^{2}}$.
By logarithmic convexity of $\ell^{p}$ norms, this implies
\[
\norm{a}_{\ell^{1}}
\leq
\sqrt{r(r-1)} \norm{a}_{\ell^{1}}^{1-\theta} \norm{a}_{\ell^{r}}^{\theta},
\quad \text{with} \quad
\frac{1}{2} = \frac{1-\theta}{1} + \frac{\theta}{r}.
\]
Thus, $\norm{a}_{\ell^{1}} \leq (r(r-1))^{1/(2\theta)} \norm{a}_{\ell^{r}} = (r(r-1))^{1-1/r} \norm{a}_{\ell^{r}}$, and this also shows \eqref{eq:26}.
\end{proof}

For any $M\subseteq \N_k$, let
\[
S_{M} := \prod_{j\in M} S_{j},
\]
with the convention $S_{\emptyset}=\Set{1}$.
Suppose that $M\subseteq \N_k$ and $L\subseteq M$ satisfy $M\setminus L\not=\emptyset$, then, for $j=\min M\setminus L$, we will show
\begin{align}
\label{eq:138}
\begin{split}
\sum_{x\in\Z^{d}} \sum_{\sigma\in S_L} \Big( \sum_{\tau\in S_{M\setminus L}}\ &\sum_{w\in\calR(S_{M^c})} \norm[\big]{\sum_{u\in \calR(\sigma\tau)}f_{u+w}(x)}_{H}^{2} \Big)^r\\
& \lesssim \sum_{x\in\Z^{d}} \sum_{\sigma\in S_{L\cup\Set{j}}} \Big( \sum_{\tau\in S_{M\setminus (L\cup\Set{j})}}\ \sum_{w\in\calR(S_{M^c})} \norm[\big]{\sum_{u\in \calR(\sigma\tau)}f_{u+w}(x)}_{H}^{2} \Big)^r\\
& +\sum_{x\in\Z^{d}} \sum_{\sigma\in S_{L}} \Big(
\sum_{\tau\in S_{M\setminus
(L\cup\Set{j})}}\ \sum_{w\in\calR(S_{(M\setminus\Set{j})^c})}
\norm[\big]{\sum_{u\in
\calR(\sigma\tau)}f_{u+w}(x)}_{H}^{2} \Big)^r.
\end{split}
\end{align}
Here we use the convention $M^c=\N_k\setminus M$ for subsets $M\subseteq \N_k$, in particular, $\N_k^c=\emptyset$.
Assume momentarily that \eqref{eq:138} holds.
By iterative application of \eqref{eq:138}, we obtain
\begin{align}
\label{eq:155}
\begin{split}
\MoveEqLeft
\sum_{x\in\Z^{d}} \Big( \sum_{q_{1}\in S_{1}} \dotsb \sum_{q_{k}\in S_{k}} \norm[\big]{\sum_{u\in \calR(q_{1} \dotsm q_{k})}f_u(x)}_{H}^{2} \Big)^{r}\\
&=\sum_{x\in\Z^{d}} \sum_{\sigma \in S_{\emptyset}} \Big( \sum_{\tau\in S_{\N_k}} \norm[\big]{\sum_{u\in \calR(\sigma\tau)}f_u(x)}_{H}^{2} \Big)^{r}\\
&\lesssim_r \sum_{x\in\Z^{d}} \sum_{\sigma\in S_{\Set{1}}} \Big( \sum_{\tau\in S_{\N_k\setminus \Set{1}}}\ \sum_{w\in\calR(S_{\N_k^c})} \norm[\big]{\sum_{u\in \calR(\sigma\tau)}f_{u+w}(x)}_{H}^{2} \Big)^r\\
& +\sum_{x\in\Z^{d}} \sum_{\sigma\in S_{\emptyset}} \Big(
\sum_{\tau\in S_{\N_k\setminus
\Set{1}}}\ \sum_{w\in\calR(S_{(\N_k\setminus\Set{1})^c})}
\norm[\big]{\sum_{u\in
\calR(\sigma\tau)}f_{u+w}(x)}_{H}^{2} \Big)^r\\
&\lesssim\ldots\lesssim_{r, \rho}\sum_{x\in\Z^{d}}\sum_{M\subseteq\N_k}
\sum_{\sigma\in S_{M}} \Big(\sum_{w\in\calR(S_{M^c})} \norm[\big]{\sum_{u\in \calR(\sigma)}f_{u+w}(x)}_{H}^{2} \Big)^r.
\end{split}
\end{align}
Then, by \eqref{eq:158} and \eqref{eq:155}, we get
\begin{align}
\label{eq:162}
\norm[\big]{\sum_{u\in \calR(S_{1}\dotsm S_{k})}f_u}_{\ell^{2r}(\Z^{d};H)}^{2r} \lesssim_{r, \rho}
\sum_{x\in\Z^{d}}\sum_{M\subseteq\N_k}
\sum_{\sigma\in S_{M}} \Big(\sum_{w\in\calR(S_{M^c})} \norm[\big]{\sum_{u\in \calR(\sigma)}f_{u+w}(x)}_{H}^{2} \Big)^r.
\end{align}

In order to prove \eqref{eq:138}, we use Lemma~\ref{lem:l1-lr} and obtain
\begin{align}
\MoveEqLeft
\sum_{x\in\Z^{d}} \sum_{\sigma\in S_L} \Big( \sum_{\tau\in S_{M\setminus L}} \sum_{w\in\calR(S_{M^c})} \norm[\big]{\sum_{u\in \calR(\sigma\tau)}f_{u+w}(x)}_{H}^{2} \Big)^r\nonumber \\
& =\sum_{x\in\Z^{d}} \sum_{\sigma\in S_L} \Big( \sum_{q\in S_j}\sum_{\tau\in S_{M\setminus(L\cup\Set{j})}}\ \sum_{w\in\calR(S_{M^c})} \norm[\big]{\sum_{u\in \calR(\sigma\tau q)}f_{u+w}(x)}_{H}^{2} \Big)^r \nonumber\\
& \lesssim \sum_{x\in\Z^{d}} \sum_{\sigma\in S_L} \sum_{q\in S_j} \Big(\sum_{\tau\in S_{M\setminus(L\cup\Set{j})}}\ \sum_{w\in\calR(S_{M^c})} \norm[\big]{\sum_{u\in \calR(\sigma\tau q)}f_{u+w}(x)}_{H}^{2} \Big)^r\label{eq:139}\\
&+\sum_{\sigma\in S_L}\sum_{\substack{q_{j,1},\dotsc,q_{j,r}\in S_{j}\\\text{pairwise distinct}}}\sum_{x\in\Z^{d}} \prod_{i=1}^r
\Big(\sum_{\tau\in S_{M\setminus(L\cup\Set{j})}}\ \sum_{w\in\calR(S_{M^c})} \norm[\big]{\sum_{u\in \calR(\sigma\tau q_{j, i})}f_{u+w}(x)}_{H}^{2} \Big)\label{eq:140}.
\end{align}
The expression \eqref{eq:139} appears on the right-hand side of \eqref{eq:138}.
It remains to estimate \eqref{eq:140}.
By Corollary~\ref{cor:up-est}, for all pairwise distinct $q_{j,1},\dotsc,q_{j,r}\in S_{j}$, all $\tau_1,\ldots,\tau_r\in S_{M\setminus(L\cup\Set{j})}$, and all $w_1,\ldots, w_r\in\calR(S_{M^c})$, we obtain
\begin{align}
\label{eq:num-orth-1}
\begin{split}
\sum_{x\in\Z^{d}} \prod_{i=1}^r\norm[\big]{\sum_{u\in \calR(\sigma\tau_i q_{j, i})}&f_{u+w_i}(x)}_{H}^{2}
=\sum_{x\in\Z^{d}} \prod_{i=1}^r\norm[\big]{\sum_{v\in\calR(q_{j, i})}\sum_{u\in \calR(\sigma\tau_i)}f_{u+w_i+v}(x)}_{H}^{2}\\
& \lesssim_{r}
\sum_{x\in\Z^{d}} \prod_{i=1}^r\Big(\sum_{v\in\calR(q_{j, i})}\norm[\big]{\sum_{u\in \calR(\sigma\tau_i)} f_{u+w_i+v}(x)}_{H}^{2} \Big).
\end{split}
\end{align}
The last inequality follows from \eqref{eq:up-est}, since the orthogonality condition \eqref{eq:up-orth} is satisfied.
Namely,
\begin{align}
\label{eq:141}
\sum_{x\in\Z^{d}} \prod_{i=1}^r \innerp[\big]{\sum_{u\in \calR(\sigma\tau_i)}f_{u+w_i+v_i(0)}(x)}{\sum_{u\in \calR(\sigma\tau_i)}f_{u+w_i+v_i(1)}(x)}_H=0
\end{align}
holds for every sequence
\begin{align}
\label{eq:174}
(v_1(0), v_1(1), \ldots, v_r(0), v_r(1))
\in
\prod_{i=1}^r \bigl(\calR(q_{j, i})\times \calR(q_{j, i})\bigr)
\end{align}
with the uniqueness property.
In order to verify \eqref{eq:141}, we note that the function under the sum in \eqref{eq:141} can be written as a finite sum of the functions of the form
\begin{align}
\label{eq:142}
\prod_{i=1}^r \innerp[\big]{f_{u_i(0)+w_i+v_i(0)}(x)}{f_{u_i(1)+w_i+v_i(1)}(x)}_H,
\end{align}
where $u_i(l)\in\calR(\sigma\tau_i)$ for all $i\in\N_r$ and $l\in\Set{0, 1}$.
One can easily see that, fixing $v_i(l)\in\calR(q_{j, i})$ and $u_i(l)\in\calR(\sigma\tau_i)$ for all $i\in\N_r$ and $l\in\Set{0, 1}$, the Fourier transform of the function from \eqref{eq:142} is supported in the set
\begin{align}
\label{eq:143}
\bigcup_{b\in\scrQ(Q_{0})}\Big(b+\sum_{i=1}^{r} (v_i(0)-v_i(1)+u_i(0)-u_i(1))+2r \epsilon_N \bfQ\Big),
\end{align}
which does not contain zero due to the uniqueness property of \eqref{eq:174}.
Indeed, suppose for a contradiction that the set in \eqref{eq:143} does contain zero.
This means that $\abs{b+u}_{\infty}\le r\epsilon_N$ for some $b\in\scrQ(Q_{0})$ and $u=\sum_{i=1}^{r} v_i(0)-v_i(1)+u_i(0)-u_i(1)$.
Due to the uniqueness property we can assume, without loss of generality, that $v_1(0)\not=v_1(1)$.
Also, the denominators of $v_i(l)$, $i\neq 1$, and of $u_{i}(l)$ for all $i$ are coprime to the denominator of $v_{1}(0)$.
Hence $b+u \neq 0$ can be written as a fraction with denominator at most $Q_{0} N^{4rD}$.
This implies
\[
Q_{0}^{-1} N^{-4rD} \leq \abs{b+u}_{\infty} \leq re^{-N^{\rho}},
\]
which is a contradiction for sufficiently large $N$, due to \eqref{eq:9}.

Finally, by \eqref{eq:num-orth-1}, we see that
\begin{align*}
\eqref{eq:140}
&=
\sum_{\sigma\in S_L}\sum_{\substack{q_{j,1},\dotsc,q_{j,r}\in S_{j}\\\text{pairwise distinct}}}
\sum_{\substack{\tau_1,\dotsc,\tau_r \\\in S_{M\setminus(L\cup\Set{j})}}}
\ \sum_{\substack{w_1,\dotsc, w_r\\\in\calR(S_{M^c})}}
\sum_{x\in\Z^{d}}
\prod_{i=1}^r\norm[\big]{\sum_{u\in \calR(\sigma\tau_i q_{j, i})}f_{u+w_i}(x)}_{H}^{2}
\\ & \lesssim
\sum_{\sigma\in S_L}
\sum_{\substack{q_{j,1},\dotsc,q_{j,r}\in S_{j}\\\text{pairwise distinct}}}
\sum_{\substack{\tau_1,\dotsc,\tau_r\\\in S_{M\setminus(L\cup\Set{j})}}}
\ \sum_{\substack{w_1,\dotsc, w_r\\\in\calR(S_{M^c})}}
\sum_{x\in\Z^{d}} \prod_{i=1}^r\Big(\sum_{v\in\calR(q_{j, i})}\norm[\big]{\sum_{u\in \calR(\sigma\tau_i)}f_{u+w_i+v}(x)}_{H}^{2} \Big)
\\ &=
\sum_{x\in\Z^{d}}\sum_{\sigma\in S_L}
\Big(\sum_{\tau\in S_{M\setminus(L\cup\Set{j})}}\ \sum_{w\in\calR(S_{M^c})}\sum_{v\in\calR(S_{j})}\norm[\big]{\sum_{u\in \calR(\sigma\tau)}f_{u+w+v}(x)}_{H}^{2} \Big)^r
\\ &=
\sum_{x\in\Z^{d}}\sum_{\sigma\in S_L}
\Big(\sum_{\tau\in S_{M\setminus(L\cup\Set{j})}}\ \sum_{w\in\calR(S_{(M\setminus\Set{j})^c})}\norm[\big]{\sum_{u\in \calR(\sigma\tau)}f_{u+w}(x)}_{H}^{2} \Big)^r,
\end{align*}
as claimed.
We have exhausted all orthogonality provided by disjoint Fourier supports of the functions $f_{u}$.
Now we will have to tackle the multipliers $\Theta$.

\subsection{Estimates for the multiplier $\Theta$}
\label{sec:sq-fct}
For fixed $M\subseteq\N_{k}$ and $\sigma\in S_{M}$, we view the
corresponding term of \eqref{eq:162} as the $2r$-th power of a norm of
a function with values in $\ell^{2}(\calR(S_{M^{c}});H)$. In order to
estimate the right-hand side of \eqref{eq:162}, we will use the following
vector-valued version of the Marcinkiewicz--Zygmund theorem.
\begin{theorem}[vector-valued Marcinkiewicz--Zygmund]
\label{thm:MZ}
Let $H$ be a separable Hilbert space, and $(X, \calB(X), \mu)$ and $(Y, \calB(Y), \nu)$ be $\sigma$-finite measure spaces. For given $p,q\in(0, \infty)$, we assume that $T : L^{p}(X;H) \to L^{q}(Y;H)$ is a bounded linear operator.
Let $\calZ$ be a countable set of indices.
Then, for every sequence of functions $(f_n: n\in\calZ)$, we have
\begin{align}
\label{eq:177}
\norm[\Big]{ \Big( \sum_{n\in\calZ} \norm{Tf_{n}(y)}_{H}^{2} \Big)^{1/2}}_{L^{q}(\dif\nu(y))}
\lesssim_{p,q}
\norm{T}_{L^{p}(X;H)\to L^{q}(Y;H)}
\norm[\Big]{ \Big( \sum_{n\in\calZ} \norm{f_{n}(x)}_{H}^{2} \Big)^{1/2}}_{L^{p}(\dif\mu(x))}.
\end{align}
\end{theorem}
The proof of Theorem~\ref{thm:MZ} is identical to the proof in the
scalar-valued case, but uses the following vector-valued extension of
Khintchine's inequality.
\begin{lemma}
\label{lem:kahane-khintchine}
Let $H$ be a separable Hilbert space and let $(r_n(t): n\in\N)$ be a system of Rademacher functions on $[0, 1]$.
Then, for every $p\in(0, \infty)$ and $(b_{n}: n\in\N)\subseteq H$, we have
\begin{align}
\label{eq:178}
\bigg( \int_0^1 \norm[\Big]{\sum_{n\in\N}r_{n}(t)b_{n}}_{H}^{p} \dif t \bigg)^{1/p}
\simeq_{p}
\bigg( \int_0^1 \norm[\Big]{\sum_{n\in\N}r_{n}(t)b_{n}}_{H}^{2} \dif t \bigg)^{1/2}.
\end{align}
The implicit constant does not depend on $H$.
\end{lemma}
Kahane \cite{MR0169279} showed inequality \eqref{eq:178} with any Banach space
$B$ in place of $H$. However, for our purposes \eqref{eq:178}
suffices. We give a straightforward proof of \eqref{eq:178}, which in
fact will be a consequence of double scalar-valued Khintchine's
inequality.
\begin{proof}[Proof of Lemma~\ref{lem:kahane-khintchine}]
For each $t\in[0, 1]$, we set
\[
B(t)=\norm[\Big]{\sum_{n\in\N} r_n(t)b_n}_H.
\]
Fix an orthonormal basis $(e_m:m\in\N)$ of $H$ and let $(a_{n, m}: m\in\N)$ be the coordinates of $b_n$ in that basis.
For every $n\in\N$, we have
\[
b_n=\sum_{m\in\N}a_{n, m}e_m.
\]
Observe that
\[
B(t)=\norm[\Big]{\sum_{n\in\N} r_n(t)b_n}_H= \bigg(\int_0^1F(t, s)^2\dif s\bigg)^{1/2},
\]
where
\[
F(t, s)=\sum_{n\in\N}\sum_{m\in\N} a_{n, m} r_n(t)r_m(s).
\]
We prove that, for all $p\in(0, \infty)$, we have
\begin{align}
\label{eq:179}
\norm{B}_{L^p}\simeq_p \norm{B}_{L^2}.
\end{align}
We note that \eqref{eq:179} is exactly \eqref{eq:178} rewritten in the new notation.
So if $p\ge2$, then by H{\"o}lder's inequality we obtain
\[
B(t)\le\bigg(\int_0^1\abs{F(t, s)}^p\dif s\bigg)^{1/p},
\]
and
\[
\norm{B}_{L^p}\le\bigg(\int_0^1\int_0^1\abs{F(t, s)}^p\dif t\dif s\bigg)^{1/p}.
\]
Using the scalar-valued double Khinchine inequality, we get
\[
\bigg(\int_0^1\int_0^1\abs{F(t, s)}^p\dif t\dif s\bigg)^{1/p}\lesssim \bigg(\int_0^1\int_0^1\abs{F(t, s)}^2\dif t\dif s\bigg)^{1/2}= \norm{B}_{L^2}.
\]
Since the case $p\in(0, 2)$ is obvious by H{\"o}lder's inequality, we conclude that the inequality
\begin{align}
\label{eq:175}
\norm{B}_{L^p}\lesssim_p \norm{B}_{L^2}
\end{align}
holds for all $p\in(0, \infty)$.
To prove the converse, it suffices to show that
\begin{align}
\label{eq:176}
\norm{B}_{L^2}\lesssim_p \norm{B}_{L^p}
\end{align}
for every $p\in(0, 2)$.
For this purpose, we choose $\theta\in(0, 1)$ such that $2=p\theta +4(1-\theta)$, then
\[
\norm{B}_{L^2}^2\le\norm{B}_{L^p}^{p\theta}\,\norm{B}_{L^4}^{4(1-\theta)}.
\]
By \eqref{eq:175}, we get
\[
\norm{B}_{L^2}^2\le\norm{B}_{L^p}^{p\theta}\,\norm{B}_{L^4}^{4(1-\theta)}\lesssim \norm{B}_{L^p}^{p\theta}\,\norm{B}_{L^2}^{4(1-\theta)}
\]
and, dividing both sides by $\norm{B}_{L^2}^{4(1-\theta)}$, we obtain \eqref{eq:176}.
\end{proof}
\begin{proof}[Proof of Theorem~\ref{thm:MZ}]
We can assume, without loss of generality, that $\calZ=\N$.
By the monotone convergence theorem, we may assume that only finitely many of the functions $f_{n}$ do not vanish.
For any sequence $(b_{n}:n\in\N)\subseteq H$ with finitely many non-vanishing entries, we have
\begin{equation}
\label{eq:l2-Bernoulli}
\sum_{n\in \N} \norm{b_{n}}_{H}^{2}
=
\int_0^1 \innerp[\big]{\sum_{n\in\N} r_{n}(t) b_{n}}{\sum_{m\in\N} r_{m}(t) b_{m}}_{H}\dif t=\int_0^1\norm[\Big]{\sum_{n\in\N} r_{n}(t) b_{n}}_H^2\dif t.
\end{equation}
This allows us to estimate
\begin{align*}
\MoveEqLeft
\norm[\Big]{ \Big( \sum_{n\in\N} \norm{Tf_{n}(y)}_{H}^{2} \Big)^{1/2}}_{L^{q}(\dif \nu(y))}\\
&=
\norm[\bigg]{ \bigg(\int_0^1 \norm[\Big]{\sum_{n\in\N} r_{n}(t) Tf_{n}(y)}_{H}^{2} \dif t\bigg)^{1/2}}_{L^{q}(\dif\nu(y))} && \text{by \eqref{eq:l2-Bernoulli}}\\
&\lesssim
\norm[\bigg]{ \bigg(\int_0^1 \norm[\Big]{\sum_{n\in\N} r_{n}(t) Tf_{n}(y)}_{H}^{q} \dif t\bigg)^{1/q}}_{L^{q}(\dif\nu(y))} && \text{by Lemma~\ref{lem:kahane-khintchine}}\\
&=
\bigg( \int_0^1 \norm[\Big]{T\Big(\sum_{n\in\N}r_{n}(t)f_{n}\Big)}_{L^q(Y;H)}^{q}\dif t \bigg)^{1/q} && \text{by Fubini's theorem}\\
&\lesssim
\bigg( \int_0^1 \norm[\Big]{\sum_{n\in\N}r_{n}(t)f_{n}}_{L^p(X;H)}^{q}\dif t \bigg)^{1/q}
&& \text{by the hypothesis.}
\end{align*}
If $p\geq q$, then, by H{\"o}lder's inequality, Lemma~\ref{lem:kahane-khintchine}, and \eqref{eq:l2-Bernoulli}, this is bounded by
\begin{align*}
\bigg( \int_0^1 \norm[\Big]{\sum_{n\in\N}r_{n}(t)f_{n}}_{L^p(X;H)}^{p}\dif t \bigg)^{1/p}
&\lesssim
\norm[\bigg]{ \bigg(\int_0^1 \norm[\Big]{\sum_{n\in\N} r_{n}(t) f_{n}(x)}_{H}^{2} \dif t\bigg)^{1/2}}_{L^{p}(\dif\mu(x))}\\
&=
\norm[\Big]{ \Big( \sum_{n\in\N} \norm{f_{n}(x)}_{H}^{2} \Big)^{1/2}}_{L^{p}(\dif\mu(x))}.
\end{align*}
If $p<q$, then, by Minkowski's integral inequality for the $L^{q/p}$ norm on the probability space, Lemma~\ref{lem:kahane-khintchine}, and \eqref{eq:l2-Bernoulli}, the above is again bounded by
\begin{align*}
\norm[\bigg]{ \bigg(\int_0^1 \norm[\Big]{\sum_{n\in\N} r_{n}(t) f_{n}(x)}_{H}^{q} \dif t\bigg)^{1/q}}_{L^{p}(\dif\mu(x))}&
\lesssim
\norm[\bigg]{ \bigg(\int_0^1 \norm[\Big]{\sum_{n\in\N} r_{n}(t) f_{n}(x)}_{H}^{2} \dif t\bigg)^{1/2}}_{L^{p}(\dif\mu(x))}\\
&=
\norm[\Big]{ \Big( \sum_{n\in\N} \norm{f_{n}(x)}_{H}^{2} \Big)^{1/2}}_{L^{p}(\dif\mu(x))}.
\qedhere
\end{align*}
\end{proof}
The vector-valued extension of the hypothesis~\eqref{eq:IW-Lp-hypothesis} allows us to eliminate the multiplier $\Theta$ from the picture.
Let $\phi$ be a smooth function supported on $\frac{19}{10} \bfQ$ such that $0\le \phi\le 1$ and identically equal to $1$ on $\frac{11}{10} \bfQ$.
Let $\psi$ be a non-negative smooth function with $\int_{\R^d}\psi(x)\dif x=1$ supported on $\frac{1}{10} \bfQ$.
Let $\zeta=\phi*\psi$.
Note that $\zeta$ is supported on $2 \bfQ$ and identically equal to $1$ on $\bfQ$ and $0\le \zeta\le 1$.
Denote $\zeta_{\epsilon_N}(x) = \zeta(\epsilon_N^{-1}x)$ with $0<\epsilon_N\le e^{-N^{\rho}}$ as in \eqref{eq:IW-mult}.
We fix $M\subseteq\N_{k}$ and $\sigma\in S_{M}$ and expand the corresponding term \eqref{eq:162} as
\begin{equation}
\label{eq:fu-M-sigma-expand}
\sum_{x\in\Z^d}\bigg(\sum_{w\in\calR(S_{M^c})} \norm[\Big]{\FT^{-1} \Big(\sum_{u\in\calR(\sigma)}\sum_{b\in\scrQ(Q_0)}
\Theta(\xi-b-u-w) \hat{f}(\xi)\Big)(x)}_{H}^2\bigg)^r.
\end{equation}
Let $Q := Q_0 \cdot \sigma$ and note that the difference of two fractions of the form $a/Q-w$ has denominator at most $e^{\frac{1}{2}N^{\rho}}$, due to \eqref{eq:9},  thus we can write
\begin{multline*}
\sum_{u\in\calR(\sigma)}\sum_{b\in\scrQ(Q_0)}
\Theta(\xi-b-u-w)\\
=
\Big( \sum_{b\in \scrQ(Q)} \Theta(\xi-b-w)\Big)
\Big( \sum_{u\in\calR(\sigma)}\sum_{b\in\scrQ(Q_0)}
\zeta_{\epsilon_N}(\xi-b-u-w) \Big),
\end{multline*}
because the summands have disjoint supports.
The former multiplier in this product is the shift by $w\in\calR(S_{M^c})$ of the periodic multiplier
\[
\sum_{b\in \scrQ(Q)} \Theta(\xi-b).
\]
This multiplier is bounded on $\ell^{2r}(\Z^{d};H)$ with norm $\lesssim_{r}\boldA_{2r}$ by the hypothesis \eqref{eq:IW-Lp-hypothesis} and \cite[Corollary 2.1]{MR1888798} with $B_1=B_2=H$.
By Theorem~\ref{thm:MZ} with $\calZ=\calR(S_{M^c})$, we obtain
\begin{multline*}
\sum_{x\in\Z^d}\bigg(\sum_{w\in\calR(S_{M^c})} \norm[\Big]{\FT^{-1} \Big(\sum_{u\in\calR(\sigma)}\sum_{b\in\scrQ(Q_0)}
\Theta(\xi-b-u-w) \hat{f}(\xi)\Big)(x)}_{H}^2\bigg)^r\\
\lesssim_{\rho, r}
\bfA_{2r}^{2r}
\sum_{x\in\Z^d}\bigg(\sum_{w\in\calR(S_{M^c})} \norm[\Big]{\FT^{-1} \Big(\sum_{u\in\calR(\sigma)}\sum_{b\in\scrQ(Q_0)}
\zeta_{\epsilon_N}(\xi-b-u-w) \hat{f}(\xi)\Big)(x)}_H^2\bigg)^r.
\end{multline*}
We recall that the sum of this expression over $M\subseteq\N_{k}$ and $\sigma\in S_{M}$ controls the left-hand side of \eqref{eq:677} in view of \eqref{eq:162}.

\subsection{Estimates for the square functions}
Since there are only $2^{k}$ choices of subsets $M\subseteq\N_{k}$, the proof of Theorem~\ref{thm:IW:fu} will be completed if we show the following square function estimate that no longer involves the multiplier $\Theta$.
\begin{lemma}
\label{lem:sq-fct-embed}
Under the assumptions of Theorem~\ref{thm:IW:fu}, for any $M\subseteq\N_{k}$, we have
\begin{align*}
\sum_{\sigma\in S_{M}}
\sum_{x\in\Z^d}\bigg(\sum_{w\in\calR(S_{M^c})} \norm[\Big]{\FT^{-1} \Big(\sum_{u\in\calR(\sigma)}\sum_{b\in\scrQ(Q_0)}
\zeta_{\epsilon_N}(\xi-b-u-w) \hat{f}(\xi)\Big)(x)}_{H}^2\bigg)^r
\lesssim_{r}
\norm{f}_{\ell^{2r}(\Z^{d};H)}^{2r}.
\end{align*}
\end{lemma}
\begin{proof}
The claim can be viewed as the statement that a certain operator maps $\ell^{2r}(\Z^{d};H)$ to $\ell^{2r}(S_{M} \times \Z^{d};\ell^{2}(\calR(S_{M^{c}});H))$.
By interpolation, it suffices to consider $r=1$ and $r=\infty$.
For $r=1$ the claim follows from Plancherel's theorem since the functions $\zeta_{\epsilon_N}(\xi-b-u-w)$ are disjointly supported for all $b\in \scrQ(Q_0)$, $w\in\calR(S_{M^{c}})$, $u\in\calR(\sigma)$, and $\sigma \in S_{M}$.

In the case $r=\infty$ we have to show that
\begin{align*}
\sup_{\sigma\in S_{M}}\sup_{x\in\Z^{d}}\sum_{w\in\calR(S_{M^{c}})}\norm[\Big]{\FT^{-1}\Big(\sum_{u\in\calR(\sigma)}\sum_{b\in\scrQ(Q_{0})}
\zeta_{\epsilon_N}(\xi-b-u-w)\hat{f}(\xi)\Big)(x)}_{H}^2
\lesssim
\norm{f}_{\ell^{\infty}(\Z^{d};H)}^{2}.
\end{align*}
By translation invariance it suffices to consider $x=0$, and by duality it is enough to show
\begin{align}
\label{eq:144}
\norm[\Big]{\FT^{-1}\Big(\sum_{w\in\calR(S_{M^c})}\alpha(w)\sum_{u\in\calR(\sigma)}\sum_{b\in\scrQ(Q_{0})} \zeta_{\epsilon_N}(\xi-b-u-w)\Big)}_{\ell^1(\Z^{d};H)}
\lesssim
1
\end{align}
for any sequence $(\alpha(w): w\in\calR(S_{M^c})) \subseteq H$ such that
\begin{align}
\label{eq:50}
\sum_{w\in\calR(S_{M^c})} \norm{\alpha(w)}_H^2=1.
\end{align}
Recall $\sigma = p_{1}^{\gamma_{1}} \dotsm p_{m}^{\gamma_{m}}$ with distinct prime numbers $p_{1},\dotsc,p_{m}$ that do not divide $Q_{0}$ and $\gamma_{j}\geq 1$.
It follows that
\[
\sum_{u\in\calR(\sigma)}
=
\sum_{\kappa\in\Set{0,1}^{m}} (-1)^{\abs{\kappa}} \sum_{u \in \scrQ(\prod_{j=1}^m p_{j}^{\gamma_{j}-\kappa_{j}})}.
\]
It suffices to obtain an estimate for a fixed $\kappa\in\Set{0, 1}^m$.
By the Chinese remainder theorem, we have to show
\begin{equation*}
\norm[\Big]{\FT^{-1}\Big(\sum_{w\in\calR(S_{M^c})}\alpha(w)\sum_{b\in\scrQ(Q)}\zeta_{\epsilon_N}(\xi-b-w)\Big)}_{\ell^1(\Z^{d};H)}
\le
C_{\rho, d},
\end{equation*}
where $Q=Q_{0} \cdot \prod_{j=1}^m p_{j}^{\gamma_{j}-\kappa_{j}}$, so that in particular $Q\leq Q_{0} N^{D}$ and $Q$ and $S_{M^c}$ are coprime.
In fact, the inequality from the last display is equivalent to the following inequality
\begin{equation}
\label{eq:180}
Q^d\norm[\Big]{\FT^{-1}\Big(\sum_{w\in\calR(S_{M^c})}\alpha(w)\zeta_{\epsilon_N}(\xi-w)\Big)}_{\ell^1(Q\Z^{d};H)}
\le
C_{\rho, d},
\end{equation}
since
\[
\sum_{b\in\scrQ(Q)}e(b\cdot x)=
\begin{cases}
Q^d, & \text{ if } x\equiv 0 \mod Q,\\
0, &\text{ otherwise}.
\end{cases}
\]
Recall $ \zeta_{\epsilon_N}(\xi) = \epsilon_N^{-d} \phi_{\epsilon_N} * \psi_{\epsilon_N}(\xi)$, where $\phi_{\epsilon_N}(\xi)=\phi(\epsilon_N^{-1}\xi)$ and $\psi_{\epsilon_N}(\xi)=\psi(\epsilon_N^{-1}\xi)$. Thus
\[
\sum_{w\in\calR(S_{M^c})}\alpha(w) \zeta_{\epsilon_N}(\xi-w)
=
\epsilon_N^{-d} \psi_{\epsilon_N}*\Big(\sum_{w\in\calR(S_{M^c})}\alpha(w)\phi_{\epsilon_N}(\cdot-w)\Big)(\xi).
\]
We have
\begin{align}
\label{eq:181}
Q^d\epsilon_N^{-d} \norm{\calF^{-1}\psi_{\epsilon_N}}_{\ell^{2}(Q\Z^{d})}
\lesssim
(Q/\epsilon_N)^{d/2},
\end{align}
whereas by Plancherel's theorem and \eqref{eq:50} we get
\begin{align}
\label{eq:184}
\begin{split}
\norm[\Big]{\FT^{-1}\Big(\sum_{w\in\calR(S_{M^c})}\alpha(w)\phi_{\epsilon_N}(\xi-w)\Big)}_{\ell^2(Q\Z^{d};H)}
&=
Q^{-d}\norm[\Big]{\sum_{w\in\calR(S_{M^c})}\alpha(w)\phi_{\epsilon_N Q}(\xi-Q w)}_{L^2(\T^{d};H)}\\
&\lesssim
Q^{-d}\Big((\epsilon_N Q)^{d} \sum_{w\in\calR(S_{M^c})} \norm{\alpha(w)}_{H}^{2} \Big)^{1/2}\\
&\lesssim
(\epsilon_N /Q)^{d/2}.
\end{split}
\end{align}
In the first inequality in \eqref{eq:184}, we have used the
disjointness of the supports of the functions
$\phi_{\epsilon_N Q}(\xi-Q w)$, and only the diagonal terms survived,
since, for any $w_1, w_2\in\calR(S_{M^c})$ with $w_1\neq w_2$, we have
$\abs{w_1- w_2}_{\infty} \geq N^{-2D} \geq 2\epsilon_N$, provided that $N$ is large enough.

The claim \eqref{eq:180} now
follows because the Fourier transform intertwines convolution with the
pointwise product, and by the Cauchy--Schwarz inequality the $\ell^1$
norm in \eqref{eq:180} is controlled by the product of $\ell^2$ norms
in \eqref{eq:181} and \eqref{eq:184}.
\end{proof}


\section{Jump estimates for discrete operators of Radon type: general theory}
\label{sec:discrete}

First we set up notation and terminology, which will be also used in Section~\ref{sec:disrad}. We will apply the results from the previous sections with $d=\abs{\Gamma}$, where
\[
\Gamma=\Set{(\gamma_1, \dotsc, \gamma_k)\in\N_0^k \given 0 < \abs{\gamma} := \abs{\gamma_1}+\dotsb+\abs{\gamma_k}\le d_0}
\]
for some $d_0\in\N$.
We work in the Euclidean space $\R^{\Gamma}$ with coordinates labeled by multi-indexes $\gamma\in\Gamma$, and similarly for $\Z^{\Gamma}$.
Let $I$ be the identity matrix of size $\abs{\Gamma} \times \abs{\Gamma}$ and let $A$ be the diagonal $\abs{\Gamma} \times \abs{\Gamma}$ matrix such that $(Av)_{\gamma} = \abs{\gamma} v_{\gamma}$.
Let
\begin{align*}
\frakq_*(\xi) := \max_{\gamma\in\Gamma}\big(\abs{\xi_{\gamma}}^{\frac{1}{\abs{\gamma}}}\big), \quad
\text{for}\quad \xi\in\R^{\Gamma}
\end{align*}
be the quasi-norm associated with $A^*=A$.

We will be working with a family of convolution operators $(T_t)_{ t\ge0}$
satisfying conditions \ref{MF:A}, \ref{MF:B}, \ref{MF:C}, and \ref{MF:D} 
from Definition~\ref{def:1}. We think of the operator $T_{t}$ as having scale $2^{t}$.  

\begin{definition}
\label{def:1}
For every $t\geq 0$, let $K_{2^t}: \Z^{\Gamma}\to\C$ be an absolutely summable function.
Consider the corresponding convolution operators
\[
T_tf(x)=K_{2^t}*f(x), \quad \text{for} \quad x\in \Z^{\Gamma}
\]
and multipliers
\[
m_t(\xi)=\widehat{K_{2^t}}(\xi), \quad \text{for} \quad \xi\in \T^{\Gamma}.
\]
Suppose that we are given the following objects.
\begin{enumerate}
\item A function $G: \bigcup_{q\in\N}(A_q/q) \to \C$ and numbers $\delta>0$ and $0<C_{\delta}<\infty$ so that for every $q\in \N$ and $a\in A_q$ we have
\begin{align}
\label{eq:G-decay}
\abs{G(a/q)}\le C_{\delta}q^{-\delta}.
\end{align}
\item A family of multipliers $\Phi_t: \R^{\Gamma} \to \C$ indexed by $t\ge0$ such that for every $p\in(1, \infty)$ and every function $f\in L^2(\R^{\Gamma})\cap L^p(\R^{\Gamma})$ we have the jump estimate
\begin{equation}
\label{eq:Phi-jump}
J_{p}^{2}(\FT^{-1}(\Phi_{t}\FT f) : \R^{\Gamma} \times [0,\infty)\to\C)
\lesssim_{p}
\norm{f}_{L^p(\R^{\Gamma})}
\end{equation}
and, for every increasing sequence $0 \le t_{1} < t_{2} < \dotsb$, the square function estimate
\begin{equation}
\label{eq:Phi-square}
\norm[\Big]{ \big( \sum_{j\in\N} \abs{\FT^{-1}((\Phi_{t_{j+1}}-\Phi_{t_{j}})\FT f)}^{2} \big)^{1/2} }_{L^p(\R^{\Gamma})}
\lesssim_p
\norm{f}_{L^p(\R^{\Gamma})},
\end{equation}
and such that the decay condition
\begin{equation}
\label{eq:Phi-decay}
\abs{\Phi_{t_2}(\xi)-\Phi_{t_1}(\xi)} \lesssim \abs{2^{tA}\xi}^{-\epsilon}, 
\end{equation}
holds for every $0<t\le t_1\le t_2\le t+1$ and $\xi\in\R^{\Gamma}\setminus\{0\}$.
\item A number $\chi\in(0, 1)$.
\end{enumerate}
Suppose that the following conditions hold.
\begin{enumerate}[label=(\textbf{\Alph*})]
\item\label{MF:A}
For every $0 < \tau \leq 1$ and every $n\in\N$,
\begin{align}
\label{eq:109}
\norm{V^1(K_{2^{t}}: t \in [n^{\tau}, (n+1)^{\tau}])}_{\ell^1(\Z^{\Gamma})}
\lesssim n^{\tau-1}.
\end{align}
\item\label{MF:B}
For every $\alpha>0$, there exist $0<\beta = \beta(\alpha) < \infty$ and $0<C_{\alpha}<\infty$ such that, for every $N\in\N$, every multi-index $\gamma_0\in\Gamma$, every integers $a, q$ such that $0\le a < q$, $(a, q) = 1$, and
\begin{align}
\label{eq:113}
N^\beta \leq q \leq 2^{N \abs{\gamma_0}} N^{-\beta},
\end{align}
and every $\xi\in\T^{\Gamma}$ with
\[
\abs[\big]{\xi_{\gamma_0} - \frac{a}{q}}
\leq
\frac{1}{q^2},
\]
we have
\begin{align}
\label{eq:114}
\sup_{N \leq t_1, t_2 \leq N+1}|m_{t_1}(\xi)-m_{t_2}(\xi)|
\leq C_{\alpha}
N^{-\alpha}.
\end{align}
\item\label{MF:C}
For every $\alpha,\beta>0$ there exists $0<C_{\alpha,\beta} < \infty$ such that, for every $N\in\N$, every $1 \leq q < (N+1)^{\beta}$, every $a\in A_{q}$, and every $\xi\in\T^{\Gamma}$ with $\abs{\xi_{\gamma}-a_{\gamma}/q} \leq 2^{-N\abs{\gamma}+N^{\chi}}$ for all $\gamma\in\Gamma$, we have
\begin{align}
\label{eq:110}
\begin{split}
\sup_{N \leq t_1, t_2 \leq N+1}\abs[\big]{m_{t_1}(\xi)-m_{t_2}(\xi)&-G(a/q)\big(\Phi_{t_1}(\xi-a/q)-\Phi_{t_2}(\xi-a/q)\big)}\\
& \leq C_{\alpha,\beta}
N^{-\alpha}.
\end{split}
\end{align}
\item\label{MF:D}
There is a family of multipliers $\tilde{m}_{N}$, indexed by $N\in\N$, that are uniformly bounded on $\ell^{p}(\Z^{\Gamma})$ for all $p\in(1, \infty)$.
Moreover, for every $\alpha>0$ there exists $0<C_{\alpha} < \infty$ such that, for every $N\in\N$, every $1 \leq q \leq e^{N^{\chi/5}}$, every $a \in A_{q}$, and every $\xi\in\T^{\Gamma}$ with $\abs{\xi_{\gamma} - a_{\gamma}/q} \leq 2^{-N\abs{\gamma} + N^{\chi}}$ for all $\gamma\in\Gamma$, we have
\begin{equation}
\label{eq:104}
\abs{\tilde{m}_{N}(\xi)-G(a/q)}
\leq C_{\alpha}
N^{-\alpha}.
\end{equation}
\end{enumerate}
\end{definition}

\begin{theorem}
\label{thm:4}
Suppose that $(T_t)_{ t\ge0}$ is a family of convolution operators satisfying conditions \ref{MF:A}, \ref{MF:B}, \ref{MF:C}, and \ref{MF:D} from Definition~\ref{def:1}.
Then, for every $p\in(1, \infty)$, there is $0<C_{p}<\infty$ such that, for every $f\in \ell^p(\Z^{\Gamma})$, we have
\begin{equation}
\label{eq:95}
J^{p}_{2}((T_tf)_{t\ge0} : \Z^{\Gamma} \to\C)
\le C_{p}
\norm{f}_{\ell^p(\Z^{\Gamma})}.
\end{equation}
In particular, for every $p\in(1, \infty)$ and $r\in(2, \infty]$, there is $0<C_{p, r}<\infty$ such that
\begin{align}
\label{eq:118}
\norm{\sup_{t\ge0}\abs{T_tf}}_{\ell^p(\Z^{\Gamma})}
&
\le C_{p, \infty}\norm{f}_{\ell^p(\Z^{\Gamma})},\\
\label{eq:119}
\norm{V^{r}(T_tf: t\ge0)}_{\ell^p(\Z^{\Gamma})}
&\le C_{p, r}
\norm{f}_{\ell^p(\Z^{\Gamma})},
\end{align}
for every $f\in \ell^p(\Z^{\Gamma})$.
\end{theorem}
Inequality \eqref{eq:95} implies inequalities \eqref{eq:118} and
\eqref{eq:119} by appealing to \cite[Lemma 2.12]{arxiv:1808.04592} and the Marcinkiewicz interpolation theorem.
From now on, for every $p\in[1, \infty]$, we shall abbreviate $\|\cdot\|_{L^p(\R^{\Gamma})}=:\|\cdot\|_{L^p}$, $\|\cdot\|_{\ell^p(\Z^{\Gamma})} =: \|\cdot\|_{\ell^p}$, and $\|\cdot\|_{L^p(\T^{\Gamma})} =: \|\cdot\|_{p}$.

By the monotone convergence theorem and standard density arguments, the estimate \eqref{eq:95} will follow if we can show
\begin{equation}
\label{eq:67}
J^{p}_{2}(T_tf : \Z^{\Gamma}\times\bbI \to\C)
\le C_{p}
\norm{f}_{\ell^p}
\end{equation}
for every finite subset $\bbI\subset [0, \infty)$ with a constant $C_{p}$ that does not depend on $\bbI$.

Fix $p\in(1, \infty)$ and chose $p_0>1$, close to $1$, such that $p\in(p_0, p_0')$.
Take $\tau\in(0, 1)$ such that
\begin{align}
\label{eq:74}
\tau<\frac{1}{2}\min\{p_0-1, 1\}.
\end{align}
Using \cite[Lemma 1.3]{MR2434308}, we split \eqref{eq:67} into a long and short $\lambda$-jumps using respectively
\begin{align}
J^{p}_{2} (T_tf : \Z^{\Gamma}\times\bbI \to\C)
&\lesssim
J^{p}_{2} (T_{n^{\tau}}f : \Z^{\Gamma} \times \N_{0}\to\C) \label{eq:long-jumps}\\
&+
\norm[\Big]{ \Big( \sum_{n\in\N_0} V^2\big( T_{t}f : t\in[n^{\tau}, (n+1)^{\tau}]\cap\bbI \big)^2 \Big)^{1/2}}_{\ell^p}. \label{eq:short-jumps}
\end{align}

\subsection{Estimate for short variations}
We repeat the argument from \cite{arxiv:1403.4085} to estimate the short variations in \eqref{eq:short-jumps}.
For each $n\in\N$, let $s_{n,0}<s_{n,1}<\dotsc<s_{n,J(n)}$ be the increasing enumeration of $[n^\tau, (n+1)^\tau]\cap\bbI$.
It follows that
\begin{align}
\label{eq:21}
\MoveEqLeft
\norm[\Big]{ \Big( \sum_{n\in\N_0} V^2\big(K_{2^t}*f: t\in[n^{\tau}, (n+1)^{\tau}]\cap\mathbb I\big)^2 \Big)^{1/2}}_{\ell^p}\\
\nonumber & \leq
\norm[\Big]{ \Big( \sum_{n\in\N_0} V^1\big(K_{2^t}*f: t\in[n^\tau, (n+1)^\tau]\big)^2 \Big)^{1/2}}_{\ell^p}\\
\nonumber & \le
\norm[\bigg]{ \bigg( \sum_{n\in\N_0} \Big(\sum_{j=1}^{J(n)} \abs{(K_{2^{s_{n, j}}}-K_{2^{s_{n, j-1}}})*f} \Big)^q \bigg)^{1/q}}_{\ell^p} &&\text{with $q=\min\{p, 2\}$}\\
\nonumber & \le
\bigg( \sum_{n\in\N_0} \Big(\sum_{j=1}^{J(n)}\norm{(K_{2^{s_{n, j}}}-K_{2^{s_{n, j-1}}})*f}_{\ell^p}\Big)^q \bigg)^{1/q} &&\text{by Minkowski's inequality}\\
\nonumber & \le
\bigg( \sum_{n\in\N_0} \Big(\sum_{j=1}^{J(n)}\norm{K_{2^{s_{n, j}}}-K_{2^{s_{n, j-1}}}}_{\ell^1}\Big)^q \bigg)^{1/q} \norm{f}_{\ell^p} &&\text{by Young convolution inequality}\\
\nonumber & \lesssim
\Big( \sum_{n\in\N_0} n^{-q(1-\tau)} \Big)^{1/q}\norm{f}_{\ell^p}
&& \text{by \eqref{eq:109}}\\
\nonumber & \lesssim
\norm{f}_{\ell^p},
\end{align}
since $q(1-\tau)>1$ by \eqref{eq:74}.
This finishes the estimate for \eqref{eq:short-jumps}.

\subsection{Major and minor arcs for long jumps}
\label{sec:major-minor-arcs}
The decomposition of the long jumps \eqref{eq:long-jumps} will involve several parameters that are chosen depending on $p\in(1, \infty)$, $\tau\in(0, 1)$ as in \eqref{eq:74}, the function $\alpha \mapsto \beta(\alpha)$ from condition~\ref{MF:B}, and $\delta>0$ from \eqref{eq:G-decay}.
We specify our choices in advance in order to make sure that all conditions that we use are compatible.

We fix $p_{0}\in(1, \infty)$ such that $p\in(p_0, p_0')$ and choose
\begin{align}
\label{eq:choice:alpha}
\alpha > \biggl( \frac1{p_{0}} - \frac12 \biggr) \cdot \biggl( \frac1{p_{0}} - \frac{1}{\min(p,p')} \biggr)^{-1},
\end{align}
so that the estimates of the form
\[
\norm{T}_{\ell^{2} \to \ell^{2}} \lesssim j^{-\alpha},
\quad
\norm{T}_{\ell^{p_{0}} \to \ell^{p_{0}}} \lesssim \log(j+2),
\quad
\norm{T}_{\ell^{p_{0}'} \to \ell^{p_{0}'}} \lesssim \log(j+2),
\]
for a linear operator $T$, can be  interpolated to obtain  $\norm{T}_{\ell^{p} \to \ell^{p}} \lesssim j^{-\tilde{\alpha}}$ with some $\tilde{\alpha} > 1$.
We also choose an integer
\begin{align}
\label{eq:choice:u}
u > \beta(\alpha/\tau) \card{\Gamma}.
\end{align}
Finally, we will use Ionescu--Wainger multipliers as constructed  in Section~\ref{sec:iw} with the parameter
\begin{align}
\label{eq:choice:rho}
\rho := \min \Bigl( \frac{\chi}{10 u}, \frac{\delta}{6 \alpha} \Bigr).
\end{align}

Let $\phi : \R\to [0,1]$ be a smooth function such that
\[
\phi(x) :=
\begin{cases}
1 & \text{ if } \abs{x} \leq 1/8,\\
0 & \text{ if } \abs{x} \geq 1/4.
\end{cases}
\]
We will use the cutoff functions $\tilde\eta(x) := \prod_{\gamma\in\Gamma} \phi(x_{\gamma})$ and $\eta(x):=\tilde\eta(2x)$, $x\in\R^{\Gamma}$.
Note for future reference that $\eta=\tilde\eta\eta$.
For $N\in (0,\infty)$, let
\[
\eta_{N}(\xi)
:=
\eta(2^{N A- N^{\chi} I}\xi)
=
\prod_{\gamma\in\Gamma} \phi(2 \cdot 2^{\abs{\gamma} N - N^{\chi}} \xi_{\gamma}).
\]

Recall the family of rational fractions $\scrU_{S}$ defined in \eqref{eq:156}.
For dyadic integers $S \in 2^{u\N}$, we define
\[
\dscrU_{S} :=
\begin{cases}
\scrU_{S}, & S = 2^{u},\\
\scrU_{S} \setminus \scrU_{S/2^{u}}, & S > 2^{u}.
\end{cases}
\]
We will use the convention that $S\in 2^{u\N}$ whenever it appears as a summation index.

Similarly to \cite{arXiv:1512.07518, MR3681393, MR3738256}, we shall exploit, for every $n\in\N$, the partition of unity
\begin{align}
\label{eq:76}
\one_{\T^{\Gamma}}(\xi) = (\one_{\T^{\Gamma}}(\xi)-\Xi_{n^{\tau}}(\xi))+\sum_{S \leq n^{\tau u}}\Xi_{n^{\tau}}^S(\xi),
\end{align}
where
\[
\Xi_{n^{\tau}}^S(\xi)
:=
\sum_{a/q\in\dscrU_{S}}\eta_{n^{\tau}}(\xi-a/q)
\]
and
\[
\Xi_{n^{\tau}}(\xi)
:=
\sum_{S \leq n^{\tau u}} \Xi_{n^{\tau}}^{S}(\xi)
=
\sum_{a/q \in \scrU_{\tilde{S}}} \eta_{n^{\tau}}(\xi-a/q),
\quad \text{with}\quad
\tilde{S} = \max (2^{u\N} \cap [1,n^{\tau u}]).
\]
Theorem~\ref{thm:IW-mult} and complex interpolation ensure that, for every $\tilde{p}\in(1, \infty)$ and every $f\in\ell^{\tilde{p}}(\Z^{\Gamma})$, we have
\begin{align}
\label{eq:22}
\begin{split}
\norm[\big]{ \FT^{-1}\big(\Xi_{n^{\tau}}\hat{f}\big) }_{\ell^{\tilde{p}}}
&\lesssim \log (n+2) \norm{f}_{\ell^{\tilde{p}}},\\
\norm[\big]{ \FT^{-1}\big(\Xi_{n^{\tau}}^S\hat{f}\big) }_{\ell^{\tilde{p}}}
&\lesssim \log (S+2) \norm{f}_{\ell^{\tilde{p}}}.
\end{split}
\end{align}
This is due to the small supports in the definition of $\Xi_{n^{\tau}}$ and $\Xi_{n^{\tau}}^S$, since for every $\gamma\in\Gamma$ we have $2^{-n^{\tau} \abs{\gamma}+n^{\tau \chi}} \leq e^{-n^{2\rho \tau u}}$ by \eqref{eq:choice:rho} for sufficiently large $n\in\N$.

Using \eqref{eq:76}, we obtain
\begin{align}
\eqref{eq:long-jumps}
&=
J^{p}_{2} \Big(\sum_{0\le j<n}\FT^{-1}((m_{(j+1)^\tau}-m_{j^\tau})\hat{f} : n\in\N_0\Big) \nonumber\\
&\lesssim J^{p}_{2}\Big(\sum_{0\le j<n}\FT^{-1}((m_{(j+1)^\tau}-m_{j^\tau})(1-\Xi_{j^{\tau}})\hat{f}): n\in\N_0\Big)\label{eq:77}\\
&+ \sum_{S\in 2^{u\N}} J^{p}_{2}\Big( \sum_{0\le j<n : S \leq j^{\tau u}} \FT^{-1}((m_{(j+1)^\tau}-m_{j^\tau})\Xi_{j^{\tau}}^S\hat{f}): n\in\N_0\Big) \label{eq:207}.
\end{align}
In the last line we have used the fact that the jump quasi-seminorm \eqref{eq:jump-space} admits an equivalent subadditive norm \cite[Corollary 2.11]{arxiv:1808.04592}.
\subsection{Minor arcs}
In order to estimate \eqref{eq:77}, we will appeal to the inequality
\begin{align}
\label{eq:276}
J^{p}_{2}((F(\cdot, n))_{n\in\N}) \leq \norm{V^1(F(\cdot, n): n\in\N)}_{L^{p}(X)} \leq \sum_{n\in\N}\norm{F(\cdot, n+1)-F(\cdot, n)}_{L^{p}(X)}.
\end{align}
Thus,
\begin{equation}
\label{eq:225}
\eqref{eq:77}
\le
\sum_{j \ge0} \norm[\big]{\FT^{-1}((m_{(j+1)^\tau}-m_{j^\tau})(1-\Xi_{j^{\tau}})\hat{f})}_{\ell^p}.
\end{equation}
The next lemma will suffice to handle this series.
Using Lemma~\ref{lem:10} and \eqref{eq:225}, we finish the estimate for \eqref{eq:77}.

\begin{lemma}
\label{lem:10}
For every $j\in\N$ and $f\in\ell^p(\Z^\Gamma)$, we have
\begin{align}
\label{eq:3}
\norm[\big]{\FT^{-1}((m_{(j+1)^\tau}-m_{j^\tau})(1-\Xi_{j^{\tau}})\hat{f})}_{\ell^p}
\lesssim
j^{-\tilde{\alpha}}\norm{f}_{\ell^p}
\end{align}
with some $\tilde{\alpha} > 1$.
\end{lemma}
\begin{proof}
It suffices to consider large values of $j$ (depending on all other parameters).
In view of \eqref{eq:109} and \eqref{eq:22}, for $q\in\Set{p_{0},p_{0}'}$, we have
\begin{align}
\label{eq:14}
\norm[\big]{\FT^{-1}((m_{(j+1)^\tau}-m_{j^\tau})(1-\Xi_{j^{\tau}})\hat{f})}_{\ell^{q}}
\lesssim
\log (j+2) \norm{f}_{\ell^{q}}.
\end{align}
By Plancherel's theorem and complex interpolation, it suffices to show
\begin{equation}
\label{eq:minor-arc-mult}
\norm{(m_{(j+1)^\tau}-m_{j^\tau})(1-\Xi_{j^{\tau}})}_{\infty} \lesssim j^{-\alpha}.
\end{equation}
Let $\beta=\beta(\alpha/\tau)$ be as in~\ref{MF:B}, $\tilde{S} := \max (2^{u\N} \cap [1,j^{\tau u}])$, and $N := \floor{j^{\tau}}$.

Let $\xi \in \T^{\Gamma}$.
By Dirichlet's principle, for every $\gamma\in\Gamma$ there exist coprime natural numbers $a_{\gamma}, q_{\gamma}$ such that $1\le q_{\gamma}\le N^{-\beta}2^{N\abs{\gamma}}$ and
\[
\abs{\xi_{\gamma}-a_{\gamma}/q_{\gamma}}
\leq
q_{\gamma}^{-1} N^{\beta}2^{-N\abs{\gamma}}
\leq
q_{\gamma}^{-2}.
\]
We distinguish two cases.
\paragraph{Case 1}
Suppose that $1\le q_{\gamma }< N^{\beta}$ for every $\gamma \in \Gamma$.
Then $q := \lcm(q_{\gamma}: \gamma\in\Gamma) \leq N^{\beta \card{\Gamma}} \leq \tilde{S}$, since $\beta \card{\Gamma} < u$ by the choice of $u$ in \eqref{eq:choice:u}, and $j$ is sufficiently large.
Hence, for some $a' \in A_{q}$, we have $(a'/q) = (a_{\gamma}/q_{\gamma})_{\gamma} \in \scrU_{\tilde{S}}$.
On the other hand,
\[
\abs{\xi_{\gamma}-a_{\gamma}'/q}
\leq
N^{\beta}2^{-N \abs{\gamma}}
\leq
2^{-N \abs{\gamma}+N^{\chi}}/16.
\]
It follows that $\Xi_{j^{\tau}}(\xi)=1$, so that the multiplier \eqref{eq:minor-arc-mult} vanishes for this value of $\xi$.

\paragraph{Case 2}
If the previous case did not occur, then, in fact, for some $\gamma\in\Gamma$ the condition \eqref{eq:113} holds, that is, $N^{\beta}\le q_{\gamma }\le N^{-\beta}2^{N \abs{\gamma}}$.
Therefore, \eqref{eq:114} applies, and we obtain
\[
\abs{(m_{(j+1)^\tau}-m_{j^\tau})(\xi)}
\lesssim
N^{-\alpha/\tau}
\simeq
j^{-\alpha}.
\]
This finishes the estimate for the multiplier \eqref{eq:minor-arc-mult} at the point $\xi$.
\end{proof}

\subsection{Major arcs}
It remains to estimate the series \eqref{eq:207}.
We will consider each summand
\begin{align}
\label{eq:264}
J^{p}_{2} \Big( \sum_{0\le j<n : S \leq j^{\tau u}} \FT^{-1}((m_{(j+1)^\tau}-m_{j^\tau})\Xi_{j^{\tau}}^S \hat{f}): n\in\N_0 \Big)
\end{align}
separately and provide estimates that are summable in $S \in 2^{u\N}$.

For this purpose we split the jump norm in \eqref{eq:264} at scale
\begin{equation}
\label{eq:kappa}
\kappa_S := 2^{\ceil{S^{2 \rho}} + C}
\end{equation}
with a large integer $C$, that is, we estimate
\begin{align}
\eqref{eq:264}
&\leq
J^{p}_{2} \Big( \sum_{S^{1/(\tau u)} \leq j<n}\FT^{-1}((m_{(j+1)^\tau}-m_{j^\tau})\Xi_{j^{\tau}}^S\hat{f}): n^{\tau} \le 2\kappa_S \Big)
\label{eq:265}\\
&+
J^{p}_{2} \Big( \sum_{\kappa_S^{1/\tau}\le j<n}\FT^{-1}((m_{(j+1)^\tau}-m_{j^\tau})\Xi_{j^{\tau}}^S\hat{f}): n^{\tau} \geq \kappa_S \Big).
\label{eq:266}
\end{align}

We begin with the definition of approximating multipliers for the respective scales.
Let
\[
\nu_j^S(\xi)
:=
\sum_{a/q\in\dscrU_{S}}G(a/q)\big(\Phi_{(j+1)^\tau}(\xi-a/q)-\Phi_{j^\tau}(\xi-a/q)\big)\eta_{j^{\tau}}(\xi-a/q).
\]
\begin{align}
\label{eq:100}
\Lambda_{j}^S(\xi) :=
\sum_{a/q\in\dscrU_{S}}\big(\Phi_{(j+1)^\tau}(\xi-a/q)-\Phi_{j^\tau}(\xi-a/q)\big)\eta_{j^{\tau}}(\xi-a/q).
\end{align}

\begin{lemma}
\label{lem:mXi-nu}
For every $S \leq j^{\tau u}$ and $N \leq j^{\tau} \leq 4N$, we have
\begin{align}
\label{eq:mXi-nu}
\norm{(m_{(j+1)^\tau}-m_{j^\tau})\Xi_{j^{\tau}}^S-\nu_j^S}_{\infty}
&\lesssim
(j+1)^{-\alpha}\\
\label{eq:mXi-tildem}
\norm{(m_{(j+1)^\tau}-m_{j^\tau})\Xi_{j^{\tau}}^S-\Lambda_{j}^S\tilde{m}_{N}}_{\infty}
&\lesssim
(j+1)^{-\alpha}.
\end{align}
\end{lemma}
\begin{proof}
It suffices to consider large $j$.
Let $a/q \in \dscrU_{S}$ be a reduced fraction, so that in particular $q \leq e^{S^{\rho}} \leq e^{j^{\tau u \rho}}$.
Let $\xi\in\T^{\Gamma}$ be such that $\eta_{j^{\tau}}(\xi-a/q)\neq 0$, then
\[
\abs{\xi_{\gamma}-a_{\gamma}/q}
\leq
2^{-j^{\tau}\abs{\gamma}+ j^{\chi\tau}}
\]
for every $\gamma\in \Gamma$.
Let $\beta = \max(\beta(\alpha/\tau),\alpha/(\delta\tau))$.
There are now two cases.
If $q < (j^{\tau})^{\beta}$, then by \eqref{eq:110} we obtain
\begin{align*}
\MoveEqLeft
\abs{((m_{(j+1)^\tau}-m_{j^\tau})\Xi_j^S-\nu_j^S)(\xi)}\\
&\lesssim
\abs{(m_{(j+1)^\tau}-m_{j^\tau})(\xi)-G(a/q)(\Phi_{(j+1)^\tau}-\Phi_{j^\tau})(\xi-a/q)}\\
&\lesssim
(j^{\tau})^{-\alpha/\tau} = j^{-\alpha}.
\end{align*}
If $q \geq (j^{\tau})^{\beta}$, then the condition \eqref{eq:113} holds for large $j$, so by \eqref{eq:114} and \eqref{eq:G-decay} we obtain
\begin{align*}
\MoveEqLeft
\abs{((m_{(j+1)^\tau}-m_{j^\tau})\Xi_j^S-\nu_j^S)(\xi)}\\
&\lesssim
\abs{(m_{(j+1)^\tau}-m_{j^\tau})(\xi)} + \abs{G(a/q)}\\
&\lesssim
(j^{\tau})^{-\alpha/\tau} + (j^{\tau\beta})^{-\delta}
\lesssim
j^{-\alpha}.
\end{align*}
This finishes the proof of \eqref{eq:mXi-nu}.
On the other hand,
\begin{align*}
\abs{(\nu_{j}^{S}-\Lambda_{j}^S\tilde{m}_{N})(\xi)}
&=
\abs{G(a/q) - \tilde{m}_{N}(\xi)} \abs{(\Phi_{(j+1)^\tau}-\Phi_{j^\tau})(\xi-a/q)} \abs{\eta_{j^{\tau}}(\xi-a/q)}\\
&\lesssim
\abs{G(a/q) - \tilde{m}_{N}(\xi)}\\
&\lesssim_{\alpha}
N^{-\alpha/\tau}
\lesssim
j^{-\alpha}
\end{align*}
by \eqref{eq:104}, since $2^{-j^{\tau} \abs{\gamma} + j^{\chi\tau}} \leq 2^{-N \abs{\gamma} + N^{\chi}}$ and $q \leq e^{N^{\chi/5}}$.
\end{proof}

\subsection{Small scales: estimate for \eqref{eq:265}}
We follow and simplify the ideas that originate in \cite{MR2188130} and in \cite{arXiv:1512.07518,MR3681393}.
The important refinement that the multiplier $\tilde{m}_{N}$ in \eqref{eq:2} and \eqref{eq:4} is not of the form \eqref{eq:Pi_s} has been independently found by Trojan \cite{arxiv:1803.05406}.

Using \cite[Lemma 2.5]{arxiv:1808.09048}, we obtain
\begin{align}
\eqref{eq:265}
&\leq
\norm[\Big]{V^{2} \Big(\sum_{S^{1/(\tau u)}\leq j<n}\FT^{-1}((m_{(j+1)^\tau}-m_{j^\tau})\Xi_{j^{\tau}}^S\hat{f}): n^{\tau}\leq 2\kappa_S \Big)}_{\ell^p} \notag\\
&\leq
\sum_{N \in 2^{\N} \cap [S^{1/u}, \kappa_{S})} \norm[\Big]{V^{2} \Big(\sum_{S^{1/(\tau u)}\leq j<n}\FT^{-1}((m_{(j+1)^\tau}-m_{j^\tau})\Xi_{j^{\tau}}^S\hat{f}): N \leq n^{\tau} \leq 4N \Big)}_{\ell^p} \notag\\
&\leq
\sum_{N \in 2^{\N} \cap [S^{1/u}, \kappa_{S})} \norm[\Big]{V^{2} \Big( \sum_{\mathrlap{S^{1/(\tau u)}\leq j<n}}\FT^{-1}(((m_{(j+1)^\tau}-m_{j^\tau})\Xi_{j^{\tau}}^S-\Lambda_{j}^{S}\tilde{m}_{N})\hat{f}): N \leq n^{\tau} \leq 4N \Big)}_{\ell^p} \notag\\
&+
\sum_{N \in 2^{\N} \cap [S^{1/u}, \kappa_{S})} \norm[\Big]{V^{2} \Big(\sum_{S^{1/(\tau u)}\leq j<n}\FT^{-1}(\Lambda_{j}^{S}\tilde{m}_{N}\hat{f}): N \leq n^{\tau} \leq 4N \Big)}_{\ell^p} \notag\\
&\lesssim
\sum_{N \in 2^{\N} \cap [S^{1/u}, \kappa_{S})} \sum_{j : N < j^{\tau} \leq 4N} \norm[\Big]{\FT^{-1}(((m_{(j+1)^\tau}-m_{j^\tau})\Xi_{j^{\tau}}^S-\Lambda_{j}^{S}\tilde{m}_{N})\hat{f})}_{\ell^p} \label{eq:2}\\
&+
\sum_{N \in 2^{\N} \cap [S^{1/u}, \kappa_{S})} \sum_{i \lesssim \log N} \norm[\Big]{\Big(\sum_{l} \abs[\big]{\sum_{j \in I^{i}_{l}}\FT^{-1}(\Lambda_{j}^{S}\tilde{m}_{N}\hat{f})}^{2} \Big)^{1/2}}_{\ell^p}, \label{eq:4}
\end{align}
where the summation in \eqref{eq:4} is taken over all  $l\ge0$ such
that the sets $I_{l}^i \subset [N^{1/\tau},(4N)^{1/\tau}] \cap \N$,
and by  \cite[Lemma 2.5]{arxiv:1808.09048}  we know that the sets $I_{l}^i$  are pairwise disjoint intervals of length at most $2^{i}$.

\subsubsection{Error terms}
We handle \eqref{eq:2} by the following lemma.
\begin{lemma}
\label{lem:14}
If $S < j^{\tau u}$ and $N \leq j^{\tau} \leq 4N$, then
\begin{align}
\label{eq:274}
\norm[\big]{\FT^{-1}\big(((m_{(j+1)^\tau}-m_{j^\tau})\Xi_{j^{\tau}}^{S} - \Lambda_{j}^{S} \tilde{m}_{N})\hat{f}\big)}_{\ell^p(\Z^{\Gamma})}
\lesssim
(j+1)^{-\tilde{\alpha}}\norm{f}_{\ell^p(\Z^{\Gamma})}
\end{align}
with some $\tilde{\alpha} > 1$.
\end{lemma}
\begin{proof}
For any $p_0\in(1, \infty)$, and in particular for the $p_{0}$ chosen above \eqref{eq:choice:alpha}, we have
\begin{align}
\label{eq:284}
\norm[\big]{\FT^{-1}\big(((m_{(j+1)^\tau}-m_{j^\tau})\Xi_{j}^S-\Lambda_{j}^S\tilde{m}_{N})\hat{f}\big)}_{\ell^{p_0}}
\lesssim
\log (S+1) \norm{f}_{\ell^{p_0}},
\end{align}
since by Theorem~\ref{thm:IW-mult}
\begin{align*}
\norm[\big]{\FT^{-1}\big((m_{(j+1)^\tau}-m_{j^\tau})\Xi_{j^{\tau}}^S\hat{f}\big)}_{\ell^{p_0}}
&\lesssim
\norm[\big]{\FT^{-1}\big(\Xi_{j^{\tau}}^S\hat{f}\big)}_{\ell^{p_0}}
\lesssim
\log (S+1) \norm{f}_{\ell^{p_0}},\\
\norm[\big]{\FT^{-1}\big(\Lambda_{j}^S\tilde{m}_{N}\hat{f}\big)}_{\ell^{p_0}}
&\lesssim
\norm[\big]{\FT^{-1}\big(\Lambda_{j}^S\hat{f}\big)}_{\ell^{p_0}}
\lesssim
\log (S+1) \norm{f}_{\ell^{p_0}}.
\end{align*}
Interpolation with the $\ell^{2}$ estimate coming from \eqref{eq:mXi-tildem} and Plancherel's theorem finishes the proof.
\end{proof}
It follows from Lemma~\ref{lem:14} that
\[
\eqref{eq:2}
\lesssim
\sum_{N \in 2^{\N} : N \geq S^{1/u}} \sum_{j : N < j^{\tau} \leq 4N} (j+1)^{-\tilde{\alpha}} \norm{f}_{\ell^p}
\lesssim
S^{(1-\tilde{\alpha})/(\tau u)} \norm{f}_{\ell^p},
\]
and this is summable in $S$, since the exponent is strictly negative.

\subsubsection{Square functions}
We now  estimate \eqref{eq:4}.
To this end it suffices to prove the following.
\begin{lemma}
\label{lem:dis-sq-fct}
\[
\norm[\Big]{\Big(\sum_{l} \abs[\big]{\sum_{j \in I^{i}_{l}}\FT^{-1}(\Lambda_{j}^{S}\tilde{m}_{N}\hat{f})}^{2} \Big)^{1/2}}_{\ell^p}
\lesssim
S^{-5\rho} \norm{f}_{\ell^{p}}.
\]
\end{lemma}
Indeed, assuming Lemma~\ref{lem:dis-sq-fct} we obtain
\[
\eqref{eq:4}
\lesssim
\sum_{N \in 2^{\N} \cap [S^{1/u}, \kappa_{S})} \sum_{i \lesssim \log N} S^{-5\rho} \norm{f}_{\ell^p}
\lesssim
(\log \kappa_{S})^{2} S^{-5\rho} \norm{f}_{\ell^p},
\]
and this is summable in $S \in 2^{u\N}$ with the choice of $\kappa_{S}$ in \eqref{eq:kappa}.

Let
\[
\tilde{\Xi}_{N}^S(\xi)
:=
\sum_{a/q\in\dscrU_{S}}\tilde{\eta}_N(\xi-a/q).
\]

\begin{lemma}
\label{lem:tildemtildeXi}
For $S^{1/u} \leq N$, we have
\[
\norm{\tilde{m}_{N}\tilde{\Xi}_{N}^S}_{\infty}
\lesssim
S^{-\delta}.
\]
\end{lemma}
\begin{proof}
Since by the hypothesis the multipliers $\tilde{m}_{N}$ are uniformly bounded on $L^{2}$, we clearly have a uniform estimate in $N$ for each fixed $S$.
Hence we may assume that $S$ is so large that the functions $\tilde{\eta}_{N}(\cdot-a/q)$ have disjoint support for $a/q\in\dscrU_{S}$.

Let $\xi\in\T^{\Gamma}$ be such that $\tilde{\Xi}_{N}^{S}(\xi) \neq 0$.
Then there exists $a/q \in \dscrU_{S}$ with
\[
\abs{\xi_{\gamma} - {a_{\gamma}}/{q}}
\leq
2^{-N\abs{\gamma}+N^{\chi}}
\]
for every $\gamma\in\Gamma$.
In particular, $S/2^{u} < q \leq e^{S^{\rho}} \leq e^{N^{\chi/10}}$.
We estimate
\[
\abs{(\tilde{m}_{N}\tilde{\Xi}_{N}^S)(\xi)}
\leq
\abs{\tilde{m}_{N}(\xi)-G(a/q)}
+
\abs{G(a/q)}.
\]
Using \eqref{eq:G-decay}, we estimate the second term by $|G(a/q)|\lesssim q^{-\delta}\lesssim S^{-\delta}$.
Using \eqref{eq:104} with $\alpha = \delta  u$, we also estimate the first term by $S^{-\delta}$.
\end{proof}

\begin{proof}[Proof of Lemma~\ref{lem:dis-sq-fct}]
Since $\Lambda_{j}^S\tilde{m}_{N}=\Lambda_{j}^S\tilde{m}_{N}\tilde{\Xi}_{N}^S$, it suffices to show that
\begin{align}
\label{eq:269}
\norm[\Big]{\Big(\sum_{l} \abs[\big]{\sum_{j \in I^{i}_{l}}\FT^{-1}(\Lambda_{j}^{S}\hat{f})}^{2} \Big)^{1/2}}_{\ell^p}
\lesssim
\log (S) \norm{f}_{\ell^p},
\end{align}
and
\begin{align}
\label{eq:270}
\norm[\big]{\FT^{-1}(\tilde{m}_{N}\tilde{\Xi}_{N}^S\hat{f})}_{\ell^p}
\lesssim
S^{-6\rho} \norm{f}_{\ell^p},
\end{align}
where the implicit constants are independent of $S$ and $N$.

By Theorem~\ref{thm:IW-mult}, the estimate \eqref{eq:269} is a consequence of its continuous counterpart
\begin{align*}
\norm[\bigg]{ \bigg(\sum_{l}\abs[\Big]{\sum_{j\in I_l^i}\FT^{-1}((\Phi_{(j+1)^\tau}-\Phi_{j^\tau})\eta_{j^{\tau}}\FT{f})}^2\Bigg)^{1/2}}_{L^p}
\lesssim \norm{f}_{L^p}.
\end{align*}
Indeed, by the hypothesis \eqref{eq:Phi-square}, we have the square function estimate
\begin{align*}
\norm[\bigg]{ \bigg(\sum_{l}\abs[\Big]{\sum_{j\in I_l^i}\FT^{-1}((\Phi_{(j+1)^\tau}-\Phi_{j^\tau})\FT{f})}^2\Bigg)^{1/2}}_{L^p}
\lesssim \norm{f}_{L^p},
\end{align*}
whereas the error term can be handled by the inequality
\begin{equation}
\label{eq:trim-cont-mult}
\sum_{j\ge0}\norm[\big]{\FT^{-1}((\Phi_{(j+1)^\tau}-\Phi_{j^\tau})(1-\eta_{j^{\tau}})\FT{f})}_{L^p}\lesssim \norm{f}_{L^p}
\end{equation}
that holds for every $p\in (1,\infty)$, and in particular for $p$ fixed at the beginning of Section~\ref{sec:major-minor-arcs}.
Indeed, it is easy to obtain uniform $L^{p}$ estimates for the $j$-th term and a quickly decaying $L^{2}$ estimate follows from \eqref{eq:Phi-decay}.

We now prove \eqref{eq:270}.
By Theorem~\ref{thm:IW-mult} and the hypothesis \ref{MF:D}, we get
\begin{align}
\label{eq:281}
\norm[\big]{\FT^{-1}(\tilde{m}_{N}\tilde{\Xi}_{N}^S\hat{f})}_{\ell^{p_0}}
\lesssim
\log (S) \norm{f}_{\ell^{p_0}}.
\end{align}
The claim follows by interpolation with the $\ell^{2}$ estimate provided by Lemma~\ref{lem:tildemtildeXi} and Plancherel's theorem using \eqref{eq:choice:rho} and \eqref{eq:choice:alpha}.
\end{proof}

This completes the estimate for \eqref{eq:265}.

\subsection{Large scales: estimate for \eqref{eq:266}}
We have
\begin{align}
\eqref{eq:266}
&\lesssim
J^{p}_{2} \Big( \sum_{\kappa_S^{1/\tau} \le j<n}\FT^{-1}(\nu_j^S\hat{f}): n^{\tau} \geq \kappa_S \Big)^{1/2} \label{eq:5}\\
&+
\norm[\Big]{\sum_{j^{\tau} \geq \kappa_S} \abs[\big]{\FT^{-1}(((m_{(j+1)^\tau}-m_{j^\tau})\Xi_{j^{\tau}}^S-\nu_{j}^{S})\hat{f})} }_{\ell^p}. \label{eq:6}
\end{align}

\subsubsection{Error terms}
We estimate \eqref{eq:6} by the following result.

\begin{lemma}
\label{lem:12}
Suppose $S \leq j^{\tau u}$.
Then
\begin{align}
\label{eq:290}
\norm[\big]{\FT^{-1}\big(((m_{(j+1)^\tau}-m_{j^\tau})\Xi_{j^{\tau}}^S-\nu_j^S)\hat{f}\big)}_{\ell^p}
\lesssim
e^{(\card{\Gamma}+1) S^{\rho}} (j+1)^{-\tilde{\alpha}} \norm{f}_{\ell^p}
\end{align}
with some $\tilde{\alpha} > 1$.
\end{lemma}
\begin{proof}
Considering each fraction in $\dscrU_{S}$ individually and invoking \eqref{eq:UN-size} with $d=|\Gamma|$, we obtain for every $p_0\in(1, \infty)$ the estimate
\begin{align}
\label{eq:291}
\norm[\big]{\FT^{-1}\big(((m_{(j+1)^\tau}-m_{j^\tau})\Xi_{j^{\tau}}^S-\nu_j^S)\hat{f}\big)}_{\ell^{p_0}}
\lesssim
\abs{\scrU_{S}} \norm{f}_{\ell^{p_0}}
\lesssim
e^{(\card{\Gamma}+1) S^{\rho}} \norm{f}_{\ell^{p_0}}.
\end{align}
On the other hand, by \eqref{eq:mXi-nu} and Plancherel's theorem, we have
\begin{align}
\label{eq:292}
\norm[\big]{\FT^{-1}\big(((m_{(j+1)^\tau}-m_{j^\tau})\Xi_{j^{\tau}}^S-\nu_j^S)\hat{f}\big)}_{\ell^2}
\lesssim
(j+1)^{-\alpha} \norm{f}_{\ell^2}.
\end{align}
Interpolation between \eqref{eq:291} and \eqref{eq:292} finishes the proof of \eqref{eq:290}.
\end{proof}
It follows that
\[
\eqref{eq:6}
\lesssim
e^{(\card{\Gamma}+1) S^{\rho}} \sum_{j^{\tau} \geq \kappa_{S}} j^{-\tilde{\alpha}} \norm{f}_{\ell^{p}}
\lesssim
e^{(\card{\Gamma}+1) S^{\rho}} \kappa_{S}^{(1-\tilde{\alpha})/\tau} \norm{f}_{\ell^{p}},
\]
and this is summable in $S$ by the choice of $\kappa_{S}$ in \eqref{eq:kappa}.

\subsubsection{Jumps}
In order to estimate \eqref{eq:5}, we factorize
\[
\sum_{\kappa_{S}^{1/\tau}\le j<n}\nu_j^S(\xi) = \Big(\sum_{b\in\scrQ(Q_S)}\Psi_n(\xi-b)\Big) \cdot \Pi^{S}(\xi),
\]
where
\begin{align*}
Q_S
:=
\lcm P_{S}
\leq
3^{S},\qquad  \text{by \eqref{eq:IW-denom-lcm}}.
\end{align*}
\begin{equation}
\label{eq:Pi_s}
\Pi^{S}(\xi)
:=
\sum_{a/q\in\dscrU_{S}}G(a/q)\tilde\eta_{\kappa_{S}}(\xi-a/q).
\end{equation}
\[
\Psi_n(\xi)
:=
\sum_{\kappa_{S}^{1/\tau}\le j<n}\big(\Phi_{(j+1)^\tau}(\xi)-\Phi_{j^\tau}(\xi)\big)\eta_{j^{\tau}}(\xi).
\]
Notice that the latter function is supported on the cube centered at the origin with the side length $2^{-\kappa_{S}+\kappa_{S}^{\chi}}$ that is $\leq (4Q_{S})^{-1}$ provided that $C$ is chosen sufficiently large in \eqref{eq:kappa}.

Hence, it suffices to show that
\begin{align}
\label{eq:7}
J^{p}_{2} \Big(\FT^{-1}\Big(\sum_{b\in\scrQ(Q_S)}\Psi_n(\cdot-b) \hat{f}\Big): n^{\tau} \geq \kappa_S \Big)
\lesssim
\norm{f}_{\ell^{p}},
\end{align}
and, for some $\varepsilon>0$, that
\begin{align}
\label{eq:279}
\norm[\big]{\FT^{-1}(\Pi^{S}\hat{f})}_{\ell^p}
\lesssim
S^{-\varepsilon} \norm{f}_{\ell^p}.
\end{align}
To see \eqref{eq:7}, we invoke the sampling principle for the jumps  from \cite[Theorem 1.7]{arxiv:1808.04592}, and we use our assumption \eqref{eq:Phi-jump}, which provides the desired bounds for the trimmed multipliers $\Psi_n$.

To see \eqref{eq:279} notice that it holds for $p=2$ by \eqref{eq:G-decay} and Plancherel's theorem.
Therefore, it suffices to obtain a uniform (in $S$) estimate for $p=p_{0}$.
To this end, let $N=\kappa_{S}^{1/u}$ and split
\begin{equation}
\label{eq:PiS-split}
\Pi^{S} = \tilde\Pi^{S} \tilde{m}_{\kappa_{S}} + (\Pi^{S} - \tilde\Pi^{S} \tilde{m}_{\kappa_{S}}),
\end{equation}
where
\[
\tilde\Pi^{S}(\xi)
:=
\sum_{a/q\in\dscrU_{S}}\tilde\eta_{\kappa_{S}}(\xi-a/q).
\]
The former term in \eqref{eq:PiS-split} defines a bounded multiplier on $\ell^{p_{0}}$ by hypothesis \ref{MF:D} and \cite[Corollary 2.1]{MR1888798}.
For the latter term we proceed in a similar way as in \eqref{eq:mXi-tildem} and obtain
\[
\norm{\Pi^{S} - \tilde\Pi^{S} \tilde{m}_{\kappa_{S}}}_{\infty}
\lesssim
\kappa_{S}^{-\alpha/\tau}.
\]
On the other hand, we also have an $\ell^{\tilde{p}}$ estimate similar to \eqref{eq:291} for every $\tilde{p}\in(1, \infty)$.
Interpolating these estimates we obtain that the family of multipliers $\Pi^{S} - \tilde\Pi^{S} \tilde{m}_{\kappa_{S}}$ is uniformly bounded on $\ell^{p_{0}}$.


\section{Applications to operators of Radon type on $\Z^d$}
\label{sec:disrad}

This section is intended to prove Theorem~\ref{thm:discrete-jump}, which will be a consequence of Theorem~\ref{thm:4}. We will study discrete averaging Radon transform $M_t$ and truncated singular Radon transform $H_t$ in $\Z^{\Gamma}$ with the set of multi-indices $\Gamma$ as in the previous section. Generally, we will follow the notation used in Section~\ref{sec:discrete} and assume that $d=\abs{\Gamma}$. 

Later on, the averaging Radon transform $M_t$ and the truncated singular Radon transform $H_t$ will be thought of as convolution operators having scale $2^t$. Namely, for any finitely supported function $f: \Z^{\Gamma} \to \C$, any $x\in\Z^{\Gamma}$, and any $t\ge0$ we have
\[
M_tf(x)=K_{2^t}^{M}*f(x)\qquad \text{ and } \qquad H_tf(x)=K_{2^t}^{H}*f(x)
\]
with the kernels
\begin{align}
\label{eq:88}
K_{2^t}^{M}(x)=\frac{1}{\abs{\Omega_{2^t}\cap\Z^k}}\sum_{y \in \Omega_{2^t}\cap\Z^k} \delta_{(y)^{\Gamma}}(x)
\quad \text{ and } \quad
K_{2^t}^{H}(x)=\sum_{y \in \Omega_{2^t}\cap\Z^k\setminus\{0\}}\delta_{(y)^{\Gamma}}(x) K(y),
\end{align}
where $(y)^{\Gamma}$ is the canonical polynomial and $K : \R^{k}\setminus\Set{0} \to \C$ is a  Calder\'on--Zygmund kernel satisfying conditions \eqref{eq:size-unif}, \eqref{eq:cancel} and \eqref{eq:K-modulus-cont}.

The condition \eqref{eq:K-modulus-cont} could be replaced by a Dini type condition like in \cite{arxiv:1808.09048}.
However, we will not pursue this direction.

For any Schwartz function $f$ in $\R^{\Gamma}$, for every $x\in\R^{\Gamma}$ and $t\in\R$, we define respectively the continuous averaging Radon transform and truncated singular Radon transform
by setting
\begin{align}
\label{eq:125}
\begin{split}
\calM_{t} f(x)&:= \frac{1}{\meas{\Omega_{2^t}}} \int_{\Omega_{2^t}}
f(x-\stPol{y}) \dif y,\\
\calH_{t} f(x)& := {\rm p.v.}\int_{\Omega_{2^t}} f(x-\stPol{y}) K(y) \dif y,
\end{split}
\end{align}
where $K$
is a Calder\'on--Zygmund kernel satisfying
\eqref{eq:size-unif}, \eqref{eq:cancel}, and \eqref{eq:K-modulus-cont}.

The operators $M_t$ and $H_t$ are discrete counterparts of the continuous Radon operators defined in \eqref{eq:125}.
Let us stress that the continuous Radon transforms play an important role in the proof of Theorem \ref{thm:discrete-jump}.
We now state Theorem \ref{thm:12} for operators \eqref{eq:125}, which was recently proved in \cite{arxiv:1808.09048}  and will be used to verify condition \eqref{eq:Phi-jump}.
\begin{theorem}
\label{thm:12}
Let $\calT_t$ be either $\calM_t$ or $\calH_t$.
Then for every $p\in(1, \infty)$ there is $0<C_{p}<\infty$ such that, for every $f\in L^p(\R^{\Gamma})$, we have
\begin{equation}
\label{eq:173}
\sup_{\lambda>0} \norm{\lambda N_{\lambda}(\calT_tf: t\in\R)^{1/2}}_{L^p(\R^{\Gamma})}
\le C_{p}
\norm{f}_{L^p(\R^{\Gamma})}.
\end{equation}
In particular, \eqref{eq:173} implies that, for every $p\in(1, \infty)$ and $r\in(2, \infty]$, there is $0<C_{p, r}<\infty$ such that
\begin{align}
\label{eq:236}
\norm{\sup_{t\in\R}\abs{\calT_tf}}_{L^p(\R^{\Gamma})}&\le C_{p, \infty}\norm{f}_{L^p(\R^{\Gamma})},\\
\label{eq:237}
\norm{V^{r}(\calT_tf: t\in\R)}_{L^p(\R^{\Gamma})}
&\le C_{p, r}
\norm{f}_{L^p(\R^{\Gamma})},
\end{align}
for every $f\in L^p(\R^{\Gamma})$.
\end{theorem}
Using the methods from \cite{arxiv:1808.09048}, if $\calT_t$ is either $\calM_t$ or $\calH_t$, we can prove that for every $p\in(1, \infty)$ there is  $0<C_{p}<\infty$ such that for every $f\in L^p(\R^{\Gamma})$ and for every increasing sequence $0<t_1<t_2<\cdots$ we have 
\begin{align}
\label{eq:8}
  \norm[\Big]{\big(\sum_{k\in\N}\abs{(\calT_{t_{k+1}}-\calT_{t_{k}}) f}^2\big)^{1/2}}_{L^p(\R^{\Gamma})}\le C_p\norm{f}_{L^p(\R^{\Gamma})}.
\end{align}
Estimate \eqref{eq:8} will be used to verify condition \eqref{eq:Phi-square}.
Inequality \eqref{eq:8} may be thought of as a weak form of $r$-variational inequality \eqref{eq:237} with $r=2$.
We know that, for many operators in harmonic analysis $r$-variational inequalities only make sense for $r>2$, since L\'epingle's inequality fails outside this range.
However, the square function estimate \eqref{eq:8} follows from the randomized bound
\begin{align*}
  \norm[\Big]{\sum_{k\in\Z}\varepsilon_k(\calT_{k+1}-\calT_k)f}_{L^p(\R^{\Gamma})}\lesssim_p\norm{f}_{L^p(\R^{\Gamma})},
\end{align*}
which holds unifomly for every sequence $(\varepsilon_{k})_{k\in\Z}$ bounded by $1$ by \cite[Theorem 2.28]{arxiv:1808.09048} and the short $2$-variation bound
\begin{align*}
  \norm[\Big]{\big(\sum_{k\in\Z}V^2(\calT_tf: t\in[k, k+1])^2\big)^{1/2}}_{L^p(\R^{\Gamma})}\lesssim_p\norm{f}_{L^p(\R^{\Gamma})},
\end{align*}
which holds by \cite[Theorem 2.39]{arxiv:1808.09048}.

Let $m_t^M$ and $m_t^H$ be the multipliers corresponding to the kernels $ K_{2^t}^{M}$ and $ K_{2^t}^{H}$, respectively.
More precisely, for every $t\ge0$, we have
\begin{align}
\label{eq:169}
m_t^{M}(\xi) :=
\frac{1}{\abs{\Omega_{2^t}\cap\Z^k}}\sum_{y \in \Omega_{2^t}\cap\Z^k} e(\xi\cdot{(y)^{\Gamma}})
\quad \text{ and } \quad
m_t^{H}(\xi) :=
\sum_{y \in \Omega_{2^t}\cap\Z^k\setminus\{0\}}e(\xi\cdot{(y)^{\Gamma}})K(y).
\end{align}
The continuous versions of the multipliers from \eqref{eq:169} are given by
\begin{equation*}
\Phi^{M}_t(\xi) :=
\frac{1}{\abs{\Omega_{2^t}}} \int_{\Omega_{2^t}} e(\xi\cdot{(y)^{\Gamma}}) \dif y
\quad \text{ and } \quad
\Phi^{H}_t(\xi) :={\rm p.v.}
\int_{\Omega_{2^t}} e(\xi\cdot{(y)^{\Gamma}}) K(y) \dif y.
\end{equation*}
We see that $\Phi^{M}_t$ corresponds to the Fourier transform of the averaging operator $\calM_t$ and $\Phi^{H}_t$ corresponds to the Fourier
transform of the truncated Radon transform $\calH_t$.
Therefore, \eqref{eq:173} and \eqref{eq:8} can be applied, and conditions \eqref{eq:Phi-jump} and \eqref{eq:Phi-square} are verified respectively with $\Phi_t=\Phi^{M}_t$ or $\Phi_t=\Phi^{H}_t$.
In order to prove \eqref{eq:Phi-decay}, we appeal to van der Corput's estimates from \cite[Proposition B.2]{arxiv:1808.09048}.
Indeed, the van der Corput estimate ensures that
\begin{align}
\label{eq:172}
\abs{\Phi_t^M(\xi)}\lesssim \abs{ 2^{tA} \xi }_{\infty}^{-1/d}\lesssim
(t\frakq_*(\xi))^{-1/d}
\quad \text{if}\quad t\frakq_*(\xi)\ge1.
\end{align}
By a simple calculation, we get
\begin{align}
\label{eq:166}
\abs{\Phi_t^M(\xi)-1}\lesssim \abs{ 2^{tA} \xi }_{\infty}^{1/d}\lesssim (t\frakq_*(\xi))^{1/d}
\quad \text{if}\quad t\frakq_*(\xi)\le1.
\end{align}
Taking into account \eqref{eq:K-modulus-cont} and again van der Corput's estimate, we obtain
\begin{equation}
\label{eq:245}
\abs{ \Phi_{t}^H(\xi)-\Phi_{s}^H(\xi) }
\lesssim
(t\frakq_*(\xi))^{-\sigma/d},
\quad \text{if}\quad t\frakq_*(\xi)\ge1,
\end{equation}
for all $\kappa t\leq s\leq t$ with the implicit constant depending on $\kappa\in(0, 1)$.
Additionally, due to the cancellation condition \eqref{eq:cancel} and \eqref{eq:size-unif},  we have
\begin{align}
\label{eq:246}
\abs{\Phi_{t}^H(\xi)-\Phi_{s}^H(\xi)}
\lesssim
(t\frakq_*(\xi))^{\sigma/d},\quad \text{if}\quad t\frakq_*(\xi)\le1.
\end{align}
Therefore, \eqref{eq:172} and \eqref{eq:245} guarantee that \eqref{eq:Phi-decay} holds respectively with $\Phi_t=\Phi^{M}_t$ or $\Phi_t=\Phi^{H}_t$.

The \emph{Gauss sum} corresponding to $q\in\N$ and $a\in A_{q}$ is defined by
\begin{equation}
\label{eq:171}
G(a/q) := q^{-k} \sum_{r\in\N_{q}^{k}} e((a/q)\cdot \stPol{r}).
\end{equation}
Using Theorem~\ref{thm:weyl-sums}, we can also deduce the decay claimed in condition \eqref{eq:G-decay} for the Gauss sums.

\begin{lemma}
\label{lem:Gaussian-sum-decay}
For every $k\in\N$ there exists $\delta=\delta(k)>0$ such that for every $q\in\N$ and $a\in A_{q}$ we have
\[
\abs{G(a/q)} \lesssim_{k} q^{-\delta}.
\]
\end{lemma}
\begin{proof}
Let $a_{\gamma}'/q_{\gamma} = a_{\gamma}/q$ be the reduced form of the
entries of $a/q$ and let $\epsilon=\abs{\Gamma}^{-1}$. We first
assume that there
exists $\gamma\in\Gamma$ with $\abs{\gamma}\geq 2$ such that
$q^{\epsilon} < q_{\gamma}$, then also
$q_{\gamma}\leq q^{\abs{\gamma}-\epsilon}$, and we obtain the
conclusion by applying Theorem~\ref{thm:weyl-sums} with $\Omega=[0,q)^{k}$ and $\phi\equiv1$.

Otherwise, we assume that $q_{\gamma}\le q^{\epsilon} $ for every $\abs{\gamma}\geq 2$, thus $Q := \lcm\Set{q_{\gamma} \given \abs{\gamma}\geq 2} < q$ and the non-linear part of the polynomial $(a/q)\cdot \stPol{y}$ is constant modulo $1$ on congruence classes modulo $(Q\Z)^{k}$.
On the intersection of each of these congruence classes with $\N_{q}^{k}$, the exponential sum vanishes unless $q_{\gamma} \divides Q$ for all $\gamma\in\Gamma$ with $\abs{\gamma}=1$.
But in the latter case,
\[
q
\le\lcm\Set{q_{\gamma}\given\gamma\in\Gamma}=
\lcm\Set{Q,q_{\gamma}\given\abs{\gamma}=1}
=
Q,
\]
which gives a contradiction.
\end{proof}
Our task now, in the next four paragraphs, is to verify conditions \ref{MF:A}, \ref{MF:B}, \ref{MF:C}, and \ref{MF:D} from Definition \ref{def:1}.

\begin{remark}
\label{rem:1}
Theorem \ref{thm:4} allows us to prove that if $T_t$ is either $\tilde{M}_t^P$ or $\tilde{H}_t^P$ from \eqref{eq:1} with $P(x)=(x)^{\Gamma}$, then \eqref{eq:95} holds for all $p\in(1, \infty)$, and consequently we recover \eqref{eq:118} and \eqref{eq:119}, which were studied in \cite{arxiv:1803.05406}. The strategy is the same as we present here for the proof of Theorem \ref{thm:discrete-jump}. To verify
conditions in  Definition \ref{def:1}, let us briefly indicate that \eqref{eq:G-decay} is given by \cite[Theorem 3]{arxiv:1803.05406}. Conditions \eqref{eq:Phi-jump}, \eqref{eq:Phi-square} and \eqref{eq:Phi-decay} have been already verified above. For condition \ref{MF:A} see below.   Condition \ref{MF:B} is given by \cite[Theorem 1]{arxiv:1803.05406}. Condition \ref{MF:C} is given by \cite[Propositions 4.1, 4.2]{arxiv:1803.05406}. Condition \ref{MF:D} is satisfied with $\tilde{m}_{N}$ being the multiplier corresponding to an averaging operator of (logarithmic) scale $N-N^{\chi/10}$, say.
\end{remark}

\subsection{Condition \ref{MF:A}}
The $1$-variation norm in \eqref{eq:109} for $K_{2^t}=K_{2^t}^M$ or $K_{2^t}=K_{2^t}^H$ is controlled by a constant times $2^{-k n^{\tau}}$ times the number of lattice points in $\Omega_{2^{(n+1)^{\tau}}}\setminus\Omega_{2^{n^{\tau}}}$.
The latter quantity will be controlled by Proposition~\ref{cor:lattice-point-near-boundary} that refines \cite[Proposition 9]{MR1719802}.
A partial result in this direction was obtained in \cite[Proposition 3.1]{arXiv:1512.07518}.

\begin{proposition}
\label{cor:lattice-point-near-boundary}
Let $\Omega\subset\R^{k}$ be a bounded and convex set and let
$1\leq s\leq \diam(\Omega)$. Then
\begin{equation}
\label{eq:lattice-point-near-boundary}
\#\Set{x \in \Z^{k} : \dist(x,\partial\Omega)<s}
\lesssim_{k}
s \diam(\Omega)^{k-1}.
\end{equation}
The implicit constant depends only on the dimension
$k$, but not on the convex set $\Omega$.
\end{proposition}
\begin{proof}
Let $\Omega(s)=\Set{x \in \Z^{k} : \dist(x,\partial\Omega)<s}$ and
observe that
\[
\#\Omega(s)= \sum_{x\in\Omega(s)}1\lesssim
\sum_{x\in\Omega(s)}\abs{B(x,1/2)}\le \meas{\Set{ x\in\R^k : \dist(x,\partial\Omega) < s+1/2 }}\lesssim s \diam(\Omega)^{k-1},
\]
since the balls $B(x,1/2)$ for $x\in\Omega(s)$, are disjoint, have measure $\gtrsim 1$, and are contained in the set $\Set{ x\in\R^d : \dist(x,\partial\Omega) < s+1/2 }$, which has measure $\lesssim s \diam(\Omega)^{k-1}$ by \cite[Lemma A.1]{arxiv:1808.09048}.
\end{proof}
It follows that
\[
2^{-k n^{\tau}} \card{\Z^k \cap \Omega_{2^{(n+1)^{\tau}}}\setminus\Omega_{2^{n^{\tau}}}}
\lesssim
n^{\tau-1},
\]
and the condition \ref{MF:A} follows using monotonicity of the family of sets $\Omega_{t}$.

\subsection{Condition \ref{MF:B}}
We fix $N\in\N$ and suppose that, for some  multi-index $\gamma_0\in\Gamma$ and some integers $a, q$ such that $0\le a\le q$ and $(a, q) = 1$, we have $|\xi_{\gamma_0} - {a}/{q}|\le q^{-2}$.
For the multipliers associated with the averaging Radon operators we apply  Weyl's
inequality \eqref{eq:56} with $\phi\equiv1$, $\Omega=\Omega_{2^t}$ for
$t\in[N, N+1]$, and obtain
        \begin{align*}
          \sup_{N\le t_1, t_2\le N+1}|m_{t_1}^M(\xi)-m_{t_2}^M(\xi)|
          \lesssim \max\bigg\{\frac{1}{q}, \frac{q}{2^{|\gamma_0|N}}\bigg\}^{\varepsilon}\log (2^N)\lesssim N^{-\varepsilon\beta+1}, 
        \end{align*}
provided that
        \begin{align*}
        	N^\beta \leq q \leq 2^{|\gamma_0|N} N^{-\beta}.
        \end{align*}
Choosing $\beta=\beta({\alpha})=\varepsilon^{-1}(\alpha+1)$, we obtain
        \begin{align*}
          \sup_{N\le t_1, t_2\le N+1}|m_{t_1}^M(\xi)-m_{t_2}^M(\xi)|\lesssim N^{-\alpha}
        \end{align*}
with the implicit constant independent of $N\in\mathbb N$.

For the multipliers associated with the truncated singular Radon operators, we apply  Weyl's inequality \eqref{eq:56} with $\phi=K$, $\Omega=\Omega_{2^{t_2}}\setminus\Omega_{2^{t_1}}$ for $N\le t_1<t_2\le N+1$ and obtain, taking into account
\eqref{eq:size-unif} and \eqref{eq:K-modulus-cont}, that
        \begin{align*}
          \sup_{N\le t_1, t_2\le N+1}|m_{t_1}^H(\xi)-m_{t_2}^H(\xi)|
&          \lesssim \max\bigg\{\frac{1}{q},
          \frac{q}{2^{|\gamma_0|N}}\bigg\}^{\varepsilon}\log (2^N)+
  2^{kN} \sup_{\abs{x-y}\leq 2^N\kappa^{-\varepsilon}} \abs{K(x)-K(y)}\\
          &\lesssim N^{-\varepsilon\beta+1}
            +\max\bigg\{\frac{1}{q},\frac{q}{2^{|\gamma_0|N}}\bigg\}^{\sigma\varepsilon}\lesssim  N^{-\sigma\varepsilon\beta+1},
        \end{align*}
provided that
        \begin{align*}
        	N^\beta \leq q \leq 2^{|\gamma_0|N} N^{-\beta}.
        \end{align*}
As before, choosing $\beta=\beta({\alpha})=(\sigma \varepsilon)^{-1}(\alpha+1)$, we obtain
        \begin{align*}
          \sup_{N\le t_1, t_2\le N+1}|m_{t_1}^H(\xi)-m_{t_2}^H(\xi)|\lesssim N^{-\alpha}
        \end{align*}
with the implicit constant independent of $N\in\mathbb N$.

\subsection{Condition \ref{MF:C} and \ref{MF:D}}
In order to prove inequality \eqref{eq:110} from condition \ref{MF:C} or inequality \eqref{eq:104} from condition \ref{MF:D}, we will appeal to the following proposition.

\begin{proposition}
\label{prop:major-arc-singular}
Let $\Omega \subseteq B(0,N) \subset \R^{k}$ be a convex set or a
Boolean combination of finitely many convex sets and let $\calK : \Omega \to
\C$ be a continuous function. Then, for every $q\in\N$, $a\in A_{q}$, and $\xi = a/q + \theta \in \R^\Gamma$, we have
\begin{multline*}
\abs[\Big]{ \sum_{y \in \Omega\cap\Z^{k}} e({\xi}\cdot{\stPol{y}}) \calK(y)
- G(a/q) \int_{\Omega}e(\theta\cdot\stPol{t}) \calK(t) \dif t}\\
\lesssim_{k}
\frac{q}{N} N^{k} \norm{\calK}_{L^{\infty}(\Omega)} +N^{k} \norm{\calK}_{L^{\infty}(\Omega)} \sum_{\gamma \in \Gamma} \big(q\abs{\theta_\gamma} N^{\abs{\gamma}-1}\big)^{\epsilon_{\gamma}}
+ N^{k} \sup_{x, y\in\Omega : \abs{x-y}\leq q} \abs{\calK(x)-\calK(y)},
\end{multline*}
for any sequence $(\epsilon_{\gamma}: \gamma\in\Gamma)\subseteq[0, 1]$.
The implicit constant is independent of $a, q, N, \theta$, and the kernel $\calK$.
\end{proposition}
\begin{proof}
We split the sum into congruence classes modulo $q$ as follows:
\[
\sum_{y \in \Omega\cap\Z^{k}} e({\xi}\cdot\stPol{y}) \calK(y)
=
q^{-k}\sum_{r \in \N_q^k}
e(\stPol{r}\cdot a/q)
\cdot \Big(q^k\sum_{\substack{y \in \Z^k \\ qy+r\in \Omega}}
e({\theta}\cdot\stPol{qy+r}) \calK(qy+r)\Big).
\]
In order to approximate the expression in the brackets on the right-hand side by an integral, we write
\begin{multline*}
\abs[\Big]{q^{k}\sum_{\substack{y \in \Z^k \\ qy+r\in\Omega}}
e({\theta}\cdot\stPol{qy+r}) \calK(qy+r)
-\int_{\Omega}e(\theta\cdot\stPol{t}) \calK(t) \dif t}\\
=
\abs[\Big]{q^{k}\sum_{y \in \Z^k} e({\theta}\cdot\stPol{qy+r}) \calK(qy+r)\ind{\Omega}(qy+r)
-\sum_{y \in \Z^k}\int_{qy+[0,q)^k}e(\theta\cdot\stPol{t}) \calK(t)\ind{\Omega}(t) \dif t}\\
\leq
\sum_{y \in \Z^k} \int_{[0,q)^k} \abs[\big]{e(\theta\cdot\stPol{qy+r})\calK(qy+r)\ind{\Omega}(qy+r)
- e(\theta\cdot\stPol{qy+t}) \calK(qy+t)\ind{\Omega}(qy+t)} \dif t.
\end{multline*}
Notice that
\[
\abs[\big]{
\theta\cdot\stPol{q y + r} - \theta\cdot\stPol{q y+t}
}
\lesssim
\sum_{\gamma \in \Gamma} \big(q\abs{\theta_\gamma} N^{\abs{\gamma}-1}\big)^{\epsilon_{\gamma}},
\]
and
\[
\abs[\big]{\calK(q y + r) - \calK(q y+t)}
\lesssim
\sup_{x, y\in\Omega : \abs{x-y}\leq q} \abs{\calK(x)-\calK(y)},
\]
and
\[
\sum_{y\in\Z^k}\abs{\ind{\Omega}(qy+r)-\ind{\Omega}(qy+t)}\le\abs{(q^{-1}(\Omega-{r}))
\triangle (q^{-1}(\Omega-{t}))}\lesssim (N/q)^{k-1},
\]
where the last inequality is a consequence of Proposition~\ref{cor:lattice-point-near-boundary}.
Hence, we obtain the estimate 
\begin{multline*}
\abs[\Big]{q^{k}\sum_{\substack{y \in \Z^k \\ qy+r\in\Omega}}
e({\theta}\cdot\stPol{qy+r}) \calK(qy+r)
-\int_{\Omega}e(\theta\cdot\stPol{t}) \calK(t) \dif t}\\
\lesssim q^{k} \norm{\calK}_{L^{\infty}(\Omega)} (N/q)^{k-1}
+ N^{k} \norm{\calK}_{L^{\infty}(\Omega)} \sum_{\gamma \in \Gamma} \big(q\abs{\theta_\gamma} N^{\abs{\gamma}-1}\big)^{\epsilon_{\gamma}}
+ N^{k}\sup_{x, y\in\Omega : \abs{x-y}\leq q} \abs{\calK(x)-\calK(y)}.
\end{multline*}
Averaging in $r$, we obtain the claim.
\end{proof}

In order to verify \eqref{eq:110} for the averaging Radon operators, we apply Proposition~\ref{prop:major-arc-singular} with $\mathcal
K=|\Omega|^{-1}\ind{\Omega}$, where  $\Omega=\Omega_{2^t}$ for  $t\in[N, N+1]$, and $\epsilon_{\gamma}=1/\abs{\gamma}$.
Since $\norm{\calK}_{L^{\infty}(\Omega)}\lesssim 2^{-kN}$ and  $\sup_{x, y\in\Omega : \abs{x-y}\leq q} \abs{\calK(x)-\calK(y)}=0$, we obtain
\begin{align}
  \label{eq:12}
  \begin{split}
\sup_{N\le t\le N+1}\abs[\big]{m_{t}^M(\xi)-G(a/q)\Phi_{t}^M(\xi-a/q)}
& \lesssim q2^{-N}+\frakq_*\big(2^{NA}(q2^{-N})^{I}(\xi-a/q)\big)\\
&\lesssim N^{-\alpha},
  \end{split}
\end{align}
provided that  $1 \leq q \lesssim N^{\beta}$,  and $\abs{\xi_{\gamma}-a_{\gamma}/q} \leq 2^{-N\abs{\gamma}+N^{\chi}}$ for all $\gamma\in\Gamma$.
Thus \eqref{eq:12} implies \eqref{eq:110}.

For the truncated singular Radon operators, we apply Proposition \ref{prop:major-arc-singular} with $\calK=K\ind{\Omega}$, where  $\Omega=\Omega_{2^{t_2}}\setminus\Omega_{2^{t_1}}$ for
$N\le t_1< t_2\le N+1$, and $\epsilon_{\gamma}=1/\abs{\gamma}$.
Since $\norm{\calK}_{L^{\infty}(\Omega)}\lesssim 2^{-kN}$ by \eqref{eq:size-unif} and  $\sup_{x, y\in\Omega : \abs{x-y}\leq q}
\abs{\calK(x)-\calK(y)}\lesssim (q/2^{-N})^{\sigma}$ by \eqref{eq:K-modulus-cont}, we obtain
\begin{align*}
\sup_{N\le t_1, t_2\le N+1}\abs[\big]{m_{t_1}^H(\xi)-m_{t_2}^H(\xi)&-G(a/q)\big(\Phi_{t_1}^H(\xi-a/q)-\Phi_{t_2}^H(\xi-a/q)\big)}\\
& \lesssim q2^{-N}+(q2^{-N})^{\sigma}+\frakq_*\big(2^{NA}(q2^{-N})^{I}(\xi-a/q)\big)\\
&\lesssim N^{-\alpha}
\end{align*}
provided that  $1 \leq q \lesssim N^{\beta}$,  and $\abs{\xi_{\gamma}-a_{\gamma}/q} \leq 2^{-N\abs{\gamma}+N^{\chi}}$ for all $\gamma\in\Gamma$.

In order to verify condition \ref{MF:D}, we proceed in a similar way as in \eqref{eq:12} with the multiplier $\tilde{m}_{N} = \abs{\tilde{\Omega}_{N}}^{-1} \widehat{\one_{\tilde{\Omega}_{N}}}$, $\tilde{\Omega}_{N}=(\Z \cap [0,2^{N^{\chi} - 2N^{\chi/2}}])^{\Gamma}$.


\appendix

\section{Multidimensional exponential sums}
\label{sec:weyl}
In this section we will present a refinement of \cite[Proposition 3]{MR1719802} and \cite[Theorem 3.1]{arXiv:1512.07518}.
\begin{theorem}
\label{thm:weyl-sums}
For every $k,d \in\N$, there exists $\epsilon>0$ such that, for every polynomial
\[
P(n) = \sum_{\gamma\in\N_0^{k} : 0<\abs{\gamma}\leq d} \xi_{\gamma} n^{\gamma},
\]
every $N>1$, convex set $\Omega\subseteq B(0,N) \subset \R^{k}$, function $\phi : \Omega\cap\Z^{k}\to\C$, multi-index $\gamma_{0}\in\Gamma$, and integers $0\leq a<q$ with $(a, q) = 1$ and
\begin{align}
\label{eq:190}
\abs[\Big]{\xi_{\gamma_0} - \frac{a}{q}}
\leq
\frac{1}{q^2},
\end{align}
we have
\begin{equation}
\label{eq:56}
\abs[\Big]{\sum_{n \in \Omega \cap \Z^k} e(P(n))\phi(n)}
\lesssim_{d,k}
N^k \kappa^{-\epsilon} \log (N+1) \norm{\phi}_{L^\infty(\Omega)} + N^{k} \sup_{\abs{x-y}\leq N\kappa^{-\epsilon}} \abs{\phi(x)-\phi(y)},
\end{equation}
where
\[
\kappa = \min \{q,N^{\abs{\gamma_{0}}}/q\}.
\]
The implicit constant in \eqref{eq:56} is independent on the coefficients of  $P$ and the numbers $a$, $q$, and $N$.
\end{theorem}

Throughout the proof of Theorem~\ref{thm:weyl-sums}, we may assume $1 < q < N^{\abs{\gamma_{0}}}$.
Otherwise, the estimate follows from the triangle inequality.
Moreover, we may assume that $\kappa$ is sufficiently large depending on $k,d$.
Implicit constants in this section may depend on the dimension $k$ and the degree $d$ of the polynomial $P$, but will be independent of its coefficients, integers $a, q$, and $N$.

\subsection{Reduction to the constant coefficient case}
Suppose that \eqref{eq:56} is known in the case $\phi\equiv\const$.
We pass to the general case by partitioning $\Omega$ into
$J\le C\kappa^{k \epsilon/(k+1)}$ sets $\Omega_j\subseteq\Omega$ of diameter at most $N\kappa^{-\epsilon/(k+1)}$, for some $C>0$.
Then we fix $m_j\in\Omega_j$ for every $1\le j\le J$ and observe
that
\begin{align*}
\abs[\Big]{\sum_{n \in \Omega \cap \Z^k} e(P(n))\phi(n)}
&\le \sum_{j=1}^J\abs[\Big]{\sum_{n \in \Omega_j \cap \Z^k} e(P(n))\phi(m_j)}
\\ & \quad +
\sum_{j=1}^J\abs[\Big]{\sum_{n \in \Omega_j \cap \Z^k} e(P(n))(\phi(n)-\phi(m_j))}.
\end{align*}
This yields \eqref{eq:56} for general $\phi$ with $\epsilon$ replaced by $\epsilon/(k+1)$.
Thus, from now on, we may assume that $\phi\equiv 1$.

It seems more efficient (in terms of the dependence of $\epsilon$ on $k,d$) to make the above reduction in the one-dimensional case, but this would require keeping the dependence of the exponential sums on $\phi$ throughout the remaining argument. Since the dependence given by that argument is most probably far from sharp, we will not keep track of it.

\subsection{One-dimensional, leading coefficient case}
\label{sec:weyl:1d}
In this section we prove the case $k=1$, $\gamma_{0}=(d)$ of Theorem~\ref{thm:weyl-sums}.
In this case $\Omega$ is an interval.
Since the polynomial $n\mapsto P(n+m)$ has the same leading coefficient as $P$, for simplicity of notation we may assume $\Omega = [1,N']$, for some $N'\leq N$.
The interval $\Omega$ either has length $\geq N/2$ or can be written as the symmetric difference of two intervals of length between $N/2$ and $N$.
Thus we may assume $N/2 \leq N' \leq N$, and it suffices to consider $N'=N$, since the form of the claimed estimate does not change when $N$ is multiplied by a bounded factor.

Consider first the case $d=1$. Then $P(n)=\xi_{(1)} n=\xi n$.
Assume that $\xi\not=0$ and $q\ge2$, we obtain
\[
\abs[\Big]{\sum_{n = 1}^{N} e(P(n))}
\lesssim\frac{1}{\norm{\xi}}
\lesssim
q
\leq
N/\kappa.
\]

In the case $d\ge2$, we will invoke Weyl estimate with logarithmic loss due to Wooley (see \cite[Remark after Theorem 1.5]{MR2912712})
\begin{align}
\label{eq:189}
\abs[\Big]{\sum_{n = 1}^{N} e(P(n))}
\lesssim
N \log N \bigg( \frac{1}{q} + \frac{1}{N} + \frac{q}{N^{d}} \bigg)^{\frac{1}{2d^2-2d+1}}.
\end{align}
Then by \eqref{eq:189} we get the desired claim.

\subsection{Multidimensional case}
\label{sec:weyl:l0..0}
We prove Theorem~\ref{thm:weyl-sums} for a fixed $k$ and $d$ by
downward induction on $\abs{\gamma_{0}}$.
This will require a certain
change of variables in our exponential sum, which will allow us to
reduce the matter to the situation when \eqref{eq:190} holds for the
multi-index $\gamma_{0}$, which is of the form $(l,0,\dotsc,0)$ for some $l\in\N$.
Then, we deduce the desired bound from the one-dimensional, leading coefficient case.

\subsubsection{Change of variables}
In this section we reduce the case of general $\abs{\gamma_{0}}=l$ to the case $\gamma_{0}=(l,0,\dotsc,0)$.
To this end, we find a linear map $L$ on $\R^{k}$ that restricts to an automorphism of $\Z^{k}$ and for which we can control the coefficient of $x_{1}^{l}$ in the polynomial $P(L(x))$.
This will be provided by the following result.

\begin{proposition}
\label{prop:10}
Under the assumptions of Theorem~\ref{thm:weyl-sums}, there is an automorphism $L$ of $\Z^k$ such that
\begin{align}
\label{eq:187}
\sum_{n \in \Omega \cap \Z^k} e(P(n))
=
\sum_{n \in \tilde{\Omega} \cap \Z^k} e(P(L(n))),
\end{align}
where $\tilde{\Omega} = L^{-1}[\Omega]$ and
$\tilde{\Omega} \subseteq B(0,CN)$ for some $C > 0$ depending only on
$d,k$.

Moreover, let $l=\abs{\gamma_0}$, $\nu= \binom{l+k-1}{k-1}$ and $\sigma$
be the coefficient of $x_{1}^{l}$ in the polynomial $P(L(x))$. Then there exists a constant
$C_{d, k}\ge1$ such that if $\kappa\ge C_{d, k}$ and $\delta\in(0, (4\nu)^{-1})$ then we can find integers $0\le a'<q'$ with $(a', q')=1$ and
$\kappa^{\delta}\le q' \leq q \kappa^{1/2}$ satisfying
\begin{align}
\label{eq:188}
\abs[\Big]{\sigma - \frac{a'}{q'}}
\leq
\frac{1}{q' q \kappa^{1/2}}.
\end{align}
\end{proposition}
The construction of the automorphism from Proposition~\ref{prop:10} is very simple and is provided by the following lemma.

\begin{lemma}[{cf.~\cite[Lemma 1, p.~1305]{MR1719802}}]
\label{lem:16}
Let $k, l\in\N$ and let $\nu = \binom{l+k-1}{k-1}$ denote the dimension of the vector space of all polynomials with $k$ variables in $\R^k$ which are homogeneous of degree $l$.
Then there exist linear transformations $L_1, \ldots, L_\nu$ of $\R^k$ with integer coefficients and determinant $1$ so that each $L_j$ restricts to an automorphism of $\Z^k$. Moreover, for each $\gamma_0$ with $\abs{\gamma_0} = l$ there exist integers $c_0, \ldots, c_\nu$, with $c_0 > 0$, such that for every homogeneous polynomial $P$ of degree $l$,
 if $\theta$ is the coefficient of $x^{\gamma_0}$ of $P(x)$, and $\sigma_j$ is the coefficient of $x_1^l$ of $P(L_j (x))$, then
\begin{align}
\label{eq:186}
c_0 \theta = c_1 \sigma_1 + \ldots + c_\nu\sigma_\nu.
\end{align}
\end{lemma}
\begin{proof}
For every multi-index $\alpha$ such that $\abs{\alpha}=l$, we define the function $\mu(\alpha)=\sum_{i=1}^k\mu_i\alpha_i$, where
$1=\mu_1 <\mu_2<\ldots< \mu_k$ are suitably chosen rapidly growing integers, which ensure that $\mu(\alpha)\not=\mu(\alpha')$ for any $\alpha\not=\alpha'$.

For each $j\in \N_{\nu}$, we consider the $\Z^{k}$-automorphism
\[
L_{j}(x)=\big(x_{1},x_{2}+j^{\mu_2}x_{1},\dotsc,x_{k}+j^{\mu_k}x_{1}\big).
\]
If $\sigma_j$ is the coefficient of $x_{1}^{l}$ in the polynomial $P(L_{j}(x))$, then
\begin{align}
\label{eq:185}
\sigma_j=\sum_{\abs{\alpha}=l} \theta_{\alpha} j^{\mu(\alpha)},\quad \text{ for }\quad j\in \N_{\nu}.
\end{align}
The identities from \eqref{eq:185} define the generalized Vandermonde matrix $V$ of size $\nu\times\nu$ such that $V\theta=\sigma$, where $\theta=(\theta_{\alpha}: \abs{\alpha}=l)$ and $\sigma=(\sigma_1, \ldots,\sigma_{\nu})$.
Since $\mu(\alpha)\not=\mu(\alpha')$ for any $\alpha\not=\alpha'$, it is not difficult to show that $\det V\not=0$, hence $V$ is invertible, and \eqref{eq:186} holds.
\end{proof}

\begin{proof}[Proof of Proposition~\ref{prop:10}]
Let $L_{1},\dots,L_{\nu}$ be the linear maps given by Lemma~\ref{lem:16}.
Denote the coefficient of $x_{1}^{l}$ in the polynomial $P(L_{j}(x))$ by $\sigma_{j}$.
In fact, by Lemma~\ref{lem:16},
\begin{equation}
\label{eq:160}
\xi_{\gamma_{0}} = c_0^{-1} (c_1 \sigma_1 + \ldots + c_\nu \sigma_\nu),
\end{equation}
with some integers $c_{0},\dotsc,c_{\nu}$ depending only on $d,k$, with $c_0>0$.
Fix $\delta>0$ such that $4\delta\nu<1$.
By Dirichlet's principle, for all $j\in\N_{\nu}$, there exist reduced fractions $a_j/q_j$ such that
\begin{equation}
\label{eq:183}
\abs[\Big]{\sigma_j - \frac{a_j}{q_j}}
\leq
\frac{1}{q_j q \kappa^{1/2}}
\leq
\frac{1}{q_j^{2}},
\quad \text{ and } \quad
1\le q_j \leq q \kappa^{1/2},
\quad \text{ and } \quad
(a_j, q_j) = 1.
\end{equation}
Suppose first that
\begin{equation}
\label{eq:qj-all-major}
1\le q_j \leq \kappa^{\delta}
\quad\text{for every}\quad
j\in\N_{\nu}.
\end{equation}
This will lead to a contradiction provided that $\kappa\ge C_{d, k}:=(2(\abs{c_0}+\ldots+\abs{c_{\nu}}))^{4}$.
We write
$
c_0^{-1}(c_1 a_1/q_1 + \ldots +c_\nu a_\nu /q_{\nu})
$
in the reduced form $\tilde a/\tilde q$ with $(\tilde a, \tilde q) = 1$, then $1\le \tilde q \leq c_{0} \kappa^{\delta\nu}$.
Suppose that $\tilde a/\tilde q = a/q$. Thus
\[
\kappa \leq q =\tilde q \leq c_{0} \kappa^{\delta\nu} \iff \kappa^{1-\delta\nu}\le c_0,
\]
and this is a contradiction, since $\kappa\ge C_{d, k}$.
In the remaining case $\tilde a/\tilde q \neq a/q$, we get
\[
\frac{1}{q \tilde q}
\leq
\abs[\Big]{ \frac{\tilde a}{\tilde q} - \frac{a}{q} }
\leq
\abs[\Big]{ \xi_{\gamma_{0}} - \frac{a}{q} }
+
\abs[\Big]{ \xi_{\gamma_{0}} - \frac{\tilde a}{\tilde q} }
\leq
\frac{1}{q^2} + c_{0}^{-1} \sum_{j=1}^{\nu} \abs{c_{j}} \frac{1}{q_{j} q \kappa^{1/2}}
\]
using \eqref{eq:160} and \eqref{eq:183}.
Thus, for $C=c_{0}^{-1} \sum_{j=1}^{\nu} \abs{c_{j}}$, we have
\[
c_{0}^{-1} \kappa^{-\nu\delta}
\leq
\frac{1}{\tilde q}
\leq
\frac{1}{q} + \frac{C}{\kappa^{1/2}}
\leq
\kappa^{-1} + C \kappa^{-1/2}\iff \kappa^{1/2-\nu\delta}\le c_0(1+C),
\]
and this is again a contradiction, since $\kappa\ge C_{d, k}$.

Thus we have shown that \eqref{eq:qj-all-major} is impossible, so $q_j \geq \kappa^{\delta}$ for at least one $j \in \N_{\nu}$.
In that case, we see
\begin{equation*}
\sum_{n \in \Omega \cap \Z^k} e(P(n))
=
\sum_{n \in \tilde{\Omega} \cap \Z^k} e(P(L_{j}(n)))
\end{equation*}
with $\tilde{\Omega} = L_{j}^{-1}[\Omega]$.
We also obtain that $\tilde{\Omega} \subseteq B(0,CN)$ for some $C > 0$ depending only on $d,k$, since $\Omega \subseteq B(0,N)$. This proves the desired claim in \eqref{eq:187} and \eqref{eq:188} for $L=L_j$, $\sigma =\sigma_j$, $a'=a_j$ and $q'=q_j$.
\end{proof}

\subsubsection{The case $\abs{\gamma_{0}} = d$}
To verify the base case of
our backward induction for $\abs{\gamma_{0}} = d$, we may assume that
$\kappa\ge C_{d, k}$, with $C_{d, k}$ which was defined in Proposition~\ref{prop:10}.
Otherwise, the desired bound follows from the triangle
inequality, since we allow the implicit constant in \eqref{eq:56} to
depend on $d$ and $k$.
Proposition~\ref{prop:10} with $l=d$ and $\delta<1/(4\nu)$ provides
an automorphism $L$ of $\Z^k$ such that \eqref{eq:187} holds. If
$\sigma$ is the coefficient of $x_{1}^{d}$ in the polynomial
$P(L(x))$, then also \eqref{eq:188} holds for some integers $a'$ and
$q'$. Next, for each $n\in\Z^k$ we write $n=(n_1, n')$, where $n_1\in\Z$ and $n'\in\Z^{k-1}$, and by \eqref{eq:187} we obtain
\begin{align}
\label{eq:191}
\abs[\Big]{\sum_{n \in \Omega \cap \Z^k} e(P(n))}= \abs[\Big]{
\sum_{n \in \tilde{\Omega} \cap \Z^k} e(P(L(n)))}\le\sum_{\substack{n'\in\Z^{k-1}\\
\abs{n'}\le CN}}\abs[\Big]{\sum_{\substack{n_1\in\Z\\ (n_1, n')\in \tilde{\Omega} \cap \Z^k}} e(P(L(n_1, n'))) }.
\end{align}
We reduced the matter to the case
$\gamma_{0} = (d,0,\dotsc,0)$. We now deduce the desired bound from
the one-dimensional, leading coefficient case. Indeed, the inner sum on the right-hand side of \eqref{eq:191} can be
dominated by a constant multiple of $N\tilde\kappa^{-\epsilon}\log(N+1)$, where
$\tilde\kappa=\min\{q', N^d/q'\}$. Since
$\kappa^{\delta}\le q'\le q\kappa^{1/2}$, then
$\tilde\kappa\ge\kappa^{-1/2}\min\{\kappa^{\delta+1/2}, N^d/q\}\ge
\kappa^{\delta}$. Therefore, we conclude
\begin{align*}
\abs[\Big]{\sum_{n \in \Omega \cap \Z^k} e(P(n))}\lesssim N^k\kappa^{-\epsilon\delta}\log(N+1)
\end{align*}
and we are done.

\subsubsection{The case $|\gamma_0|=l$ for  $\gamma_{0} = (l,0,\dotsc,0)$}
For the inductive step, let $1\leq l\leq d$ and assume that the result
is known for all $\gamma\in\N_0^k$ with $l<\abs{\gamma}\leq d$. We will
consider first the special case when $\gamma_{0} = (l,0,\dotsc,0)$.
Before we do this, we need to recall simple but significant
fact about Diophantine approximations.

\begin{lemma}[cf.~{\cite[Lemma 1, p.~1298]{MR1719802}}]
\label{lem:11}
Let $\theta\in\R$ and $Q\in\N$.
Suppose that
\[
\abs[\Big]{\theta - \frac{a}{q}}
\leq
\frac{1}{q^2}
\]
with $(a,q)=1$ and $0\le a<q \leq M$ for some $M>0$.
Then there is a reduced fraction $a'/q'$ so that $(a', q') = 1$ and
\[
\abs[\Big]{Q \theta - \frac{a'}{q'}}
\leq
\frac{1}{2q'M}
\]
with $q/(2Q) \leq q' \leq 2 M$.
\end{lemma}
\begin{proof}[Proof of Lemma~\ref{lem:11}]
By Dirichlet's principle, there exists a reduced fraction $a'/q'$ such that $0\le a'<q' \leq 2 M$, $(a', q') = 1$, and
\begin{equation}
\label{eq:182}
\abs[\Big]{Q \theta - \frac{a'}{q'}}
\leq
\frac{1}{2q'M}.
\end{equation}
We have to show that $q/(2Q) \leq q'$.
There are two cases.
If $a'/q' = Q a/q$, then $q' \geq q/Q$, and we are done.
Otherwise $a'/q' \neq Q a /q$, and, since the difference of these numbers can be written as a (not necessarily reduced) fraction with denominator $qq'$, we have
\[
\frac{1}{q q'}
\leq
\abs[\Big]{ \frac{a'}{q'} - \frac{Qa}{q} }
\leq
\abs[\Big]{ Q \theta - \frac{a'}{q'} }
+
Q \abs[\Big]{ \theta - \frac{a}{q} }
\leq
\frac{1}{2q' M}
+
\frac{Q}{q^{2}},
\]
so
\[
\frac{1}{q'}
\leq
\frac{q}{2 q' M}
+
\frac{Q}{q}
\leq
\frac{1}{2 q'}
+
\frac{Q}{q},
\]
and the conclusion follows.
\end{proof}

We now return to the proof of the special case $\gamma_{0} = (l,0,\dotsc,0)$. Let $\chi > 0$ be chosen later.
By Dirichlet's principle, for every index $\gamma\in\N_0^k$ with $l<\abs{\gamma} \le d$, there is a reduced fraction $a_\gamma/q_\gamma$ such that
\[
\abs[\Big]{{\xi}_\gamma - \frac{a_\gamma}{q_\gamma}}
\leq
\frac{\kappa^{\chi}}{q_\gamma N^{\abs{\gamma}}}
\]
with $(a_\gamma, q_\gamma) = 1$ and $1\le q_\gamma \leq N^{\abs{\gamma}}/\kappa^{\chi}$.
If $q_\gamma\ge \kappa^{\chi} $ for some $\gamma\in\N_0^k$ such that $l<\abs{\gamma}\le d$, then by the induction hypothesis we are done.

Suppose now that $1\le q_\gamma \leq \kappa^{\chi}$ for all $\gamma\in\N_0^k$ such that $l<\abs{\gamma} \le d$.
Then
\[
Q := \lcm \Set{q_\gamma \given l<\abs{\gamma} \le d} \leq \kappa^{u \chi},
\]
where $u=\abs{\Set{q_\gamma \given l<\abs{\gamma} \le d}}$. Partitioning
$\Omega$ into the sets of the form
$\Omega \cap (Q\Z + r) \times \Set{n'}$, where $n'\in\Z^{k-1}$ and
$r\in\Set{0,\dotsc,Q-1}$, we see that it suffices to show
\begin{equation}
\label{eq:163}
\abs[\Big]{\sum_{m\in\Z : (Qm+r,n')\in\Omega} e(P(Qm + r,n'))}
\lesssim
(N/Q) \kappa^{-\epsilon} \log (N+1).
\end{equation}
By the convexity assumption on $\Omega$, the sum in \eqref{eq:163} runs over a disjoint union of finitely many intervals.
We consider each interval separately, so we sum over $m\in\Z$ such that $-N/Q \lesssim M_{0} \leq m \leq M_{1} \lesssim N/Q$.
After this restriction, we write the left-hand side of \eqref{eq:163} in the form
\[
\sum_{m=M_{0}}^{M_{1}} e(P(Qm + r,n'))
=
\sum_{m = M_0}^{M_1}
A_{m} B_{m},
\]
where
\begin{align*}
A_{m} &= e\Big(\sum_{l<\abs{\gamma} \le d}\xi_\gamma (Qm + r, n')^\gamma\Big),\\
B_{m} &= e(Q^{l} \xi_{\gamma_{0}} m^{l} + R(m)).
\end{align*}
and $R$ is a polynomial in $m$ of degree $\le l-1$ depending on $r$ and $n'$.
Summation by parts gives
\[
\sum_{m = M_0}^{M_1} A_{m} B_{m}
=
\sum_{m = M_0}^{M_1-1} (A_{m} - A_{m+1}) S_{m_1}
+ A_{M_{1}} S_{M_{1}}
\]
with $S_{m} = \sum_{n = M_{0}}^{m} B_{n}$.
Using the rational approximation of $\xi_{\gamma}$, we estimate the difference on the right-hand side by
\begin{align*}
\abs{A_{m} - A_{m+1}}
&\lesssim
\norm[\Big]{\sum_{l<\abs{\gamma} \le d}\xi_\gamma \big((Qm + r, n')^\gamma-(Q(m+1) + r, n')^\gamma\big) \mod 1}_{\R/\Z}\\
&=
\norm[\Big]{\sum_{l<\abs{\gamma} \le d}(\xi_\gamma - a_\gamma/q_\gamma) \big((Qm + r, n')^\gamma-(Q(m+1) + r, n')^\gamma\big) \mod 1}_{\R/\Z}\\
&\lesssim
\sum_{l<\abs{\gamma} \le d} \abs{\xi_\gamma - a_\gamma/q_\gamma} QN^{\abs{\gamma}-1}\\
&\lesssim
N^{-1} Q \kappa^{\chi}.
\end{align*}
Since $\abs{A_{M_{1}}}=1$, it remains to estimate $S_{m}$. This is a
one-dimensional exponential sum, and we can estimate it using the case
$k=1$, $d=\abs{\gamma_{0}}$ of Theorem~\ref{thm:weyl-sums}. To this
end, we need the information on rational approximation of the leading
coefficient $Q^{l}\xi_{\gamma_{0}}$. Applying Lemma~\ref{lem:11} to
$\xi_{\gamma_{0}}$ with $Q^{l}$ in place of $Q$ and $M=q$, we obtain a
reduced fraction $a'/q'$ such that
\[
\abs[\Big]{Q^{l} \xi_{\gamma_{0}} - \frac{a'}{q'}}
\leq
\frac{1}{2 q' q}
\leq
\frac{1}{(q')^2}
\]
with $(a', q') = 1$, and $\kappa^{1-lu\chi}/2 \leq q' \leq 2 q \leq 2 N^{l}/\kappa$.
By the one-dimensional result, it follows that
\[
\abs{S_{m}}
\lesssim
(N/Q) \tilde\kappa^{-\epsilon'} \log (N+1),
\]
for some $\epsilon'>0$, where $\tilde\kappa \ge \kappa^{1-lu\chi}/2$.
Hence,
\begin{align*}
\abs[\Big]{\sum_{m_1 = M_0}^{M_1} A_{m_1} B_{m_1}}
&\lesssim
\big((N/Q) N^{-1}Q \kappa^{\chi}+1\big)(N/Q) \tilde\kappa^{-\epsilon'} \log (N+1)\\
&\lesssim
(N/Q) \kappa^{-\epsilon'+\chi(1+lu\epsilon')} \log (N+1),
\end{align*}
which suffices provided that $\chi>0$ is small enough.

\subsubsection{The case $\abs{\gamma_{0}} = l$}
We now deduce the general case $\abs{\gamma_{0}} = l$ from the previous case for $\gamma_{0}=(l,0,\dotsc,0)$.
As in the case $\abs{\gamma_{0}} = d$, we may assume that
$\kappa\ge C_{d, k}$, with $C_{d, k}$ which was defined in Proposition~\ref{prop:10}.
Applying Proposition~\ref{prop:10}, with $\delta<1/(4\nu)$, we obtain
an automorphism $L$ of $\Z^k$ such that
\begin{align*}
\abs[\Big]{\sum_{n \in \Omega \cap \Z^k} e(P(n))}= \abs[\Big]{
\sum_{n \in \tilde{\Omega} \cap \Z^k} e(P(L(n)))}.
\end{align*}
If $\sigma$ is the coefficient of $x_{1}^{l}$ in the polynomial
$P(L(x))$, then also \eqref{eq:188} holds for some integers $a'$ and
$q'$.
Invoking now the case
$\gamma_{0} = (l,0,\dotsc,0)$ with $\kappa$ replaced by $\kappa^{\delta}$, we obtain the claim and this completes the induction.


\printbibliography

\end{document}